\documentclass[10pt]{scrartcl}

\usepackage{amsmath, amssymb, amsthm, marginnote, ltxcmds, adjustbox, stmaryrd, graphicx, bbold, stackengine}

\usepackage{mathtools, stackengine, changepage, ragged2e}

\usepackage{todonotes}

\definecolor{cite}{HTML}{11871E}
\definecolor{url}{HTML}{698996}
\definecolor{link}{HTML}{912F1B}

\usepackage[pdfencoding=unicode, colorlinks=true, linkcolor=link, citecolor=cite, urlcolor=url, linktocpage]{hyperref}

\usepackage[backend=biber, style=alphabetic, maxnames=10, maxalphanames=3, minalphanames=3]{biblatex}
\usepackage{tikz}
\usetikzlibrary{cd}
\usetikzlibrary{arrows, arrows.meta, positioning, calc}
\tikzcdset{arrow style=tikz, diagrams={>={Straight Barb[scale=0.8]}}}

\tikzstyle{arrow} = [-{Straight Barb[scale=0.8]}, line width=0.2mm]

\usepackage{enumitem}
{\end{enumerate}%
}

\usepackage[
	letterpaper,
	twoside=false,
	textheight=22cm,
	textwidth=14.4cm,
	marginparsep=0.75cm,
	marginparwidth=2.5cm,
	heightrounded,
	centering
]{geometry}

\usepackage{stmaryrd}

\usepackage[capitalise]{cleveref}

\Crefname{prop}{Proposition}{Propositions}
\Crefname{lem}{Lemma}{Lemmas}
\Crefname{cor}{Corollary}{Corollaries}
\Crefname{thm}{Theorem}{Theorems}
\Crefname{alphThm}{Theorem}{Theorems}
\Crefname{alphCor}{Corollary}{Corollaries}

\Crefname{defn}{Definition}{Definitions}
\Crefname{notation}{Notation}{Notations}
\Crefname{cons}{Construction}{Constructions}
\Crefname{rmk}{Remark}{Remarks}
\Crefname{obs}{Observation}{Observations}
\Crefname{trick}{Trick}{Tricks}
\Crefname{warning}{Warning}{Warnings}
\Crefname{conj}{Conjecture}{Conjectures}
\Crefname{assump}{Assumption}{Assumptions}
\Crefname{recollect}{Recollection}{Recollections}
\Crefname{terminology}{Terminology}{Terminologies}

\Crefname{question}{Question}{Questions}
\Crefname{example}{Example}{Examples}

\Crefname{figure}{Figure}{Figures}

\crefformat{equation}{(#2#1#3)}
\crefformat{section}{\S#2#1#3}
\crefmultiformat{section}{\S\S#2#1#3}{and~#2#1#3}{, #2#1#3}{, and~#2#1#3}

\newtheorem{thm}[subsubsection]{Theorem}
\newtheorem{prop}[subsubsection]{Proposition}
\newtheorem{lem}[subsubsection]{Lemma}
\newtheorem*{lem*}{Lemma}
\newtheorem{cor}[subsubsection]{Corollary}

\newtheorem{alphThm}{Theorem}

\newcommand{\neutralize}[1]{\expandafter\let\csname c@#1\endcsname\count@}
\makeatother

\newtheorem{alphCor}{Corollary}



\theoremstyle{definition}
\newtheorem{defn}[subsubsection]{Definition}
\newtheorem{cons}[subsubsection]{Construction}
\newtheorem{nota}[subsubsection]{Notation}

\newtheorem{recollect}[subsubsection]{Recollections}
\newtheorem{terminology}[subsubsection]{Terminology}

\theoremstyle{remark}
\newtheorem{rmk}[subsubsection]{Remark}
\newtheorem{obs}[subsubsection]{Observation}
\newtheorem{example}[subsubsection]{Example}

\newtheorem{warning}[subsubsection]{Warning}

\DeclareMathOperator{\fibfunc}{\mathrm{fun}}

\newcommand{\baseCat}{{\mathcal{T}}}

\newcommand{\tkappa}{^{\underline{\kappa }}}
\newcommand{\tomega}{^{\underline{\omega}}}

\newcommand{\norm}{\mathrm{N}}
\newcommand{\calg}{\mathrm{CAlg}}
\newcommand{\spc}{\mathcal{S}}
\newcommand{\vop}{^{{\myuline{\mathrm{op}}}}}

\newcommand{\ind}{\mathrm{Ind}}

\newcommand{\filteredKappa}{^{\kappa\operatorname{-filt}}}
\newcommand{\filtered}{^{\mathrm{filt}}}

\newcommand{\lexact}{^{\mathrm{lex}}}

\newcommand{\Tsemiadd}{^{\underline{\mathrm{sadd}}}}
\newcommand{\Tinert}{^{\underline{\mathrm{inert}}}}
\newcommand{\cocartesianCategory}{\mathrm{coCart}}
\newcommand{\linear}{\mathrm{Lin}}
\newcommand{\accessible}{{\mathrm{Acc}}}

\DeclareMathOperator{\cmonoid}{\mathrm{CMon}}
\newcommand{\canonical}{\mathrm{can}}

\DeclareMathOperator{\forget}{\mathrm{fgt}}

\DeclareMathOperator{\rfunc}{\mathrm{RFun}}

\DeclareMathOperator{\lfunc}{\mathrm{LFun}}

\newcommand{\eval}{\mathrm{ev}}
\newcommand{\idem}{\mathrm{Idem}}

\newcommand{\constant}{\operatorname{const}}
\DeclareMathOperator{\tconstant}{\underline{\mathrm{const}}}
\newcommand{\cat}{\mathrm{Cat}}

\newcommand{\presentable}{\mathrm{Pr}}

\newcommand{\Texact}{^{{\mathrm{ex}}}}
\newcommand{\Tlexact}{^{{\mathrm{lex}}}}
\newcommand{\Trexact}{^{{\mathrm{rex}}}}

\newcommand{\A}{\mathcal{A}}
\newcommand{\sC}{{\mathcal C}}
\newcommand{\D}{{\mathcal D}}

\newcommand{\B}{\mathcal{B}}
\newcommand{\op}{^{\mathrm{op}}}

\newcommand{\unstraighten}{\mathrm{UnStr}}

\newcommand{\cocartesian}{^{\mathrm{cocart}}}
\newcommand{\sG}{{\mathcal G}}

\newcommand{\K}{\mathcal{K}}
\newcommand{\R}{\mathcal{R}}
\newcommand{\E}{\mathcal{E}}

\DeclareMathOperator{\mackey}{\mathrm{Mack}}
\DeclareMathOperator{\presheaf}{\mathrm{PSh}}
\DeclareMathOperator{\nattrans}{\mathrm{Nat}}

\newcommand{\spectra}{\mathrm{Sp}}
\newcommand{\finite}{\mathrm{Fin}}

\DeclareMathOperator{\effBurn}{\mathrm{Span}}
\DeclareMathOperator{\im}{\mathrm{Im}}
\newcommand{\id}{\mathrm{id}}
\DeclareMathOperator{\map}{\mathrm{Map}}
\DeclareMathOperator{\fib}{\operatorname{fib}}

\DeclareMathOperator{\func}{\mathrm{Fun}}

\newcommand{\catinf}{\mathrm{Cat}}
\DeclareMathOperator{\twistedArrow}{\operatorname{TwAr}}
\newcommand{\orbit}{\mathcal{O}}

\DeclareMathOperator{\cofree}{\operatorname{\underline{Cofree}}}
\DeclareMathOperator{\totalCategory}{\operatorname{Total}}
\newcommand{\tcone}{^{\underline{\triangleleft}}}
\newcommand{\tcocone}{^{\underline{\triangleright}}}
\newcommand{\terminalTCat}{\underline{\ast}}

\newcommand{\pointwise}{^{\mathrm{pw}}}

\newcommand{\totimes}{^{\underline{\otimes}}}

\def\colim{\qopname\relax m{colim}}

\newcommand{\arrdisp}{0.33ex}
\newcommand{\arrdisplacementsp}{0.72ex}

\newcommand{\ardis}{\ar@<\arrdisp>}
\newcommand{\ardissp}{\ar@<\arrdisplacementsp>}

\usepackage{contour}
\usepackage{ulem}

\contourlength{0.5pt}

\newcommand{\myuline}[1]{%
  \uline{\phantom{#1}}%
  \llap{\contour{white}{#1}}%
}

\makeatletter
\newcommand*{\saved@myuline}{}
\let\saved@myuline\myuline

\newcommand*{\mathuline}{%
  \mathpalette{\math@myuline\saved@myuline}%
}
\newcommand*{\math@myuline}[3]{%
  \mbox{#1{$#2#3\m@th$}}%
}

\renewcommand*{\myuline}{%
  \relax  
  \ifmmode
    \expandafter\mathuline
  \else
    \expandafter\saved@myuline
  \fi
}
\makeatother

\title{\LARGE Parametrised Presentability over Orbital Categories}
\date{ \today}
\author{\textsc{Kaif Hilman}\thanks{Max Planck Institute for Mathematics, Bonn\\kaif@mpim-bonn.mpg.de}}

\begin{document}
\maketitle

\begin{abstract}
    In this paper, we develop the notion of presentability in the parametrised homotopy theory framework of \cite{parametrisedIntroduction} over orbital categories. We formulate and prove a  characterisation of parametrised presentable categories in terms of its associated straightening. From this we deduce a parametrised adjoint functor theorem from the unparametrised version, prove various localisation results, and we record the interactions of the notion of presentability here with multiplicative matters. Such a theory is of interest for example in equivariant homotopy theory, and we will apply it in \cite{kaifNoncommMotives} to construct the category of parametrised noncommutative motives for equivariant algebraic K-theory.
\end{abstract}

\tableofcontents

\section{Introduction}

Parametrised homotopy theory is the study of higher categories fibred over a base $\infty$-category. This is a generalisation of the usual theory of higher categories, which can be viewed as the parametrised homotopy theory over a point. The advantage of this approach is that many structures can be cleanly encoded by the morphisms in the base $\infty$-category. For example, in the algebro-geometric world, various forms of pushforwards exist for various classes of scheme morphisms (see \cite{bachmannHoyois} for more details). Another example, which is the main motivation of this work for our applications in \cites{kaifNoncommMotives}, is that of genuine equivariant homotopy theory for a finite group $G$ - here the base $\infty$-category would be $\orbit_G\op$, the opposite of the $G$-orbit category. In this case, for subgroups $H\leq K\leq G$, important and fundamental constructions such as \textit{indexed coproducts}, \textit{indexed products}, and  \textit{indexed tensors} $$\coprod_{K/H} \quad\quad\quad\quad \prod_{K/H}\quad\quad\quad\quad\bigotimes_{K/H}$$ can be encoded by the morphisms in $\orbit_G\op$.  One framework in which to study this is the series of papers following \cite{parametrisedIntroduction} and the results in this paper should be viewed as a continuation of the vision from the aforementioned series - we refer the reader to them for more motivations and examples. \\

For an $\infty$-category to admit all small colimits and limits is a very desirable property as it means that many constructions can be done in it. However, this property entails that it has to be large enough and we might lose control of it due to size issues. Fortunately, there is a fix to this problem in the form of the very well-behaved class of  \textit{presentable} $\infty$-categories: these are cocomplete $\infty$-categories that are ``essentially generated'' by a small subcategory. One of the most important features of presentable $\infty$-categories is the adjoint functor theorem which says that one can test whether or not a functor between presentables is right or left adjoint by checking that it preserves limits or colimits respectively. The $\infty$-categorical theory of presentability was developed by Lurie in \cite[{Chapter 5}]{lurieHTT}, generalising the classical 1-categorical notion of \textit{locally presentable categories}.\\

The goal of this paper is to translate the above-mentioned theory of presentable $\infty$-categories to the parametrised setting and to understand the relationship between the notion of parametrised presentability and its unparametrised analogue in \cite{lurieHTT}. As a signpost to the expert reader, we will always assume throughout this paper that the base category is orbital in the sense of \cite{nardinThesis}. Besides that, we will adopt the convention in said paper of defining a $\baseCat$-category, for a fixed based $\infty$-category $\baseCat$, to be a cocartesian fibration over the \textit{opposite}, $\baseCat\op$. This convention is geared towards equivariant homotopy theory as introduced in the motivation above where $T = \orbit_G$. Note that by the straightening-unstraightening equivalence of \cite{lurieHTT}, a $\baseCat$-category can equivalently be thought of as a functor $\baseCat\op \rightarrow \widehat{\cat}_{\infty}$.  The first main result we obtain is then the following characterisation of parametrised presentable $\infty$-categories.

\begin{alphThm}[Straightening characterisation of parametrised presentables, full version in  {\cref{simpsonTheorem}}]\label{thmA}
Let $\sC$ be a $\baseCat$-category. Then it is $\baseCat$-presentable if and only if the associated straightening $C : \baseCat\op \rightarrow \widehat{\cat}_{\infty}$ factors through the non-full subcategory $ \presentable^{\mathrm{L}}\subset \widehat{\cat}_{\infty}$ of presentable categories and left adjoint functors, and morevoer these functors themselves have left adjoints satisfying certain Beck-Chevalley conditions (\ref{beckChevalleyMeaning}). 
\end{alphThm}

In the full version, we also give a complete parametrised analogue of the characterisations of presentable $\infty$-categories due to Lurie and Simpson (cf. \cite[{Thm. 5.5.1.1}]{lurieHTT}), which in particular shows that the notion defined in this paper is equivalent to the one defined in   \cite[{$\S1.4$}]{nardinThesis}. While it is generally expected that the theory of $\infty$-cosmoi in \cite{riehlVerityCosmos} should absorb the statement and proof of the Lurie-Simpson-style characterisations of presentability, the value of the theorem above is in clarifying the relationship between the notion of presentability in the parametrised sense and in the unparametrised sense. Indeed, the description in \cref{thmA} is a genuinely parametrised statement that is not seen in the unparametrised realm where $T = \ast$. One consequence of this is that we can easily deduce the parametrised adjoint functor theorem \textit{from} the unparametrised version instead of repeating the same arguments in the parametrised setting.

\begin{alphThm}[Parametrised adjoint functor theorem,  {\cref{parametrisedAdjointFunctorTheorem}}]\label{thmB}
Let $F : \sC \rightarrow \D$ be a $\baseCat$-functor between $\baseCat$-presentable categories. Then:
\begin{enumerate}
    \item If $F$ strongly preserves $\baseCat$-colimits, then $F$ admits a $\baseCat$-right adjoint.
    \item If $F$ strongly preserves $\baseCat$-limits and is $\baseCat$-accessible, then $F$ admits a $\baseCat$-left adjoint.
\end{enumerate}
\end{alphThm}

Another application of \cref{thmA} is the construction of \textit{presentable Dwyer-Kan localisations},  \cref{parametrisedPresentableDwyerKanLocalisation}. This is deduced essentially by performing fibrewise localisations, which are in turn furnished by \cite{lurieHTT}. It is an extremely important construction, much like in the unparametrised world, and we will for example use it in \cite{kaifNoncommMotives} to understand the parametrised enhancement of the noncommutative motives of \cite{BGT13}. Other highlights include the \textit{localisation-cocompletions} construction in \cref{arbitraryCocompletions}, the idempotent-complete-presentables correspondence \cref{TPresentableIdempotentCorrespondence}, as well as studying the various interactions between presentability with multiplicative matters and functor categories in \cref{subsec4.6:functorCategoriesAndPresentables} where, among other things, we prove a formula for the tensor product of parametrised presentable categories that was claimed in \cite{nardinThesis} without proof. \\

We now comment on the methods and philosophy of this article. The approach taken here is an axiomatic one and is slightly different in flavour from the series of papers in \cite{parametrisedIntroduction} in that we freely pass between the viewpoint of parametrised $\infty$-categories as cocartesian fibrations and as $\infty$-category-valued functors via the straightening-unstraightening equivalence of Lurie. This allows us to work model-independently, ie. without thinking of our $\infty$-categories as simplicial sets. The point is that, as far as presentability and adjunctions are concerned, the foundations laid in \cites{expose1Elements, barwickGlasmanNardin, shahThesis, shahPaperII, nardinThesis} are sufficient for us to make model-independent formulations and proofs via universal properties. Indeed, a recurring trick in this paper is to say that relevant universal properties guarantee the existence of certain functors, and then we can just check that certain diagrams of $\infty$-categories commute by virtue of the essential uniqueness of left/right adjoints.\\

\textbf{Outline of paper.} In \cref{section:PrelimsGeneralBaseCategories,section:PrelimsOrbitalBaseCategories,section:PrelimsAtomicOrbitalBaseCategories} we collect all the background materials, together with references, that will be needed for the rest of the paper. We hope that this establishes notational consistency and makes the paper as self-contained as possible. We have denoted by ``recollections'' those subsections which contain mostly only statements that have appeared in the literature. We recommend the reader to skim this section on first reading and refer to it as necessary. In \cref{sec4:Compactness} we introduce the notions of $\baseCat$-compactness and $\baseCat$-idempotent-completeness. We then come to the heart of the paper in \cref{sec5:presentability} where we state and prove various basic results about parametrised presentable $\infty$-categories as enumerated above. \\

\textbf{Conventions and assumptions.} This paper is written in the language of $\infty$-categories and so from now on we will drop the adjective $\infty$- and mean $\infty$-categories when we say categories. Moreover, throughout the paper the base category $\baseCat$ will be assumed to be \textit{orbital} (cf. Definition \ref{DefinitionAtomicOrbital}) unless stated otherwise. We will also use the notation $\catinf$ for the category of small categories and $\widehat{\cat}_{\infty}$ for the category of large categories. \\

\textbf{Related work.} Since the appearance of this article, Louis Martini and Sebastian Wolf have also independently produced many similar results in \cite{martiniWolf} using a different formalism of working internal to an $\infty$--topos. An important point of departure of this work from \cite{martiniWolf} is the the following: because this article is geared towards the applications we have in mind in \cite{kaifNoncommMotives}, our setup crucially interacts well with the notion of parametrised symmetric monoidal structures (a.k.a. multiplicative norms), which to the best of our knowledge, is not yet formulated in their formalism at the present moment. Moreover, this work has also been incorporated as Chapters 1 and 2 of the author's PhD thesis \cite{kaifThesis}.\\

\textbf{Acknowledgements.} I am grateful to Jesper Grodal, Markus Land, and Maxime Ramzi for useful comments, sanity checks, and many hours of enlightening conversations. I would also like to acknowledge my debt to Fabian Hebestreit and Ferdinand Wagner whose joint notes \cite{fabianWagner} on a lecture series given by the former have taught me much about how to use higher category theory model-independently. Finally, I would like to thank Malte Leip for proofreading the various drafts of this article. This article is based on work done during the author's PhD which was supported by the Danish
National Research Foundation through the Copenhagen Centre for Geometry and Topology
(DNRF151). \\

\section{Preliminaries: general base categories}\label{section:PrelimsGeneralBaseCategories}

\subsection{Recollections: basic objects and constructions}

\begin{recollect}\label{defn:baseCategoryT}\label{nota:Cat_T} \label{nota:totalCategory}
For a category $\baseCat$, there is Lurie's straightening-unstraightening equivalence $\cocartesianCategory(\baseCat\op) \simeq \func(\baseCat\op, \cat)$ (cf. for example \cite[{Thm. I.23}]{fabianWagner}). The category of $\baseCat$\textit{--categories} is then defined simply as $\func(\baseCat\op,\cat)$ and we also write this as $\cat_{\baseCat}$. We will always denote a $\baseCat$--category with an underline $\underline{\sC}$. Under the equivalence above, the datum of a $\baseCat$--category is equivalent to the datum of a cocartesian fibration $p \colon \totalCategory(\underline{\sC}) \rightarrow \baseCat\op$, and a $\baseCat$--functor is defined just to be a morphism of $\baseCat$--categories $\underline{\sC} \rightarrow \underline{\D}$, is then equivalently a map of cocartesian fibrations $\totalCategory(\underline{\sC})\rightarrow\totalCategory(\underline{\D})$ over $\baseCat\op$. For an object $V\in \baseCat$, we will write $\sC_V$ or $\underline{\sC}_V$ for the fibre of $\totalCategory(\underline{\sC})\rightarrow \baseCat\op$ over $V$. \label{nota:fibreC_V}
\end{recollect}

\begin{rmk}
The product $\underline{\sC}\times \underline{\D}$ in $\cat_{\baseCat}$  of two $\baseCat$--categories $\underline{\sC}, \underline{\D}$ is given as the pullback $\totalCategory(\underline{\sC})\times_{\baseCat\op} \totalCategory(\underline{\D})$ in the cocartesian fibrations perspective. We will always denote with $\times$ when we are viewing things as $\baseCat$--categories and we reserve $\times_{\baseCat\op}$ for when we are viewing things as total categories. In this way, there will be no confusion as to whether or not $\times_{\baseCat\op}$ denotes a pullback in $\cat_T$: this will never be the case.
\end{rmk}

\begin{nota}\label{nota:Fun_T}
Since $\cat_{\baseCat} = \func(\baseCat\op,\cat)$ is naturally even a 2--category, for $\underline{\sC},\underline{\D}\in \cat_{\baseCat}$, we have the \textit{category} of $\baseCat$--functors from $\underline{\sC}$ to $\underline{\D}$: this we write as $\func_{\baseCat}(\underline{\sC}, \underline{\D})$. Unstraightening, we obtain $\func_{\baseCat}(\underline{\sC}, \underline{\D}) \simeq \func\cocartesian(\totalCategory(\underline{\sC}), \totalCategory(\underline{\D})) \times_{\func(\totalCategory(\underline{\sC}), \baseCat\op)}\{p\}$ where $\func\cocartesian$ is the full subcategory of functors preserving $\baseCat\op$--cocartesian morphisms. 
\end{nota}

\begin{example}
We now give some basic examples of $\baseCat$--categories to set notation.
\begin{itemize}
    \item (Fibrewise $\baseCat$--categories) \label{nota:tconstantCat} Let $K \in \cat$. Write $\tconstant_\baseCat(K) \in \cat_T$ for the constant $K$--valued diagram. In other words, $\totalCategory(\tconstant_{\baseCat}(K))\simeq K\times \baseCat\op$. 
    \item We write $\terminalTCat \coloneqq \tconstant_\baseCat(\ast)$. This is clearly a final object in $\cat_{\baseCat} = \func(\baseCat\op,\cat)$. 
    
    \item (Corepresentable $\baseCat$--categories)\label{categoryOfPoints}
Let $V\in \baseCat$. Then we can consider the left (and so cocartesian) fibration associated to the functor $\map_{\baseCat}\colon \baseCat\op \rightarrow \spc$ and denote this $\baseCat$--category by $\underline{V}$. Note that $\totalCategory(\underline{V}) \simeq (\baseCat_{/V})\op$. By corepresentability of $\underline{V}$, we have $\func_{\baseCat}(\underline{V},\underline{\sC})  \simeq \sC_V$. To wit, for $K\in\cat$, by \cref{TFunctorCategory}, we have
\begin{equation*}
    \begin{split}
        \map_{\cat}\big(K, \func_{\baseCat}(\underline{V}, \underline{\sC})\big) 
        &\simeq \map_{\cat_{\baseCat}}\big(\underline{V}, \underline{\func}_{\baseCat}(\tconstant(K), \underline{\sC})\big)\\
        &\simeq \map_{\cat}(K, \sC_V)
    \end{split}
\end{equation*}
\end{itemize}
\end{example}

\begin{defn}
The category of $\baseCat$--\textit{objects of $\underline{\sC}$} is defined to be $\func_{\baseCat}(\terminalTCat, \underline{\sC})$. 
\end{defn}
\begin{rmk}
If $\baseCat\op$ has an initial object $T\in \baseCat\op$, then this means that the category of $\baseCat$--objects in $\underline{\sC}$ is just $\sC_T$.
\end{rmk}

\begin{cons}[Parametrised opposites]\label{defn:vop}
For a $\baseCat$--category $\underline{\sC}$, its $\baseCat$\textit{--opposite} $\underline{\sC}\vop$ is defined to be the image under the functor obtained by applying $\func(\baseCat\op, -)$ to $(-)\op \colon \cat \rightarrow \cat$. In the unstraightened view, this is given by taking fibrewise opposites in the total category. In \cite{expose1Elements} this was called vertical opposites $(-)^{\mathrm{vop}}$ to invoke just such an impression.
\end{cons}

\begin{obs}\label{oppositeCorepresentability}
Let $\underline{V}$ be a corepresentable $\baseCat$--category. Then $\underline{V}\vop \simeq \underline{V}$ since the functor $(-)\op \colon \cat \rightarrow \cat$ restricts to the identity on $\spc$.
\end{obs}

\begin{cons}\label{nota:Tcones}
The cone and cocone are functors $(-)^{\triangleleft}, (-)^{\triangleright} \colon \cat \rightarrow \cat$ which add a (co)cone point to a category. Applying $\func(\baseCat\op,-)$ to this functor yields the $\baseCat$\textit{--cone} and \textit{--cocone} functors $(-)\tcone$ and $(-)\tcocone$ respectively. We refer to \cite{shahThesis} for more on this.
\end{cons}

\begin{defn}
A $\baseCat$--functor is $\baseCat$\textit{-fully faithful} (resp. $\baseCat$\textit{-equivalence}) if it is so fibrewise. There is the expected characterisation of $\baseCat$--fully faithfulness in terms of $\baseCat$--mapping spaces, see \cref{fullyFaithfulMappingSpaces}.
\end{defn}

\begin{defn}\label{DefinitionAtomicOrbital}
We say that the category $\baseCat$ is \textit{orbital} if the finite coproduct cocompletion $\finite_{\baseCat}$ admits finite pullbacks. Here, by finite coproduct cocompletion, we mean the full subcategory of the presheaf category $\func(\baseCat\op,\spc)$ spanned by finite coproduct of representables. We say that it is \textit{atomic} if every retraction is an equivalence.
\end{defn}

\begin{nota}[Basechange]\label{nota:baseChangeCat}\label{nota:Fun_V}
As in \cite{nardinThesis}, we will write
$\underline{\sC}_{\underline{V}} \coloneqq \underline{\sC}\times \underline{V} = \totalCategory(\underline{\sC})\times_{\baseCat\op}\totalCategory(\underline{V})$ for the basechanged parametrised category, which is now viewed as a $\baseCat_{/V}$--category. The $(-)_{\underline{V}}$ is a useful reminder that we have basechanged to $V$, and so for example we will often use the notation $\func_{\underline{V}}$ to mean $\func_{\baseCat_{/V}}$ and \textit{not} $\func_{\totalCategory(\underline{V})} \simeq \func_{(\baseCat_{/V})\op}$.
\end{nota}

\begin{cons}[Internal $\baseCat$--functor category, {\cite[ $\S9$]{expose1Elements}}]\label{TFunctorCategory}
For $\underline{\sC}, \underline{\D}\in \cat_{\baseCat}$, there is a $\baseCat$--category $\underline{\func}_{\baseCat}(\underline{\sC},\underline{\D})$ such that 
\[\func_{\baseCat}(\underline{\E},\underline{\func}_{\baseCat}(\underline{\sC}, \underline{\D})) \simeq \func_{\baseCat}(\underline{\E} \times \underline{\sC}, \underline{\D})\] This is because $\func(\baseCat\op, \cat)$ is presentable and the endofunctor $- \times  \underline{\sC}$ has a right adjoint since it preserves colimits. In particular, by a Yoneda argument we get $\underline{\func}_{\baseCat}(\terminalTCat,\underline{\D}) \simeq \underline{\D}$.  Moreover, plugging in $\underline{\E} = \terminalTCat$ we see that $\baseCat$--objects of the internal $\baseCat$--functor object are just $\baseCat$--functors. Furthermore, the $\baseCat$--functor categories basechange well in that 
\[\underline{\func}_{\baseCat}(\underline{\sC}, \underline{\D})_{\underline{V}} \simeq \underline{\func}_{\underline{V}}(\underline{\sC}_{\underline{V}}, \underline{\D}_{\underline{V}})\]
so  the fibre over $V\in \baseCat\op$ is given by $\func_{\underline{V}}(\underline{\sC}_{\underline{V}}, \underline{\D}_{\underline{V}}).$ To wit, for any $\baseCat_{/V}$--category $\underline{\E}$,
\begin{equation*}
    \begin{split}
    \map_{(\cat_{\baseCat})_{/\underline{V}}}(\underline{\E}, \underline{\func}_{\baseCat}(\underline{\sC}, \underline{\D})_{\underline{V}}) &\simeq \map_{\cat_{\baseCat}}(\underline{\E}, \underline{\func}_{\baseCat}(\underline{\sC}, \underline{\D}))\\
    &\simeq \map_{(\cat_{\baseCat})_{/\underline{V}}}(\underline{\E}\times\underline{\sC}, \underline{\D}_{\underline{V}})\\
    &\simeq \map_{(\cat_{\baseCat})_{/\underline{V}}}(\underline{\E}\times_{\underline{V}}\sC_{\underline{V}}, \D_{\underline{V}})\\
    &\simeq \map_{(\cat_{\baseCat})_{/\underline{V}}}(\underline{\E}, \underline{\func}_{\underline{V}}(\sC_{\underline{V}},\D_{\underline{V}}))
    \end{split}
\end{equation*}
\end{cons}

\begin{nota}[Parametrised cotensors]\label{parametrisedCotensors}
Let $I$ be a small unparametrised category. Then the adjunction $-\times I : \cat \rightleftarrows \cat : \func(I,-)$ induces the adjunction
\[(-\times I)_* : \func(\baseCat\op,\cat) \rightleftarrows \func(\baseCat\op, \cat) : \func(I,-)_*\]
Under the identification $\func(\baseCat\op,\cat) \simeq \cat_{\baseCat}$ where $\cat_{\baseCat}$ is the category of ${\baseCat}$-categories, it is clear that $(-\times I)_*$ corresponds to the ${\baseCat}$-functor $\tconstant_{\baseCat}(I)\times-$, whose right adjoint we know is $\underline{\func}_{\baseCat}(\tconstant_{\baseCat}(I), -)$. Therefore $\underline{\func}_{\baseCat}(\tconstant_{\baseCat}(I), -)$ implements the \textit{fibrewise functor construction}. We will introduce the notation \label{nota:parametrisedCotensors} ${\fibfunc}(I,-)$ for $\underline{\func}_{\baseCat}(\tconstant_{\baseCat}(I), -)$. This satisfies the following properties whose proofs are immediate.
\begin{enumerate}
    \item $\underline{\cat}_{\baseCat}$ is cotensored over $\cat$ in the sense that for any ${\baseCat}$-categories $\underline{\sC}, \underline{\D}$ we have 
\[\underline{\func}_{\baseCat}(\underline{\sC},  {\fibfunc}(I, \underline{\D})) \simeq {\fibfunc}(I, \underline{\func}_{\baseCat}(\underline{\sC},  \underline{\D}))\]
\item ${\fibfunc}(I, -)$ preserves ${\baseCat}$-adjunctions.
\end{enumerate}
\end{nota}

\begin{obs}\label{opOfFunctorCats}
There is a natural equivalence of $\baseCat$--categories
\[\underline{\func}_{\baseCat}(\underline{\sC}, \underline{\D})\vop \simeq \underline{\func}_{\baseCat}(\underline{\sC}\vop, \underline{\D}\vop)\] This is because $(-)\vop : \cat_{\baseCat} \rightarrow \cat_{\baseCat} $ is an involution, and so  for any $\underline{\E} \in \cat_{\baseCat}$,
\begin{equation*}
    \begin{split}
        \map_{\cat_{\baseCat}}(\underline{\E}, \underline{\func}_{\baseCat}(\underline{\sC}, \underline{\D})\vop) &\simeq \map_{\cat_{\baseCat}}(\underline{\E}\vop, \underline{\func}_{\baseCat}(\underline{\sC}, \underline{\D})) \\
        &\simeq \map_{\cat_{\baseCat}}(\underline{\E}\vop \times \underline{\sC}, \underline{\D})\\
        &\simeq \map_{\cat_{\baseCat}}(\underline{\E} \times \underline{\sC}\vop, \underline{\D}\vop) \\
        &\simeq \map_{\cat_{\baseCat}}(\underline{\E}, \underline{\func}_{\baseCat}(\underline{\sC}\vop, \underline{\D}\vop))
    \end{split}
\end{equation*}
\end{obs}

\begin{cons}[Cofree parametrisation, {\cite[{ Def. 1.10}]{nardinThesis}}]\label{cofreeParametrisations}
Let $\D$ be a category. There is a $\baseCat$--category $\cofree(\D)\colon \baseCat\op \rightarrow \cat$ classified by $V \mapsto \func((\baseCat_{/V})\op, \D)$. This has the following universal property: if $\underline{\sC}\in\cat_{\baseCat}$, then there is a natural equivalence
\[\func_{\baseCat}(\underline{\sC},\cofree(\D)) \simeq \func(\totalCategory(\underline{\sC}),\D)\] of ordinary $\infty$-categories. This construction is of foundational importance and it allows us to define the following two fundamental $\baseCat$--categories.
\end{cons}

\begin{nota}\label{nota:TCatOfTCats} \label{nota:TCatOfTSpaces}
We will write $\underline{\cat}_{\baseCat} \coloneqq \cofree_T(\cat)$ for the $\baseCat$--\textit{category of $\baseCat$--categories}; we write $\underline{\spc}_{\baseCat} \coloneqq \cofree_T(\spc)$ for the \textit{$\baseCat$--category of $\baseCat$--spaces}.
\end{nota}

\begin{thm}[Parametrised straightening-unstraightening, {\cite[{Prop. 8.3}]{expose1Elements}}]\label{parametrisedStraightenUnstraighten}
Let $\underline{\sC}\in \cat_{\baseCat}$. Then there are equivalences 
\[\func_{\baseCat}(\underline{\sC},\underline{\cat}_{\baseCat}) \simeq \cocartesianCategory(\totalCategory(\underline{\sC}))\quad\quad \func_{\baseCat}(\underline{\sC},\underline{\spc}_{\baseCat}) \simeq \mathrm{Left}(\totalCategory(\underline{\sC}))\]
\end{thm}
\begin{proof}
This is an immediate consequence of the usual straightening-unstraightening and the universal property of $\baseCat$--categories of $\baseCat$--objects above. For example,
\[\func_{\baseCat}(\underline{\sC},\underline{\cat}_{\baseCat}) \simeq \func(\totalCategory(\underline{\sC}),\cat) \simeq \cocartesianCategory(\totalCategory(\underline{\sC}))\] and  similarly for spaces.
\end{proof}


\subsection{Parametrised adjunctions}
$\baseCat$--adjunctions as introduced in \cite{shahThesis} is based on the \textit{relative adjunctions} of \cite{lurieHA}.

\begin{defn}[{\cite[{Def. 7.3.2.2}]{lurieHA}}]
Suppose we have diagrams of categories
\begin{center}
    \begin{tikzcd}
    \sC \ar[dr,"q"'] && \D\ar[ll, "G"']\ar[dl,"p"] &&& \sC \ar[dr,"q"']\ar[rr, "F"] && \D\ar[dl,"p"]\\
    & \E &&&&  & \E
    \end{tikzcd}
\end{center}
Then we say that:
\begin{itemize}
    \item For the first diagram, $G$ admits a \textit{left adjoint} $F$ \textit{relative to} $\E$ if $G$ admits a left adjoint $F$ such that for every $C\in \sC$, $q$ sends the unit $\eta : C \rightarrow GFC$ to an equivalence in $\E$ (equivalently, if $q\eta : q \Rightarrow p\circ F$ exhibits a commutation $p\circ F \simeq q$ by \cite[{Prop. 7.3.2.1}]{lurieHA}).
    \item For the second diagram, $F$ admits a \textit{right adjoint} $G$ \textit{relative to} $\E$ if $F$ admits a right adjoint $G$ such that for every $D\in \D$, $p$ maps the counit $\varepsilon : FGD \rightarrow D$ to an equivalence in $\E$ (equivalently if $p\varepsilon : q\circ G \Rightarrow p$ exhibits $q\circ G \simeq p$ by \cite[{Prop. 7.3.2.1}]{lurieHA}).
\end{itemize}
Observe that when $\E \simeq \ast$, this specialises to the usual notion of adjunctions.
\end{defn}

\begin{rmk}
These two definitions are compatible. To see this, assume the first condition for example, ie. that $G$ has a left adjoint $F$ relative to $\E$. We need to see that $F$ then admits a right adjoint $G$ relative to $\E$ in the sense of the second condition, ie. that $p$ sends the counit $\varepsilon : FGD \rightarrow D$ to an equivalence in $\E$. For this just consider the commutative diagram
\begin{center}
    \begin{tikzcd}
    qG \rar["q\eta_{G}", "\simeq"'] \ar[dr,equal]& qGFG \dar[d, "qG\varepsilon"] \rar["\simeq"] & pFG \dar["p\varepsilon"] \\
    & qG \rar["\simeq"] & p
    \end{tikzcd}
\end{center}
where the triangle is by the adjunction, and the square is by the natural equivalence $qG \simeq p$.
\end{rmk}

\begin{defn}\label{bousfieldLocalisationDefinition}
Let $\underline{\sC}, \underline{\D}\in\cat_{\baseCat}$. Then a $\baseCat$\textit{-adjunction} $F : \underline{\sC}\rightleftarrows \underline{\D} : G$ is defined to be a relative adjunction such that $F, G$ are $\baseCat$--functors. A $\baseCat$\textit{-Bousfield localisation} is a $\baseCat$--adjunction where the $\baseCat$--right adjoint is $\baseCat$--fully faithful.
\end{defn}

\begin{prop}[Stability of relative adjunctions under pullbacks, {\cite[{Prop. 7.3.2.5}]{lurieHA}}]\label{relativeAdjunctionPullbacks}
Suppose we have a relative adjunction 
\begin{center}
    \begin{tikzcd}
    \sC \ar[dr,"q"'] \ar[rr, shift left  = 3, "F"]&& \D\ar[ll, "G"]\ar[dl,"p"]\\
    & \E 
    \end{tikzcd}
\end{center} 
Then for any functor $\E'\rightarrow \E$ the diagram of pullbacks 
\begin{center}
    \begin{tikzcd}
    \sC\times_{\E}\E' \ar[dr,"q"'] \ar[rr, shift left  = 3, "F"]&& \D\times_{\E}\E'\ar[ll, "G"]\ar[dl,"p"]\\
    & \E'
    \end{tikzcd}
    
\end{center} 
is again a relative adjunction.
\end{prop}

We now have the following criteria to obtain relative adjunctions - these are just modified from Lurie's more general assumptions.

\begin{prop}[Criteria for relative adjunctions, {\cite[Prop. 7.3.2.6]{lurieHA}}]\label{criteriaRelativeAdjunctions}
Suppose $p : \sC \rightarrow \E$, $q : \D \rightarrow \E$ are cocartesian fibrations. If we have a map of cocartesian fibrations $F$
\begin{center}
    \begin{tikzcd}
    \sC \ar[dr,"q"'] \ar[rr, "F"]&& \D\ar[dl,"p"]\\
    & \E 
    \end{tikzcd}
    
\end{center} 
Then:
\begin{itemize}
    \item[(1)] $F$ admits a right adjoint $G$ relative to $\E$ if and only if for each $E \in \E$ the map of fibres $F_E : \sC_E\rightarrow \D_E$ admits a right adjoint $G_E$. The right adjoint need no longer be a map of cocartesian fibrations.
    
    \item[(2)] $F$ admits a left adjoint $L$ relative to $\E$ if and only if for each $E \in \E$ the map of fibres $F_E : \sC_E\rightarrow \D_E$ admits a left adjoint $L_E$ and the canonical comparison maps (constructed in \cite[{Prop. 7.3.2.11}]{lurieHA})
    \[Lf^* \rightarrow LFf^*L \xrightarrow{\varepsilon} f^*L\] constructed from the fibrewise adjunction are equivalences - here $f^*$ is the pushforward given by the cocartesian lift along some $f : E' \rightarrow E$ in $\E$. The relative left adjoint, if it exists, must necessarily be a map of cocartesian fibrations.
\end{itemize}
\end{prop}
\begin{proof}
We prove each in turn. To see (1), suppose $F$ has an $\E$-right adjoint $G$. Then for each $e\in \E$ the inclusion $\{e\} \hookrightarrow \E$ induces a pullback relative adjunction over the point $\{e\}$ by \cref{relativeAdjunctionPullbacks}, and so we get the statement on fibres. Conversely, suppose we have fibrewise right adjoints. To construct an $\E$-right adjoint $G$, since adjunctions can be constructed objectwise by the unparametrised version of \cref{pointwiseConstructionOfAdjunction} below, we need to show that for each $e\in \E$ and $d\in \D_e$,    there is a  $Gd\in \sC_e$ and a map $\varepsilon : FGd \rightarrow d$ such that:
    \begin{itemize}
        \item[(a)] For every $c \in \sC$ the following composition  is an equivalence
        \[\map_{\sC}(c, Gd) \xrightarrow{F} \map_{\D}(Fc,FGd) \xrightarrow{\varepsilon} \map_{\D}(Fc, d)\]
        \item[(b)] The morphism $p\varepsilon : pFGd \rightarrow pd$ is an equivalence in $\E$.
    \end{itemize}
    We can just define $Gd \coloneqq G_e(d) \in \sC_e$ given by the fibrewise right adjoint and let $\varepsilon : FGd \rightarrow d$ be the fibrewise counit. Since these are fibrewise, point (b) is automatic. To see point (a), let $c \in \sC_{e'}$ for some $e' \in \E$. Since the mapping space in the total category of cocartesian fibrations are just disjoint unions over the components lying under $\map_{\sC}(c,Gd)$, we can work over some $f \in \map_{\E}(e', e)$. Consider
    \begin{center}
        \begin{tikzcd}
        \map^f_{\sC}(c, Gd) \dar["\simeq"] \rar & \map_{\D_e}(f^*Fc, d) \simeq \map_{\D}^f(Fc,d)\\
        \map_{\sC_{e}}(f^*c, Gd) \rar["F"] \ar[ur, "\simeq"] & \map_{\D_e}(f^*Fc, FGd)\uar["\varepsilon"]
        \end{tikzcd}
    \end{center}
    where we have used also that $F$ was a map of cocartesian fibrations so that $f^*F\simeq Ff^*$ and that the diagonal map is an equivalence since we had a fibrewise adjunction $F_e \dashv G_e$ by hypothesis. This completes the proof of part (1). 
    
    For case (2), to see the cocartesianness of a relative left adjoint $L$,  note
    \begin{equation*}
        \begin{split}
            \map_{\sC}(Lf^*d, c) \simeq \map_{\D}(f^*d, Fc) &\simeq \map^f_{\D}(d, Fc)\\
            &\simeq \map_{\sC}^f(Ld, c)\simeq \map_{\sC}(f^*Ld,c)
        \end{split}
    \end{equation*}
    The proof for right adjoints in (1) go through in this case but now we use 
    \begin{center}
        \begin{tikzcd}
            \map_{\sC}^f(Ld, c) \dar["\simeq"] \rar & \map_{\D_E}(f^*d, Fc) \simeq \map_{\D}^f(d, Fc)\\
            \map_{\sC_E}(f^*Ld, c) \rar & \map_{\sC_E}(Lf^*d,c) \uar["\simeq"]
        \end{tikzcd}
    \end{center}
    and so the technical condition in the statement says that there is a canonical map inducing the bottom map in the square which must necessarily be an equivalence.
\end{proof}

\begin{rmk}
One might object to the notation we have adopted for the pushforward being $f^*$ instead of $f_!$. This convention is standard in the framework of \cite{parametrisedIntroduction} because the latter notation is reserved for the \textit{left adjoint} of $f^*$ (the so-called $\baseCat$--coproducts) that will be recalled later.
\end{rmk}

\begin{cor}[Fibrewise criteria for $\baseCat$--adjunctions]\label{criteriaForTAdjunctions}
Let $F : \underline{\sC} \rightarrow \underline{\D}$ be a $\baseCat$--functor. Then it admits a $\baseCat$--right adjoint if and only if it has fibrewise right adjoints $G_V$ for all $V\in \baseCat$ and 
\begin{center}
    \begin{tikzcd}
    \sC_W & \D_W \lar["G_W"']\\
    \sC_V \uar["f^*"]& \D_V\lar["G_V"]\uar["f^*"']
    \end{tikzcd}
\end{center}
commutes for all $f : W \rightarrow V$ in $\baseCat$. Similarly for left $\baseCat$--adjoints.
\end{cor}
\begin{proof}
The commuting square ensures that the relative right adjoint is a $\baseCat$--functor.
\end{proof}

\begin{prop}[Criteria for $\baseCat$--Bousfield localisations, ``{\cite[{Prop. 5.2.7.4}]{lurieHTT}}'']\label{HTT.5.2.7.4}
Let $\underline{\sC}\in\cat_{\baseCat}$ and $L : \underline{\sC} \rightarrow \underline{\sC}$ a $\baseCat$--functor equipped with a fibrewise natural transformation $\eta : \id \Rightarrow L$. Let $j : L\underline{\sC}\subseteq \underline{\sC}$ be the inclusion of the $\baseCat$--full subcategory spanned by the image of $L$. Suppose the transformations $L\eta, \eta_L : L \Longrightarrow L\circ L$ are equivalences. Then the pair $(L, j)$ constitutes a $\baseCat$--Bousfield localisation with unit $\eta$.
\end{prop}
\begin{proof}
We want to apply \cref{criteriaForTAdjunctions}. Since we are already provided with the fact that $L$ was a $\baseCat$--functor, all that is left to show is that it is fibrewise left adjoint to the inclusion $L\underline{\sC} \subseteq \underline{\sC}$. But this is guaranteed by \cite[{Prop. 5.2.7.4}]{lurieHTT}, and so we are done.
\end{proof}

Finally, we show that parametrised adjunctions have the expected internal characterisation in terms of the parametrised mapping spaces recalled in \cref{mappingAnimaYoneda}.

\begin{lem}[Mapping space characterisation of $\baseCat$--adjunctions]\label{paramAdjunctionMappingAnima}
Let $F : \underline{\sC} \leftrightarrows \underline{\D} : G$ be a pair of $\baseCat$--functors. Then there is a $\baseCat$--adjunction $F\dashv G$ if and only if we have a natural equivalence 
\[\myuline{\map}_{\underline{\D}}(F-, -) \simeq \myuline{\map}_{\underline{\sC}}(-, G-) : \underline{\sC}\vop  \times  \underline{\D} \longrightarrow \underline{\spc}_{\baseCat}\] 
\end{lem}
\begin{proof}
The if direction is clear: since $F$ and $G$ were already $\baseCat$--functors, by \cref{criteriaForTAdjunctions} the only thing left to do is to show fibrewise adjunction, and this is easily implied by the equivalence which supplies the unit and counits. For the only if direction, by definition of a relative adjunction, we have a fibrewise natural transformation $\eta : \id_{\sC}\Rightarrow GF$ (ie. a morphism in $\func_{\baseCat}(\underline{\sC},\underline{\sC}) $) and so we obtain a natural comparison
\[\myuline{\map}_{\underline{\D}}(F-,-) \xrightarrow{G}\myuline{\map}_{\underline{\sC}}(GF-, G-) \xrightarrow{\eta^*} \myuline{\map}_{\underline{\sC}}(-,G-)\] Since equivalences between $\baseCat$--functors are checked fibrewise, let $c\in \sC_V, d\in \D_V$. Then 
\begin{center}
\begin{adjustbox}{scale =0.84}
    \begin{tabular}{l l l}
         $\myuline{\map}_{\underline{\D}}(F-, -) \: :$ & $(c,d) \mapsto$ &  $\Big((W\xrightarrow{f} V)\mapsto (\map_{\D_V}(Fc,d)\rightarrow \map_{\D_W}(f^*Fc, f^*d)\Big) \in \underline{\spc}_{\underline{V}}$\\
         $\myuline{\map}_{\underline{\sC}}(-, G-) \: :$ & $(c,d) \mapsto$ &$\Big((W\xrightarrow{f} V)\mapsto (\map_{\sC_V}(c,Gd)\rightarrow \map_{\sC_W}(f^*c, f^*Gd)\Big)\in \underline{\spc}_{\underline{V}}$
    \end{tabular}
\end{adjustbox}    
\end{center}
Since $F, G$ were $\baseCat$--functors, we have $Ff^* \simeq f^*F$ and $Gf^*\simeq f^*G$, and so the natural comparison coming from the relative adjunction unit given above exhibits a pointwise equivalence between $\myuline{\map}_{\underline{\D}}(F-,-)$ and $\myuline{\map}_{\underline{\sC}}(-,G-)$ by \cref{criteriaForTAdjunctions}.
\end{proof}

\section{Preliminaries: orbital base categories}\label{section:PrelimsOrbitalBaseCategories}
Some of the notions here still make sense for general $\baseCat$, but we want orbitality in order to make formulations involving Beck-Chevalley conditions. Hence, from now on, we assume that $\baseCat$ is orbital. 

\subsection{Recollections: colimits and Kan extensions}\label{subsection:parametrisedColimits}

\begin{defn}\label{TColimits}
Let $\underline{K} \in \cat_{\baseCat}$ and $q \colon \underline{K}\rightarrow\terminalTCat$ be the unique map. Then precomposition induces the $\baseCat$--functor
$q^* : \underline{\D} \simeq \underline{\func}_{\baseCat}(\terminalTCat,\underline{\D}) \longrightarrow \underline{\func}_{\baseCat}(\underline{K},\underline{\D})$. The $\baseCat$--left adjoint $q_!$, if it exists, is called the $\underline{K}$--indexed $\baseCat$\textit{--colimit}, and similarly for \textit{T-limits} $q_*$.
\end{defn}

\begin{example}\label{importantClassOfColimits}
Here are some special and important classes of these:
\begin{itemize}
    \item A $\baseCat$--(co)limit indexed by $\tconstant_\baseCat(K)$ for some ordinary $\infty$--category $K$ is called a \textit{fibrewise} $\baseCat$\textit{-(co)limit.} 
    \item A $\baseCat$--(co)limit indexed by a corepresentable $\baseCat$--category $\underline{V}$ (cf. \cref{categoryOfPoints}) of some $V\in \baseCat$ is called the \textit{T-(co)product.}
\end{itemize}
\end{example}

\begin{defn}[{\cite[{Def. 11.2}]{shahThesis}}]\label{definitionStrongPreservation}
Let $F : \underline{\sC} \rightarrow \underline{\D}$ be a $\baseCat$--functor. 
\begin{itemize}
    \item We say that it \textit{preserves} $\baseCat$\textit{-colimits} if for all $\baseCat$--colimit diagrams $\overline{d} : K \tcocone \rightarrow \underline{\sC}$, the post-composed diagram $F\circ \overline{d} : K\tcocone \rightarrow \underline{\D}$ is a $\baseCat$--colimit. Similarly for $\baseCat$--limits.
    
    \item We say that $F$ \textit{strongly preserves} $\baseCat$\textit{-colimits} if for all $V \in \baseCat$, $F_{\underline{V}} : \underline{\sC}_{\underline{V}}\rightarrow \underline{\D}_{\underline{V}}$ preserves $\baseCat_{/V}$--colimits. Similarly for $\baseCat$--limits.
    \end{itemize}
\end{defn}

\begin{warning}[{\cite[{Rmk. 5.14}]{shahThesis}}]\label{shahStrongPreservationWarning}
Note that being $\baseCat$--cocomplete is much stronger than just admitting all $\baseCat$--colimits. This is because admitting all $\baseCat$--colimits just means that any $\baseCat_{/V}$--diagram $\underline{K}_{\underline{V}}\rightarrow \underline{\sC}_{\underline{V}}$ pulled back from a $\baseCat$--diagram $\underline{K}\rightarrow \underline{\sC}$ admits a $\baseCat_{/V}$--colimit. However not every $\baseCat_{/V}$--diagram is pulled back as such. We will elaborate on the distinction of these definitions in the next subsection. In this document, we will never consider preservations, but only strong preservations.
\end{warning}


\begin{defn}[{\cite[{Def. 5.13}]{shahThesis}}]
Let $\underline{\sC}\in\cat_{\baseCat}$. Then we say $\underline{\sC}$ is $\baseCat$\textit{--(co)complete} if for all $V\in \baseCat$ and $\baseCat_{/V}$--diagram $p : \underline{K} \rightarrow \underline{\sC}_{\underline{V}}$ with $\underline{K}$ small, $p$ admits a $\baseCat_{/V}$--(co)limit. 
\end{defn}

\begin{terminology}\label{term:strongAdmission}
When we want to specify particular kinds of parametrised (co)limits that a $\baseCat$--category admits, it is convenient to use the following terminology: for $\K = \{\K_V\}_{V\in \baseCat}$ some collection of diagrams varying over $V\in \baseCat$, we say that $\underline{\sC}$ \textit{strongly admits} $\K$--(co)limits if for all $V\in \baseCat$, $\underline{\sC}_{\underline{V}}$ admits $\underline{K}$--colimits for all $\underline{K}\in \K_V$. Examples include:
\begin{itemize}
    \item $\underline{\sC}$ strongly admits all $\baseCat$--(co)limits means that it is $\baseCat$--(co)cocomplete,
    \item Let $\kappa$ be a regular cardinal. We say that $\underline{\sC}$ strongly admits $\kappa$--small $\baseCat$--(co)products to mean that it has $\baseCat$--(co)limits for any diagram indexed over $\coprod_{a\in A}\underline{V}_a$ where $A$ is a $\kappa$--small set. Hence, strongly admitting finite $\baseCat$--(co)products means admitting finite fibrewise (co)products and (co)limits for all corepresentable diagrams $\underline{V}$.
\end{itemize}
\end{terminology}

\begin{lem}[Decomposition of indexed coproducts]\label{decompositionOfIndexedCoproducts}
Let $R_a, V \in \baseCat$ and $\coprod_af_a : \coprod_a R_a \rightarrow V$ be a map where the coproduct is not necessarily finite. Suppose $\underline{\sC}$ strongly admits finite  $\baseCat$--coproducts and arbitrary fibrewise coproducts. Then $\underline{\sC}$ admits $\coprod_af_a$-coproducts and this is computed by composing fibrewise $\baseCat$--coproduct $\coprod_a$ with the individual indexed $\baseCat$--coproducts.
\end{lem}
\begin{proof}
We will in fact show that we have  $\baseCat_{/V}$--adjunctions
\begin{center}
\adjustbox{scale=0.9}{
    \begin{tikzcd}
    \underline{\func}_{\underline{V}}(\coprod_a\underline{R_a}, \underline{\sC}_{\underline{V}}) = \prod_a\underline{\func}_{\underline{V}}(\underline{R_a}, \underline{\sC}_{\underline{V}}) \rar[shift left = 2, "\prod_a(f_a)_!"] & \prod_a\underline{\func}_{\underline{V}}(\underline{V}, \sC_{\underline{V}}) \rar[shift left= 2, "\coprod_a"] \lar[shift left = 2, "\prod_a(f_a)^*"]& \underline{\func}_{\underline{V}}(\underline{V}, \underline{\sC}_{\underline{V}})\lar[shift left = 2, "\Delta"]
    \end{tikzcd}
    }
\end{center}
That these $\baseCat_{/V}$--adjunctions exist is by our hypotheses, and all that is left to do is check that $\prod_a(f_a)^*\circ \Delta \simeq (\coprod_af_a)^*$. But this is also clear since we have the commuting diagram
\begin{center}
    \begin{tikzcd}
    \coprod_a\underline{R_a} \rar["\coprod_af_a"]\dar["f_a"'] & \underline{V}\\
    \coprod_a \underline{V} \ar[ur, "\coprod_a\id_V"']
    \end{tikzcd}
\end{center}
Applying $(-)^*$ to this triangle completes the proof.
\end{proof}

\begin{terminology}[Beck-Chevalley conditions]\label{beckChevalleyMeaning}
Let $\underline{\sC}\in\cat_{\baseCat}$ that admits finite fibrewise coproducts (resp. products) and such that for each $f : W \rightarrow V$ in $\baseCat$, $f^* : \sC_V \rightarrow \sC_W$ admits a left adjoint $f_!$ (resp. right adjoint $f_*$). We say that $\underline{\sC}$ satisfies the \textit{left Beck-Chevalley condition} (resp. \textit{right Beck-Chevalley condition}) if for every pair of morphisms $f: W\rightarrow V$ and $g : Y\rightarrow V$ in $\baseCat$: in the pullback (whose orbital decomposition exists by orbitality of $\baseCat$) 
        \begin{center}
            \begin{tikzcd}
            \coprod_a R_a = Y\times_VW\dar[dr, phantom, very near start, "\scalebox{1.5}{$\lrcorner$}"] \rar["\coprod_af_a"]\dar["\coprod_ag_a"'] & Y \dar["g"]\\
            W \rar["f"'] & V
            \end{tikzcd}
        \end{center}
        the canonical basechange transformation 
        \[\coprod_ag_{a_!}f^*_a \Longrightarrow f^*g_! \quad \Big(\text{resp.   } f^*g_* \Longrightarrow \prod_ag_{a_*}f^*_a \Big)\] is an equivalence.
\end{terminology}

Here is an omnibus of results due to Jay Shah.

\begin{thm}[{\cite[5.5-5.12 and $\S12$]{shahThesis}
}]\label{omnibusTColimits}
Let $\underline{\sC}\in\cat_{\baseCat}$. Then:
\begin{enumerate}
        \item[(1)] (Fibrewise criterion) $\underline{\sC}$ strongly admits $\baseCat$--colimits indexed by $\tconstant_\baseCat(K)$ if and only if for every $V\in \baseCat$ the fibre $\sC_V$ has all colimits indexed by $K$ and for every morphism $f : W\rightarrow V$ in $\baseCat$ the cocartesian lift $f^* : \sC_{V}\rightarrow \sC_W$ preserves colimits indexed by $K$. A cocone diagram $\overline{p} : \tconstant_{\baseCat}(K)\tcocone \rightarrow \underline{\sC}$ is a $\baseCat$--colimit if and only if it is so fibrewise.
    \item[(1)] (T-coproducts criteria) $\underline{\sC}$ strongly admits finite $\baseCat$--coproducts if and only if we have:
    \begin{itemize}
        \item[(a)] For every $W\in \baseCat$ the fibre $\sC_W$ has all finite coproducts and for every $f : W\rightarrow V$ in $\baseCat$ the map $f^* : \sC_{V}\rightarrow \sC_{W}$ preserves finite coproducts,
        \item[(b)] $\sC$ satisfies the left Beck-Chevalley condition (cf. \cref{beckChevalleyMeaning}).
    \end{itemize}
    \item[(3)] (Decomposition principle) $\underline{\sC}$ is $\baseCat$--cocomplete if and only if it has all fibrewise colimits and strongly admits finite $\baseCat$--coproducts.
\end{enumerate}
Similar statements hold for $\baseCat$--limits, and the right adjoint to $f^*$ will be denoted $f_*$.
\end{thm}

\begin{thm}[Omnibus $\baseCat$--adjunctions, {\cite[$\S8$]{shahThesis}}]\label{omnibusTAdjunctions}
Let $F : \underline{\sC} \rightleftarrows \underline{\D} : G$ be a $\baseCat$--adjunction and $\underline{I}$ be a $\baseCat$--category. Then:
\begin{enumerate}
    \item[(1)] We get adjunctions
$F_* \colon \underline{\func}_{\baseCat}(\underline{I},\underline{\sC})  \rightleftarrows \underline{\func}_{\baseCat}(\underline{I},\underline{\D}) : G_*$, $G^* \colon \underline{\func}_{\baseCat}(\underline{\sC},\underline{I}) \rightleftarrows \underline{\func}_{\baseCat}(\underline{\D},\underline{I}) : F^*$. By \cref{criteriaForTAdjunctions} this implies ordinary adjunctions when we replace $\underline{\func}_{\baseCat}$ by $\func_{\baseCat}$.

\item[(2)] $F$ strongly preserves $\baseCat$--colimits and $G$ strongly preserves $\baseCat$--limits.
\end{enumerate}
\end{thm}
\begin{proof}
In \cite[Cor. 8.9]{shahThesis}, part (2) was stated only as ordinary preservation, not strong preservation. But then strong preservation was implicit since relative adjunctions are stable under pullbacks by \cref{relativeAdjunctionPullbacks}, and the statement in \cite{shahThesis} also holds after pulling back to $-\times\underline{V}$ for all $V\in \baseCat$.
\end{proof}

\begin{prop}[$\baseCat$--cocompleteness of Bousfield local subcategories]\label{colimitsInBousfieldLocals}
If $L : \underline{\sC} \rightleftarrows \underline{\D} : j$ is a $\baseCat$--Bousfield localisation where $\underline{\sC}$ is $\baseCat$--cocomplete, then $\underline{\D}$ is too and $\baseCat$--colimits in $\underline{\D}$ is computed as $L$ applied to the $\baseCat$--colimit computed in $\underline{\sC}$.
\end{prop}
\begin{proof}
This is an immediate consequence of \cref{paramAdjunctionMappingAnima}.
\end{proof}

\begin{prop}[$\baseCat$--(co)limits of functor categories is pointwise]\label{TColimitsFunctorCategories}
Let $\underline{K}, \underline{I}, \underline{\sC}$ be $\baseCat$--categories. Suppose $\underline{\sC}$ strongly admits $\underline{K}$--indexed diagrams. Then so does $\underline{\func}_{\baseCat}(\underline{I}, \underline{\sC}) $  and the parametrised (co)limits are inherited from that of $\underline{\sC}$.
\end{prop}
\begin{proof}
This is a direct consequence of the adjunction
$(\underline{\colim}_{\underline{K}})_* : \underline{\func}_{\baseCat}(\underline{K}, \underline{\func}_{\baseCat}(\underline{I}, \underline{\sC}) ) \simeq \underline{\func}_{\baseCat}(\underline{I}, \underline{\func}_{\baseCat}(\underline{K}, \underline{\sC}) ) \rightleftarrows \underline{\func}_{\baseCat}(\underline{I},\underline{\sC})  : \constant$ for $\baseCat$--colimits. The other case is similar.
\end{proof}

\begin{defn}[$\baseCat$--Kan extensions]\label{TKanExtensions}
Let $j : \underline{I} \rightarrow \underline{K}$ be a $\baseCat$--functor. If
$j^* : \underline{\func}_{\baseCat}(\underline{K},\underline{\D}) \longrightarrow \underline{\func}_{\baseCat}(\underline{I},\underline{\D})$ has a $\baseCat$--left adjoint, then we denote it by $j_!$ and call it the \textit{T-left Kan extension}. Similarly for $\baseCat$--right Kan extensions.
\end{defn}

\begin{prop}[Fully faithful $\baseCat$--Kan extensions, {\cite[Prop. 10.6]{shahThesis}}]\label{fullyFaithfulTKanExtension}
Let $i : \underline{\sC} \hookrightarrow \underline{\D}$ be a $\baseCat$--fully faithful functor and $F : \underline{\sC} \rightarrow \underline{\E}$ be another $\baseCat$--functor. If the $\baseCat$--left Kan extension $i_!F$ exists, then the adjunction unit $F \Rightarrow i^*i_!F \colon \underline{\sC} \rightarrow \underline{\E}$ is an equivalence.
\end{prop}

\begin{thm}[Omnibus $\baseCat$--Kan extensions, {\cite[Thm. 10.5]{shahThesis}}]\label{omnibusTKanExtensions}
Let $\underline{\sC}\in\cat_{\baseCat}$ be $\baseCat$--cocomplete. Then for every $\baseCat$--functor of small $\baseCat$--categories $f : \underline{I} \rightarrow \underline{K}$, the $\baseCat$--left Kan extension $f_! : \underline{\func}_{\baseCat}(\underline{I},\underline{\sC})  \longrightarrow \underline{\func}_{\baseCat}(\underline{K},\underline{\sC})$ exists. 
\end{thm}

\subsection{Strong preservation of \texorpdfstring{$\baseCat$}--colimits}
We now explain in more detail the notion of strong preservation. In particular, the reader may find \cref{characterisationStrongPreservations} to be a convenient alternative description, and we will have many uses of it in the coming sections.

\begin{obs}[Strong preservations vs preservations]\label{strongPreservationObservation}
Here are some comments for the distinction. \cref{characterisationStrongPreservations} will then characterise strong preservations more concretely.
\begin{enumerate}
    \item[(1)] Recall \cref{shahStrongPreservationWarning} that admitting $\baseCat$--colimits is weaker than being $\baseCat$--cocomplete. In the proof of the Lurie-Simpson characterisation \cref{simpsonTheorem}, we will see that we really need $\baseCat$--cocompleteness via \cref{characterisationStrongPreservations}.
    
    \item[(2)] However, $\underline{\sC}$ admitting $\baseCat$--colimits indexed by $p : \underline{K}\rightarrow \baseCat\op$ does imply $\underline{\sC}_{\underline{V}}$ admits $\baseCat_{/V}$--colimits indexed by $\underline{K}_{\underline{V}}$. This is because the adjunction $p_! \colon \underline{\func}_{\baseCat}(\underline{K},\underline{\sC})  \rightleftarrows \underline{\func}_{\baseCat}(\terminalTCat, \underline{\sC})  : p^*$ pulls back to the $p_! \colon \underline{\func}_{\underline{V}}(\underline{K}_{\underline{V}},\underline{\sC}_{\underline{V}}) \rightleftarrows \underline{\func}_{\underline{V}}({\underline{V}}, \underline{\sC}_{\underline{V}}) : p^*$ adjunction by \cref{relativeAdjunctionPullbacks}. We have also used that functor $\baseCat$--categories basechange well by \cref{TFunctorCategory}.
    
\item[(3)] Strongly preserving fibrewise $\baseCat$--(co)limits is equivalent to preserving these (co)limits on each fibre since by \cref{omnibusTColimits} fibrewise (co)limits are constructed fibrewise.
\end{enumerate}
\end{obs}

The following result was also recorded in the recent \cite[Thm. 8.6]{shahPaperII}.

\begin{prop}[Characterisation of strong preservations]\label{characterisationStrongPreservations}
Let $\underline{\sC},\underline{\D}$ be $\baseCat$--cocomplete categories and $F : \underline{\sC} \rightarrow \underline{\D}$ a $\baseCat$--functor. Then $F$ strongly preserves $\baseCat$--colimits if and only if it preserves colimits in each fibre and for all $f : W \rightarrow V$ in $\baseCat$, the following square commutes (and similarly for $\baseCat$--limits)
\begin{center}
    \begin{tikzcd}
    \sC_W \rar["f_!"]\dar["F_W"'] & \sC_V \dar["F_V"]\\
    \D_W \rar["f_!"] & \D_V
    \end{tikzcd}
\end{center}
\end{prop}
\begin{proof}
To see the only if direction, that $F$ preserves colimits in each fibre is clear since $F$ in particular preserves fibrewise $\baseCat$--colimits. Now for $f : W\rightarrow V$, we basechange to $\underline{V}$. Since $F$ strongly preserves $\baseCat$--colimits, we get commutative squares
\begin{center}
    \begin{tikzcd}
    \underline{\func}_{\underline{V}}(\underline{W}, \underline{\sC} \times \underline{V}) \rar[shift left = 0, "f_!"] \dar["(F_{\underline{V}})_*"']& \underline{\func}_{\underline{V}}(\underline{V}, \underline{\sC} \times \underline{V}) \dar["(F_{\underline{V}})_*"]\\
    \underline{\func}_{\underline{V}}(\underline{W}, \underline{\D} \times \underline{V}) \rar[shift left = 0, "f_!"] & \underline{\func}_{\underline{V}}(\underline{V}, \underline{\D} \times \underline{V}) 
    \end{tikzcd}
\end{center}
Taking global sections by using that $\func_{\underline{V}}(\underline{W}, \underline{\sC} \times \underline{V}) \simeq \func_{\baseCat}(\underline{W}, \underline{\sC})  \simeq \sC_W$ from \cref{categoryOfPoints}, we get the desired square.

For the if direction, we know by \cref{omnibusTColimits} that all $\baseCat$--colimits can be decomposed as fibrewise $\baseCat$--colimits and indexed $\baseCat$--coproducts, and so if we show strong preservation of these we would be done. By \cref{strongPreservationObservation} (3) strong preservation of fibrewise $\baseCat$--colimits is the same as preserving colimits in each fibre, so this case is covered. Since arbitrary indexed $\baseCat$--coproducts are just compositions of orbital $\baseCat$--coproducts and arbitrary fibrewise coproducts by \cref{decompositionOfIndexedCoproducts}, we need only show for orbital $\baseCat$--coproducts, so let $f : W \rightarrow V$ be a morphism in $\baseCat$. We need to show that the canonical comparison in 
\begin{center}
    \begin{tikzcd}
    \underline{\func}_{\underline{V}}(\underline{W}, \underline{\sC}\times\underline{V}) \rar[shift left = 0, "f_!"] \dar["(F_{\underline{V}})_*"']& \underline{\func}_{\underline{V}}(\underline{V}, \underline{\sC}\times\underline{V}) \dar["(F_{\underline{V}})_*"]\\
    \underline{\func}_{\underline{V}}(\underline{W}, \underline{\D}\times\underline{V}) \rar[shift left = 0, "f_!"] & \underline{\func}_{\underline{V}}(\underline{V}, \underline{\D}\times\underline{V}) 
    \end{tikzcd}
\end{center}
 is a natural equivalence. Since equivalences is by definition a fibrewise notion, we can check this on each fibre. So let $\varphi : Y \rightarrow V$ be in $\baseCat$, and consider the pullback
 \begin{center}
     \begin{tikzcd}
     \coprod_a R_a \rar["\coprod f_a"]\ar[dr, phantom, very near start, "\scalebox{1.5}{$\lrcorner$}"] \dar["\coprod \varphi_a"']& Y\dar["\varphi"]\\
     W \rar["f"] & V
     \end{tikzcd}
 \end{center}
 by orbitality of $\baseCat$. We need to show that 
 \begin{center}
    \begin{tikzcd}
    \underline{\func}_{\underline{V}}(\underline{W}, \underline{\sC}\times\underline{V})_Y \rar[shift left = 0, "f_!"] \dar["(F_{\underline{V}})_*"']& \underline{\func}_{\underline{V}}(\underline{V}, \underline{\sC}\times\underline{V})_Y \dar["(F_{\underline{V}})_*"]\\
    \underline{\func}_{\underline{V}}(\underline{W}, \underline{\D}\times\underline{V})_Y \rar[shift left = 0, "f_!"] & \underline{\func}_{\underline{V}}(\underline{V}, \underline{\D}\times\underline{V})_Y 
    \end{tikzcd}
\end{center}
commutes. But then by the universal property of the internal functor $\baseCat$--categories from \cref{TFunctorCategory}, this is the same as 
\begin{center}
    \begin{tikzcd}
    {\func}_{\underline{Y}}(\coprod_a\underline{R_a}, \underline{\sC}\times\underline{Y})\simeq {\func}_{\underline{Y}}(\underline{W}\times_{\underline{V}}\underline{Y}, \underline{\sC}\times\underline{Y}) \rar[shift left = 0, "f_!"] \dar["(F_{\underline{Y}})_*"']& {\func}_{\underline{Y}}(\underline{Y}, \underline{\sC}\times\underline{Y}) \dar["(F_{\underline{Y}})_*"]\\
    {\func}_{\underline{Y}}(\coprod_a\underline{R_a}, \underline{\D}\times\underline{Y})\simeq{\func}_{\underline{Y}}(\underline{W}\times_{\underline{V}}\underline{Y}, \underline{\D}\times\underline{Y}) \rar[shift left = 0, "f_!"] & 
    {\func}_{\underline{Y}}(\underline{Y}, \underline{\D}\times\underline{Y}) 
    \end{tikzcd}
\end{center}
and this is in turn
\begin{center}
    \begin{tikzcd}
    \prod_a\sC_{R_a} \rar["\coprod (f_a)_!"]\dar["\prod F_a"'] & \sC_Y \dar["F_V"]\\
    \prod_a\D_{R_a} \rar["\coprod (f_a)_!"'] & \D_Y
    \end{tikzcd}
\end{center}
which commutes by hypothesis together with that $F$ commutes with fibrewise $\baseCat$--colimits (and so in particular finite fibrewise coproducts). This finishes the proof of the result.
\end{proof}

\subsection{Recollections: mapping spaces and Yoneda}

\begin{cons}[Parametrised mapping spaces and Yoneda, {\cite[Def. 10.2]{expose1Elements}}]\label{mappingAnimaYoneda}
Let $\underline{\sC}$ be a $\baseCat$--category. Then the $\baseCat$--twisted arrow construction gives us a left $\baseCat$--fibration \label{nota:parametrisedTwAr}
\[(s,t) : \underline{\twistedArrow}_T(\underline{\sC}) \longrightarrow \underline{\sC}\vop \times  \underline{\sC}\] $\baseCat$--straightening this via \cref{parametrisedStraightenUnstraighten} we get a $\baseCat$--functor \label{nota:map_T}
\[\myuline{\map}_{\underline{\sC}} : \underline{\sC}\vop \times  \underline{\sC} \longrightarrow \underline{\spc}_{\baseCat}\] By \cite[$\S5$]{barwickGlasmanNardin} we know that $\myuline{\map}_{\underline{\sC}}(-,-) : \underline{\sC}\vop \times  \underline{\sC}\rightarrow\underline{\spc}_{\baseCat}$ is given on fibre over $V$ by the map $\sC_V\op\times \sC_V \rightarrow \func((\baseCat_{/V})\op, \spc)$ \[(c,c') \mapsto \Big((W\xrightarrow{f} V)\mapsto (\map_{\sC_V}(c,c')\rightarrow \map_{\sC_W}(f^*c, f^*c')\Big)\]
Moreover, by currying we obtain the $\baseCat$--Yoneda embedding \label{nota:parametrisedPresheaves}
\[j : \underline{\sC} \longrightarrow \underline{\presheaf}_{\baseCat}(\underline{\sC})  = \underline{\func}_{\baseCat}(\underline{\sC}\vop, \underline{\spc}_{\baseCat})\] which on level $V\in \baseCat$ is given by
\begin{equation*}
    \begin{split}
        j_V : \sC_{V}\hookrightarrow \totalCategory(\underline{\sC}\vop)\times_{\baseCat\op}\totalCategory(\underline{V}) & \hookrightarrow \func_{\underline{V}}(\underline{\sC}\vop\times\underline{V}, \underline{\spc}_{\underline{V}})\\
        &\simeq \func(\totalCategory(\underline{\sC}\vop)\times_{\baseCat\op}\totalCategory(\underline{V}), \spc)
    \end{split}
\end{equation*}
\end{cons}

\begin{rmk}\label{fullyFaithfulMappingSpaces}
By the explicit fibrewise description of the parametrised mapping spaces above, we see immediately that a $\baseCat$--functor $F : \underline{\sC} \rightarrow \underline{\D}$ is $\baseCat$--fully faithful if and only if it induces equivalences on $\myuline{\map}(-,-).$
\end{rmk}

\begin{lem}[$\baseCat$--Yoneda Lemma, {\cite[Prop. 10.3]{expose1Elements}}]\label{TYonedaLemma}
Let $\underline{\sC}$ be a $\baseCat$--category and let $X \in \sC_V$ for some $V\in \baseCat$. Then for any $\baseCat_{/V}$--functor $F : \underline{\sC}\vop\times\underline{V} \longrightarrow \underline{\spc}_{\underline{V}}$, we have an equivalence of $\baseCat_{/V}$--spaces
\[F(X) \simeq \myuline{\map}_{\underline{\presheaf}_{\underline{V}}(\underline{\sC}_{\underline{V}})}\big(j_V(X), \: F\big)\] In particular, the $\baseCat$--Yoneda embedding
$j : \underline{\sC} \longrightarrow\underline{\presheaf}_{\baseCat}(\underline{\sC})$
is $\baseCat$--fully faithful.
\end{lem}
\begin{proof}
First of all note that the $V$-fibre Yoneda map above factors as 
\begin{center}
    \begin{tikzcd}
    \sC_V \rar["j_V"] \dar[hook] & \func(\totalCategory(\underline{\sC}\vop)\times_{\baseCat\op}\totalCategory(\underline{V}),\spc) \\
    (\totalCategory(\underline{\sC}\vop)\times_{\baseCat\op}\totalCategory(\underline{V}))\op \ar[ur, hook, "\widetilde{j}_V"']
    \end{tikzcd}
\end{center}
This already gives that $j_V$ is fully faithful, and so by definition of parametrised fully faithfulness, the $\baseCat$--yoneda functor $j : \underline{\sC} \rightarrow \underline{\presheaf}_{\baseCat}(\underline{\sC}) $ is $\baseCat$--fully faithful. On the other hand, by the universal property of the $\baseCat$--category of $\baseCat$--objects from \cref{cofreeParametrisations}, we can regard $F$ as an ordinary functor $F : \totalCategory(\underline{\sC}\vop)\times_{\baseCat\op}\totalCategory(\underline{V})\rightarrow \spc$. And so by ordinary Yoneda we get 
\begin{equation*}
    \begin{split}
        \map_{{\func}_{\underline{V}}(\underline{\sC}\vop\times\underline{V}, \underline{\spc}_{\underline{V}})}\big(j_V(X), \: F\big) &\simeq \map_{{\func}(\totalCategory(\underline{\sC}\vop)\times_{\baseCat\op}\totalCategory(\underline{V}), \spc)}\big(j_V(X), \: F\big) \\
        &\simeq F(X)\in \spc
    \end{split}
\end{equation*}
as required.
\end{proof}

\begin{thm}[Continuity of $\baseCat$--Yoneda, {\cite[Cor. 11.10]{shahThesis}}]\label{continuityTYoneda}
Let $\underline{\sC}\in\cat_{\baseCat}$. The $\baseCat$--yoneda embedding $j : \underline{\sC}\rightarrow \underline{\presheaf}_{\baseCat}(\underline{\sC}) $ strongly preserves and detects $\baseCat$--limits.
\end{thm}

\begin{cor}\label{mappingAdjunctionProjectionFormula}
Let $f : V \rightarrow W$ be a map in $\baseCat$. Let $B\in \sC_V$, $X\in \sC_W$, and $f_! \dashv f^* \dashv f_*$. Then 
\[\myuline{\map}_{\underline{\sC}_{\underline{W}}}(f_!B, X) \simeq f_*\myuline{\map}_{\underline{\sC}_{\underline{V}}}(B, f^*X) \in \underline{\spc}_{\underline{W}}\]
\end{cor}
\begin{proof}
Applying \cref{continuityTYoneda} on $\underline{\sC}\vop$, we see that
\[\underline{\sC}\vop \hookrightarrow \underline{\func}(\underline{\sC}, \underline{\spc}) \quad :: \quad A \mapsto \myuline{\map}_{\underline{\sC}\vop}(-, A) \simeq \myuline{\map}_{\underline{\sC}}(A,-)\] strongly preserves $\baseCat$--limits. Hence, since $f_*$ in $\underline{\sC}\vop$ is given by $f_!$ in $\underline{\sC}$, we see that 
\[\myuline{\map}_{\underline{\sC}_{\underline{W}}}(f_!B, - ) \simeq \myuline{\map}_{\underline{\sC}_{\underline{W}\vop}}(-, f_!B) \simeq f_*\myuline{\map}_{\underline{\sC}_{\underline{V}\vop}}(-, B) \simeq f_*\myuline{\map}_{\underline{\sC}_{\underline{V}}}(B,-)\] as required.
\end{proof}

\begin{thm}[$\baseCat$--Yoneda density, {\cite[Lem. 11.1]{shahThesis}}]\label{TYonedaDensity}
Let $j : \underline{\sC} \hookrightarrow \underline{\presheaf}_{\baseCat}(\underline{\sC}) $ be the $\baseCat$--yoneda embedding. Then $\id_{\underline{\presheaf}_{\baseCat}(\underline{\sC}) }\simeq j_!j$, that is, everything in the $\baseCat$--presheaf is a $\baseCat$--colimit of representables.
\end{thm}

\begin{thm}[Universal property of $\baseCat$--presheaves, {\cite[Thm. 11.5]{shahThesis}}]\label{YonedaUnivProp}
Let $\underline{\sC}, \underline{\D}\in \cat_{\baseCat}$ and suppose $\underline{\D}$ is $\baseCat$--cocomplete. Then the precompositions 
$j^* : \func_{\baseCat}^L(\underline{\presheaf}_{\baseCat}(\underline{\sC}), \underline{\D})\longrightarrow \func_{\baseCat}(\underline{\sC},\underline{\D})$ and $j^* : \underline{\func}_{\baseCat}^L(\underline{\presheaf}_{\baseCat}(\underline{\sC}), \underline{\D})\longrightarrow \underline{\func}_{\baseCat}(\underline{\sC},\underline{\D})$
are  equivalences with the inverse given by left Kan extensions. Here $\func_{\baseCat}^L$ means those functors which strongly preserve $\baseCat$--colimits (cf. \cref{LFunRFunNotations}).
\end{thm}

We learnt of the following useful procedure from Fabian Hebestreit.
\begin{defn}[Adjoint objects]
Let $R : \underline{\D} \rightarrow \underline{\sC}$ be a $\baseCat$--functor. Let $x \in \underline{\sC}$ and $y \in \underline{\D}$ and $\eta : x \rightarrow R(y)$. We say that $\eta$ \textit{witnesses $y$ as a left adjoint object to $x$ under $R$} if 
\[\myuline{\map}_{\underline{\D}}(y,-) \xrightarrow{R} \myuline{\map}_{\underline{\sC}}(Ry,R-) \xrightarrow{\eta^*} \myuline{\map}_{\underline{\sC}}(x, R-)\]
is an equivalence of $\baseCat$--functors $\underline{\D}
\rightarrow \underline{\spc}_{\baseCat}$.
\end{defn}

The following observation, due to Lurie, is quite surprising for $\infty$-categories: adjunctions can be constructed objectwise, ie. to check that we have an adjunction, it is enough to construct a left adjoint object for each object.

\begin{prop}[Pointwise construction of adjunctions] \label{pointwiseConstructionOfAdjunction}
$R : \underline{\D} \rightarrow \underline{\sC}$ admits a left adjoint $L : \underline{\sC}\rightarrow \underline{\D}$ if and only if all objects in $\underline{\sC}$ admits a left adjoint object. More generally, writing $\underline{\sC}_R$ for the full subcategory of objects admitting left adjoint objects, we obtain a $\baseCat$--functor $L : \underline{\sC}_R \rightarrow \underline{\D}$ that is $\baseCat$--left adjoint to the restriction of $R : \underline{\D}\rightarrow \underline{\sC}$ to the subcategory of $\underline{\D}$ landing in $\underline{\sC}_R$.
\end{prop}
\begin{proof}
The trick is to use the $\baseCat$--Yoneda lemma to help us assemble the various left adjoint objects into a coherent $\baseCat$--functor. We consider $\myuline{\map}_{\underline{\sC}}(-,R-) : \underline{\sC}\vop \times  \underline{\D} \rightarrow 
\underline{\spc}_{\baseCat}$ as a $\baseCat$--functor $H : \underline{\sC}\vop \rightarrow \underline{\func}_{\baseCat}(\underline{\D},\underline{\spc}_{\baseCat})$. Hence by definition of $\underline{\sC}_R$, the bottom left composition lands in the essential image of the Yoneda embedding and so we obtain a lift $L\vop$ in the commuting square 
\begin{center}
    \begin{tikzcd}
    \underline{\sC}_R\vop \rar["L\vop"] \dar[hook] & \underline{\D}\vop\dar["y",hook]\\
    \underline{\sC}\vop\rar["H"] & \underline{\func}_{\baseCat}(\underline{\D},\underline{\spc}_{\baseCat})
    \end{tikzcd}
\end{center}
To see that when $\underline{\sC}_R = \underline{\sC}$, we get a $\baseCat$--left adjoint, note that by construction $y\circ L\vop \simeq H$ in $\underline{\func}_{\baseCat}(\underline{\sC}\vop,\func(\underline{\D},\underline{\spc}_{\baseCat}))$, and hence $\myuline{\map}_{\underline{\D}}(L-,-) \simeq \myuline{\map}_{\underline{\sC}}(-,R-)$ in $\underline{\func}_{\baseCat}(\underline{\sC}\vop \times \underline{\D},\underline{\spc}_{\baseCat})$. By the characterisation of $\baseCat$--adjunctions from \cref{paramAdjunctionMappingAnima}, we are done.
\end{proof}

\begin{nota}\label{LFunRFunNotations}
We write $\underline{\rfunc}_{\baseCat}$ (resp. $\underline{\lfunc}_{\baseCat}$) for the $\baseCat$--full subcategories in $\underline{\func}_{\baseCat}$ of $\baseCat$--right adjoint functors (resp. $\baseCat$--left adjoint functors). This is distinguished from the notations $\underline{\func}^R_{\baseCat}$ (resp.  $\underline{\func}^L_{\baseCat}$) by which we mean the $\baseCat$--full subcategories in $\underline{\func}_{\baseCat}$ of strongly $\baseCat$--limit-preserving functors (resp. strongly $\baseCat$--colimit-preserving functors).
\end{nota}

\begin{prop}[``{\cite[Prop. 5.2.6.2]{lurieHTT}}'']\label{HTT5.2.6.2}
Let $\underline{\sC},\underline{\D}\in \cat_{\baseCat}$. Then there is a canonical equivalence $\underline{\lfunc}_{\baseCat}(\underline{\D},\underline{\sC}) \simeq \underline{\rfunc}_{\baseCat}(\underline{\sC},\underline{\D})\vop$.
\end{prop}
\begin{proof}
Let $j : \underline{\D} \hookrightarrow \underline{\presheaf}_{\baseCat}(\underline{\D})$ be the $\baseCat$--Yoneda embedding. Then the $\baseCat$--functor 
\[j_* : \underline{\func}_{\baseCat}(\underline{\sC}, \underline{\D}) \hookrightarrow \underline{\func}_{\baseCat}(\underline{\sC}, \underline{\presheaf}_{\baseCat}(\underline{\D})) \simeq \underline{\func}_{\baseCat}(\underline{\sC}\times\underline{\D}\vop, \underline{\spc}_{\baseCat})\]
which is $\baseCat$--fully faithful by \cref{functorsPreserveTFullyFaithfulness} has essential image consisting of those parametrised functors $\varphi : \underline{\sC} \times \underline{\D}\vop\rightarrow \underline{\spc}_{\baseCat}$ such that for all $c \in \underline{\sC}$, $\varphi(c,-) : \underline{\D}\vop \rightarrow \underline{\spc}_{\baseCat}$ is representable. The essential image under $j_*$ of $\underline{\rfunc}_{\baseCat}(\underline{\sC},\underline{\D}) \subseteq \underline{\func}_{\baseCat}(\underline{\sC}, \underline{\D})$ will then be those parametrised functors as above which moreover satisfy that for all $d\in \underline{\D}$, $\varphi(-,d) : \underline{\sC}\rightarrow\underline{\spc}_{\baseCat}$ is corepresentable - this is since $\baseCat$--adjunctions can be constructed objectwise by \cref{pointwiseConstructionOfAdjunction}. Let $\underline{\E} \subseteq \underline{\func}_{\baseCat}(\underline{\sC}\times \underline{\D}\vop, \underline{\spc}_{\baseCat})$ be the $\baseCat$--full subcategory spanned by those functors satisfying these two properties, so that $\underline{\rfunc}_{\baseCat}(\underline{\sC}, \underline{\D}) \xrightarrow{\simeq} \E$.

On the other hand, repeating the above for $\underline{\func}_{\baseCat}(\underline{\D}\vop, \underline{\sC}\vop)$ gives
\[\underline{\func}_{\baseCat}(\underline{\D}\vop, \underline{\sC}\vop) \hookrightarrow \underline{\func}_{\baseCat}(\underline{\D}\vop \times \underline{\sC}, \underline{\spc}_{\baseCat})\] where the essential image of $\underline{\rfunc}_{\baseCat}(\underline{\D}\vop, \underline{\sC}\vop)$ will be precisely those that satisfy the two properties, and so also $\underline{\rfunc}_{\baseCat}(\underline{\D}\vop, \underline{\sC}\vop) \xrightarrow{\simeq} \underline{\E}$. Thus, combining with $\underline{\rfunc}_{\baseCat}(\underline{\D}\vop, \underline{\sC}\vop)\simeq \underline{\lfunc}_{\baseCat}(\underline{\D}, \underline{\sC}) \vop$ from \cref{opOfFunctorCats}, we obtain the desired result.
\end{proof}

\subsection{(Full) faithfulness}
In this subsection we provide the parametrised analogue of the Lurie-Thomason formula for limits in categories, \cref{lurieThomasonFormula}, as well as show that parametrised functor categories preserve (fully) faithfulness in \cref{functorsPreserveTFullyFaithfulness}. 

\begin{nota}
For $p : \underline{\sC} \rightarrow \underline{I}$ a $\baseCat$--functor which is also a cocartesian fibration, we will write $\underline{\Gamma}_{\baseCat}\cocartesian(p)$ for the $\baseCat$--category of cocartesian sections of $p$. In other words, it is the  $\baseCat$--category $\underline{\func}\cocartesian_T(\underline{I}, \underline{\sC}) \times_{\underline{\func}_{\baseCat}(\underline{I},\underline{I})}\terminalTCat $ where $\underline{\func}\cocartesian_T(\underline{I}, \underline{\sC}) $ means the full $\baseCat$--subcategory of those that parametrised functors that preserve cocartesian morphisms over $\underline{I}$, and the $\baseCat$--functor $\terminalTCat \rightarrow \underline{\func}_{\baseCat}(\underline{I}, \underline{I})$ is the section corresponding to the identity on $\underline{I}$.
\end{nota}

The following proof is just a parametrisation of the unparametrised proof that we learnt from \cite[Prop. I.36]{fabianWagner}.
\begin{thm}[Lurie-Thomason formula]\label{lurieThomasonFormula}
Given a $\baseCat$--diagram $F : \underline{I}\rightarrow \underline{\cat}_{\baseCat}$, we get $$\underline{\lim}_{\underline{I}}F \simeq \underline{\Gamma}_T^{\underline{I}\mathrm{-cocart}}(\unstraighten\cocartesian(F))$$ In particular, if it factors through $F : {\underline{I}} \rightarrow \underline{\spc}_{\baseCat}$, then we have $\underline{\lim}_{\underline{I}}F \simeq \underline{\Gamma}_T(\unstraighten\cocartesian(F))$.
\end{thm}
\begin{proof}
Let $d : {\underline{I}}\rightarrow \terminalTCat$ be the unique map. Since $\underline{\cat}_{\baseCat}$ has all $\baseCat$--limits, we know abstractly that we have the $\baseCat$--right adjoint
\[d^* : \underline{\cat}_{\baseCat} \rightleftarrows \underline{\func}_{\baseCat}({\underline{I}}, \underline{\cat}_{\baseCat}) : d_* =: \underline{\lim}_{\baseCat}\] so now we just need to understand the fibrewise right adjoint formula (by virtue of \cref{criteriaForTAdjunctions}). Without loss of generality, we work with global sections and we want to describe the right adjoint in  
\[d^* : \func(\baseCat\op, {\cat}) \rightleftarrows {\func}_{\baseCat}({\underline{I}}, \underline{\cat}_{\baseCat}) \simeq \func(\totalCategory({\underline{I}}), \cat) : d_*\]
We can now identify $d^*$ concretely via the straightening-unstraightening equivalence to get 
$d^* \colon \cocartesianCategory(\baseCat\op)\rightarrow \cocartesianCategory(\totalCategory({\underline{I}}))$ given by \[\underline{\sC} \mapsto \big(\pi_{\sC} : \totalCategory(\underline{\sC})\times_{\baseCat\op}\totalCategory({\underline{I}}) \rightarrow \totalCategory({\underline{I}})\big)\] Let $(p : \E_F \rightarrow \totalCategory({\underline{I}})) \coloneqq \unstraighten^{\cocartesianCategory}(F)$ be the cocartesian fibration associated to $F : \totalCategory({\underline{I}}) \rightarrow \cat$. We need to show that $\underline{\Gamma}_T^{\underline{I}\mathrm{-cocart}}(p)$  satisfies a natural equivalence
\[\map_{\cocartesianCategory(\totalCategory({\underline{I}}))}(\pi_{\sC}, p) \simeq \map_{\cocartesianCategory(\baseCat\op)}(\underline{\sC}, \underline{\Gamma}_T^{\underline{I}\mathrm{-cocart}}(p))\]
for all $\underline{\sC}\in\cocartesianCategory(\baseCat\op)$. First of all, by definition we have the pullback 
\begin{center}
\adjustbox{scale=0.85}{
    \begin{tikzcd}
    \map_{\cocartesianCategory(\baseCat\op)_{/\totalCategory({\underline{I}})}}(\pi_{\sC}, p)  \rar\dar\ar[dr, phantom, very near start, "\scalebox{1.5}{$\lrcorner$}"] & \map_{\cocartesianCategory(\baseCat\op)}(\totalCategory(\underline{\sC})\times_{\baseCat\op}\totalCategory({\underline{I}}), \E_F) \dar  \\
    \ast \rar["\pi_{\sC}"] & \map_{\cocartesianCategory(\baseCat\op)}(\totalCategory(\underline{\sC})\times_{\baseCat\op}\totalCategory({\underline{I}}), \totalCategory({\underline{I}}))
    \end{tikzcd}
    }
\end{center}
which by currying is the same as the pullback
\begin{center}
    \begin{tikzcd}
    \map_{\cocartesianCategory(\baseCat\op)_{/\totalCategory({\underline{I}})}}(\pi_{\sC}, p)  \rar\dar\ar[dr, phantom, very near start, "\scalebox{1.5}{$\lrcorner$}"] & \map_{\cocartesianCategory(\baseCat\op)}(\underline{\sC}, \underline{\func}_{\baseCat}({\underline{I}},\E_F)) \dar  \\
    \ast \rar["\id_{\underline{I}}"] & \map_{\cocartesianCategory(\baseCat\op)}(\underline{\sC}, \underline{\func}_{\baseCat}({\underline{I}},{\underline{I}}))
    \end{tikzcd}
\end{center}
Now recall that $\map_{\cocartesianCategory(\totalCategory({\underline{I}}))}(\pi_{\sC}, p) \subseteq \map_{\cocartesianCategory(\baseCat\op)_{\totalCategory({\underline{I}})}}(\pi_{\sC}, p) $ consists precisely of those components of functors over $\totalCategory({\underline{I}})$ (in the left diagram)
\begin{center}
\adjustbox{scale=0.95}{
    \begin{tikzcd}
    \totalCategory(\underline{\sC})\times_{\baseCat\op}\totalCategory({\underline{I}}) \ar[r] \ar[dr, "\pi_{\sC}"']& \E_F \ar[d, "p"] & = &\sC \ar[r] \ar[dr, "\pi_{\sC}"']& \underline{\func}_{\baseCat}(\underline{I}, \E_F) \ar[d, "p"] \\
    & \totalCategory({\underline{I}}) & && \underline{\func}_{\baseCat}(\underline{I}, \underline{I})
    \end{tikzcd}
    }
\end{center}
preserving cocartesian morphisms over $\totalCategory(\underline{I})$. Since the cocartesian morphisms in the cocartesian fibration $\pi_{\sC} : \totalCategory(\underline{\sC})\times_{\baseCat\op}\totalCategory({\underline{I}})\rightarrow \totalCategory(\underline{I})$ are precisely the morphisms of $\totalCategory(\underline{I})$ and an equivalence in $\sC$, we see that this condition corresponds in the curried version on the right to those functors landing in $\underline{\func}_{\baseCat}^{\underline{I}\mathrm{-cocart}}(\underline{I}, \E_F)$. Finally for the statement about the case of factoring over $\underline{\spc}_{\baseCat}$ recall that unstraightening brings us to left fibrations $\E_F \rightarrow \totalCategory(\underline{I})$, and since in left fibrations all morphisms are cocartesian, we need not have imposed the condition above. This shows us that we have a bijection of components
\[\pi_0\map_{\cocartesianCategory(\totalCategory(\underline{I}))}(\pi_{\sC}, p) \simeq \pi_0\map_{\cocartesianCategory(\baseCat\op)}(\underline{\sC}, \underline{\Gamma}_T^{\underline{I}\mathrm{-cocart}}(p))\]
We now need to show that this would already imply that we have an equivalence of mapping \textit{spaces}. For this, we will need to first construct a map of spaces realising the bijection above. First note that we have a map of cocartesian fibrations over $\totalCategory(\underline{I})$
\[\varepsilon : \underline{\Gamma}^{\underline{I}\mathrm{-cocart}}_T(p) \times \underline{I} \longrightarrow \E_F\] from the evaluation. Therefore we get the following maps of spaces 

\begin{equation}\label{mappingSpaceComparisonMap}
\begin{split}
     &\map_{\cocartesianCategory(\baseCat\op)}(-,\underline{\Gamma}^{\underline{I}\mathrm{-cocart}}_T(p))\\
     & \xrightarrow{\underline{I} \times -} \map_{\cocartesianCategory(\totalCategory(\underline{I}))}(\underline{I} \times -, \underline{I} \times \underline{\Gamma}^{\underline{I}\mathrm{-cocart}}_T(p))\\
     & \xrightarrow{\varepsilon_*}\map_{\cocartesianCategory(\totalCategory(\underline{I}))}(\underline{I} \times -, \E_F)
\end{split}
\end{equation}
On the other hand, we know by the pullback definition of $\underline{\Gamma}_{\baseCat}$ that 
\begin{equation}\label{mappingSpaceEquivalence}
    \map_{\cocartesianCategory(\baseCat\op)}(-,\underline{\Gamma}_T(p)) \simeq \map_{\cocartesianCategory(\baseCat\op)_{/\totalCategory(\underline{I})}}(\underline{I} \times -, \E_F)
\end{equation}
and so the comparison map \cref{mappingSpaceComparisonMap} is induced by this equivalence. Our bijection on components then gives that the equivalence \cref{mappingSpaceEquivalence} restricts to an equivalence of spaces \cref{mappingSpaceComparisonMap}. This completes the proof of the result.
\end{proof}

As far as we are aware the following proof strategy first appeared in \cite[$\S5$]{gepnerHaugsengNikolaus}.

\begin{prop}[Mapping space formula in $\baseCat$--functor categories]\label{mapingAnimaFormulaFunctorCategories}
Let $\underline{\sC}, \underline{\D} \in \cat_{\baseCat}$  and $F, G : \underline{\sC}\rightarrow \underline{\D}$ be $\baseCat$--functors. Then we have an equivalence of $\baseCat$--spaces
\[\underline{\nattrans}_T(F, G) \simeq \underline{\lim}_{(x\rightarrow y) \in \underline{\twistedArrow}_T(\underline{\sC}) }\myuline{\map}_{\underline{\D}}(F(x), G(y))\in \underline{\spc}_{\baseCat}\]
\end{prop}
\begin{proof}
Recall from \cite[$\S10$]{expose1Elements} that by definition, the parametrised mapping spaces are classified by the parametrised twisted arrow categories. By \cref{lurieThomasonFormula} we have 
\[\underline{\lim}_{(x\rightarrow y) \in \underline{\twistedArrow}(\underline{\sC}) }\myuline{\map}_{\underline{\D}}(F(x), G(y)) \simeq \underline{\Gamma}_{\baseCat}\big(\underline{P} \rightarrow \underline{\twistedArrow}_T(\underline{\sC}) \big)\] where $p : \underline{P} \rightarrow \underline{\twistedArrow}_T(\underline{\sC}) $ is the associated unstraightening. By considering the pullbacks
\begin{center}
    \begin{tikzcd}
    \underline{P}\rar\dar\ar[dr, phantom, very near start, "\scalebox{1.5}{$\lrcorner$}"] & \underline{P}' \dar\rar\ar[dr, phantom, very near start, "\scalebox{1.5}{$\lrcorner$}"] & \underline{\twistedArrow}_T(\underline{\D})\dar \\
    \underline{\twistedArrow}_T(\underline{\sC})  \rar["(s{,}t)"'] & \underline{\sC}\vop \times \underline{\sC} \rar["F\vop \times  G"'] & \underline{\D}\vop \times \underline{\D}\ar[rr,"\myuline{\map}_{\underline{\D}}(-{,}-)"'] && \underline{\spc}_{\baseCat}
    \end{tikzcd}
\end{center}
we get that 
\[\underline{\Gamma}_{\baseCat}\big(\underline{P} \rightarrow \underline{\twistedArrow}_T(\underline{\sC}) \big) \simeq \myuline{\map}_{/\underline{\sC}\vop \times \underline{\sC}}(\underline{\twistedArrow}_T(\underline{\sC}) , \underline{P}')\] Now by the parametrised straightening of \cref{parametrisedStraightenUnstraighten} we see furthermore that
\[\myuline{\map}_{/\underline{\sC}\vop\times\underline{\sC}}(\underline{\twistedArrow}_T(\underline{\sC}) , \underline{P}')\simeq \underline{\nattrans}_{\baseCat}\big(\myuline{\map}_{\underline{\sC}}, \myuline{\map}_{\underline{\D}}\circ (F\vop  \times  G)\big)\]
Currying $\underline{\func}_{\baseCat}(\underline{\sC}\vop \times \underline{\sC}, \underline{\spc}_{\baseCat}) \simeq \underline{\func}_{\baseCat}(\underline{\sC}, \underline{\presheaf}_{\baseCat}(\underline{\sC}) )$ we see that \[\underline{\nattrans}_{\baseCat}\big(\myuline{\map}_{\underline{\sC}}, \myuline{\map}_{\underline{\D}}\circ (F\vop  \times  G)\big) \simeq \underline{\nattrans}_{\baseCat}\big(y_{\underline{\sC}}, F^*\circ y_{\underline{\D}}\circ G)\big)\]
But then now we have the sequence of equivalences
\begin{equation*}
    \begin{split}
    \underline{\nattrans}_{\baseCat}\big(y_{\underline{\sC}}, F^*\circ y_{\underline{\D}}\circ G)\big) &\simeq \underline{\nattrans}_{\baseCat}\big(F_!\circ y_{\underline{\sC}}, y_{\underline{\D}}\circ G)\big) \\
    &\simeq \underline{\nattrans}_{\baseCat}\big(y_{\underline{\D}}\circ F,  y_{\underline{\D}}\circ G)\big) \\
    &\simeq \underline{\nattrans}_T(F, G)
    \end{split}
\end{equation*}
where the last equivalence is by \cref{TYonedaLemma} and the second by the square
\begin{center}
    \begin{tikzcd}
    \underline{\sC}\rar["F"]\dar["y_{\underline{\sC}}"', hook] & \underline{\D} \dar["y_{\underline{\D}}", hook]\\
    \underline{\presheaf}_{\baseCat}(\underline{\sC})  \rar[ "F_!"] & \underline{\presheaf}_{\baseCat}(\underline{\D}) 
    \end{tikzcd}
\end{center}
which commutes by functoriality of presheaves.
\end{proof}

\begin{defn}
A $\baseCat$--functor is called $\baseCat$-faithful if it is so fibrewise, where an ordinary functor is called faithful if it induces component inclusions on mapping spaces.
\end{defn}

\begin{obs}
For $f : X \rightarrow Y$ a map of spaces, it being an inclusion of components is equivalent to the condition that for each $x\in X$, the fibre $\fib_{f(x)}\big(X \rightarrow Y\big)$ is contractible. On the other hand, it is an equivalence if and only if for each $y\in Y$, the fibre $\fib_y\big(f : X\rightarrow Y\big)$ is contractible. We learnt of this formulation and of the following proof in the unparametrised case from \cite[Appendix B]{malteThesis}.
\end{obs}

\begin{cor} \label{functorsPreserveTFullyFaithfulness}
Let $F : \underline{\sC} \rightarrow \underline{\D}$ be a $\baseCat$--(fully) faithful functor and $\underline{I}$ another $\baseCat$--category. Then $F_* : \underline{\func}_{\baseCat}(\underline{I}, \underline{\sC})  \rightarrow \underline{\func}_{\baseCat}(\underline{I}, \underline{\D})$ is again $\baseCat$--(fully) faithful.
\end{cor}
\begin{proof}
Since $\baseCat$--(fully) faithfulness was defined as a fibrewise condition, we just assume without loss of generality that $\baseCat$ has a final object and work on global sections. In the faithful case, let $\varphi, \psi : \underline{I}\rightarrow \underline{\sC}$ be two $\baseCat$--functors. We need to show that 
\[\underline{\nattrans}_{\underline{\func}_{\baseCat}(\underline{I},\underline{\sC}) }(\varphi, \psi) \longrightarrow \underline{\nattrans}_{\underline{\func}_{\baseCat}(\underline{I},\underline{\D})}(F\varphi, F\psi)\] is an inclusion of components. By the preceeding observation, we need to show that for each $\eta \in \nattrans_{\func_{\baseCat}(\underline{I},\underline{\sC}) }(\varphi, \psi)$, the fibre
\[\underline{\fib}_{\eta}\big(\underline{\nattrans}_{\underline{\func}_{\baseCat}(\underline{I},\underline{\sC}) }(\varphi, \psi) \rightarrow \underline{\nattrans}_{\underline{\func}_{\baseCat}(\underline{I},\underline{\D})}(F\varphi, F\psi)\big)\in \Gamma(\underline{\spc}_{\baseCat}\rightarrow \baseCat\op)\] is contractible. But then we are now in position to use \cref{mapingAnimaFormulaFunctorCategories}:
\begin{equation*}
    \begin{split}
         &\underline{\fib}_{\eta}\Big(\underline{\nattrans}_{\underline{\func}_{\baseCat}(\underline{I},\underline{\sC}) }(\varphi, \psi) \rightarrow \underline{\nattrans}_{\underline{\func}_{\baseCat}(\underline{I}, \underline{\D})}(F\varphi, F\psi)\Big) \\
        &\simeq \underline{\lim}_{(x\rightarrow y) \in \underline{\twistedArrow}_T(\underline{I})}\underline{\fib}_{\eta}\Big(\myuline{\map}_{\underline{\D}}(\varphi(x), \psi(y)) \rightarrow \myuline{\map}_{\underline{\D}}(F\varphi(x), G\psi(y))\Big)\\
        &\simeq \underline{\lim}_{(x\rightarrow y) \in \underline{\twistedArrow}_T(\underline{I})}\ast_{\baseCat}\simeq \ast_T
    \end{split}
\end{equation*}
as was to be shown, where the second last step is by our hypothesis that $F$ was $\baseCat$--faithful. The case of $\baseCat$--fully faithfulness can be done similarly.
\end{proof}

\subsection{Recollections: filtered colimits and Ind-completions}

\begin{cons}\label{nota:parametrisedInd}
Let $\kappa$ be a regular cardinal. We define the $\baseCat$\textit{-Ind-completion} functor $\underline{\ind}_{\kappa} \colon \cat_{\baseCat}\rightarrow \cat_{\baseCat}$ to be the one obtained by applying $\func(\baseCat\op,-)$ to the ordinary functor $\ind_{\kappa} \colon \cat\rightarrow \cat$.
\end{cons}

\begin{nota}
We will write $\underline{\func}_{\baseCat}\filtered$ for the full $\baseCat$--subcategory of parametrised functors preserving fibrewise $\omega$--filtered colimits, and similarly $\underline{\func}_{\baseCat}\filteredKappa$ for those that preserve fibrewise $\kappa$-filtered colimits.
\end{nota}

\begin{rmk}
This agrees with the definition given in the recent paper \cite{shahPaperII} by virtue of the paragraph after Theorem D therein. As indicated there, $\underline{\ind}_{\kappa}(\underline{\sC}) $ is the minimal $\baseCat$--subcategory of $\underline{\presheaf}_{\baseCat}(\underline{\sC}) $ generated by $\underline{\sC}$ under fibrewise $\kappa$-filtered colimits. In more detail, \cite[Rmk. 9.4]{shahPaperII} showed that the fibrewise presheaf construction $\underline{\presheaf}_{\baseCat}^{\text{fb}}(\underline{\sC}) $ is a $\baseCat$--full subcategory of the $\baseCat$--presheaf $\underline{\presheaf}_{\baseCat}(\underline{\sC}) $ via the fibrewise left Kan extension. In particular, this means that $\underline{\presheaf}_{\baseCat}^{\text{fb}}(\underline{\sC}) \subseteq \underline{\presheaf}_{\baseCat}(\underline{\sC}) $ preserves fibrewise colimits. On the other hand, by construction and \cite[Cor. 5.3.5.4]{lurieHTT}, $\underline{\ind}_{\kappa}(\underline{\sC})  \subseteq \underline{\presheaf}_{\baseCat}^{\text{fb}}(\underline{\sC}) $ is the minimal $\baseCat$--subcategory generated by $\underline{\sC}$ under fibrewise $\kappa$-filtered colimits. Therefore, in total, we see that $\underline{\ind}_{\kappa}(\underline{\sC})  \subseteq \underline{\presheaf}_{\baseCat}(\underline{\sC}) $ is the $\baseCat$--subcategory generated by $\sC$ under fibrewise $\kappa$-filtered colimits.
\end{rmk}

\begin{prop}[Universal property of $\underline{\ind}$, ``{\cite[Prop. 5.3.5.10]{lurieHTT}}'']\label{UnivPropInd}
Let $\underline{\sC}, \underline{\D}$ be $\baseCat$--categories where $\underline{\sC}$ is small and $\underline{\D}$ has fibrewise small $\kappa$-filtered colimits. Then:
\begin{enumerate}
    \item[(1)] $\underline{\ind}_{\kappa}(\underline{\sC})  \subseteq\underline{\presheaf}_{\baseCat}(\underline{\sC}) $ is the $\baseCat$--subcategory generated by $\sC$ under fibrewise $\kappa$-filtered colimits.
    
    \item[(2)] The $\baseCat$--inclusion $i : \underline{\sC} \hookrightarrow \underline{\ind}_{\kappa}(\underline{\sC}) $ induces an equivalence 
\[i^* : \underline{\func}_{\baseCat}\filteredKappa(\underline{\ind}_{\kappa}(\underline{\sC}) , \underline{\D}) \longrightarrow \underline{\func}_{\baseCat}(\underline{\sC},\underline{\D})\]
\end{enumerate}

\end{prop}
\begin{proof}
Part (1) is by the remark above. For part (2), we show that the $\baseCat$--left Kan extension functor $i_! : \underline{\func}_{\baseCat}(\underline{\sC},\underline{\D}) \longrightarrow \underline{\func}_{\baseCat}\filteredKappa(\underline{\ind}_{\kappa}(\underline{\sC}) , \underline{\D})$ exists and is an inverse to $i^*$. To do this, it will be enough to show that functors $F : \underline{\sC}\rightarrow \underline{\D}$ can be $\baseCat$--left Kan extended to $i_!F : \underline{\ind}_{\kappa}(\underline{\sC})  \rightarrow \underline{\D}$ and that functors $F : \underline{\ind}_{\kappa}(\underline{\sC})  \rightarrow \underline{\D}$ which preserves fibrewise $\kappa$-filtered colimits satisfy that $i_!i^*F \Rightarrow F$ is an equivalence. This will be enough since we would have shown the natural equivalence $i_!i^* \simeq \id$, and \cref{fullyFaithfulTKanExtension} gives that $i^*i_! \simeq \id$ always.

To show that the $\baseCat$--left Kan extension exists, consider the diagram 
    \begin{center}
        \begin{tikzcd}
        \underline{\sC} \dar[hook] \ar[dr, "f"]\\
        \underline{\ind}_{\kappa}(\underline{\sC}) \ar[r,dashed] \dar[hook] & \underline{\D} \ar[dr,hook]\\
        \underline{\presheaf}_{\baseCat}(\underline{\sC})  \ar[rr, dashed] && \underline{\D}'
        \end{tikzcd}
    \end{center}
    where $\underline{\D} \subseteq \underline{\D}'$ is a strongly $\baseCat$--colimit preserving inclusion into a $\baseCat$--cocomplete $\underline{\D}'$ using the opposite $\baseCat$--Yoneda embedding. In particular by hypothesis $\underline{\D}$ is closed under $\kappa$-filtered colimits in $\underline{\D}'$. The bottom dashed map is gotten from \cref{YonedaUnivProp}, and so strongly preserves $\baseCat$--colimits. Hence restriction to $\underline{\ind}_{\kappa}(\underline{\sC}) $ lands in $\underline{\D}$ so we get middle dashed map, and by the following \cref{restrictedLeftKanExtension}, this is a left Kan extension.
    
Now we show that if $F$ preserves fibrewise $\kappa$-filtered colimits, then the canonical comparison $i_!i^*F \Rightarrow F$ is an equivalence. Again, by \cref{fullyFaithfulTKanExtension} we know that both sides agree on $\underline{\sC} \subseteq \underline{\ind}_{\kappa}(\underline{\sC}) $. Also, both sides preserve $\kappa$-filtered colimits by assumption. Hence, by statement (1) of the proposition, we see that it must be an equivalence as was to be shown.
\end{proof}

\begin{lem}\label{restrictedLeftKanExtension}
Suppose we have fully faithful functors $\underline{\sC} \xhookrightarrow{i} \underline{\D} \xhookrightarrow{j} \underline{\E}$ and functors $\underline{\sC} \xrightarrow{f} \underline{\A} \xhookrightarrow{y} \underline{\B}$, where $\underline{B}$ is $\baseCat$--(co)complete. Suppose we have a factorisation $j^*j_!i_!(y\circ f) \colon \underline{\sC} \xrightarrow{\overline{f}} \underline{\A} \xhookrightarrow{y} \underline{\B}$. Then  $\overline{f} \simeq i_!f \colon \underline{\D} \rightarrow\underline{\A}$.
\end{lem}
\begin{proof}
Let $\varphi \colon \underline{\D} \rightarrow \underline{\A}$. We need to show that $\underline{\nattrans}(\overline{f}, \varphi) \simeq \underline{\nattrans}(f, i^*\varphi)$. We compute:
\begin{equation*}
    \begin{split}
        \underline{\nattrans}(\overline{f}, \varphi) & \simeq \underline{\nattrans}(y\circ\overline{f}, y\circ\varphi)\\
        & = \underline{\nattrans}(j^*j_!i_!(y\circ f), y\circ\varphi)\\
        &\simeq \underline{\nattrans}(y\circ f, i^*j^*j_*(y\circ \varphi))\\
        &\simeq \underline{\nattrans}(y\circ f, i^*(y\circ \varphi) \simeq \underline{\nattrans}(f, i^*\varphi)
    \end{split}
\end{equation*}
where the first and last equivalences are since $y$ was fully faithful; the fourth equivalence is since $j$ was fully faithful and so \cref{fullyFaithfulTKanExtension} applies. The relevant Kan extensions exist since $\underline{\B}$ was assumed to be $\baseCat$--(co)complete.
\end{proof}

We learnt of the following proof method from Markus Land.

\begin{lem}\label{2FunctorialityOfInd}
For $\sC, \D\in\cat$, we have a functor 
$\func(\sC, \D) \rightarrow \func(\ind(\sC) , \ind(\D))$ that takes $F : \sC \rightarrow \D$ to $\ind(F) : \ind(\sC)  \rightarrow \ind(\D)$.
\end{lem}
\begin{proof}
We know that $\ind(\E\times \sC)  \simeq \ind(\E) \times \ind(\sC) $.
In particular, we get functors  
\[\Delta^n \times \ind(\sC)  \longrightarrow \ind(\Delta^n\times \sC)  \simeq \ind(\Delta^n) \times \ind(\sC)  \] natural in both $\Delta^n$ and $\sC$. These then induce a map of simplicial spaces \small
\[\func(\Delta^{\bullet}\times \sC, \D)^{\simeq} \longrightarrow \func(\ind(\Delta^{\bullet}\times \sC) , \ind(\D))^{\simeq} \longrightarrow \func(\Delta^{\bullet}\times \ind(\sC) , \ind(\D))^{\simeq} \]\normalsize where the first map is just by the $(\infty,1)$--functoriality of $\ind$. Via the complete Segal space model of $\infty$-categories, we see that we have the desired functor  which behaves as in the statement by looking at the case $\bullet = 0$.
\end{proof}

\begin{lem}[Ind adjunctions]\label{indAdjunctions}
Let $f : \sC \rightleftarrows \D : g$ be an adjunction. Then we also have an adjunction 
$F \coloneqq \ind(f) : \ind(\sC)  \rightleftarrows \ind(\D) : \ind(g) =: G$.
\end{lem}
\begin{proof}
By \cite[Def. 1.1.2]{riehlVerityYTM} we know that such an adjunction is tantamount to the data of $\eta : \id_{\sC} \Rightarrow gf$ and $\varepsilon : fg \Rightarrow \id_{\D}$ such that we have the triangle identities
\begin{center}
    \begin{tikzcd}
    & \sC \ar[rr,equal]\ar[dr, "f"]\ar[d, Rightarrow, "\varepsilon"]&{}\dar[Rightarrow, "\eta"]& \sC  && = && \sC\\
    \D \ar[ur, "g"]\ar[rr,equal]&{}& \D\ar[ur, "g"'] &&& && \D\uar[bend left = 60, "g"]\uar[bend right=60, "g"']\uar[phantom, "{}^{\Rightarrow}_{\id_g}"]
    \end{tikzcd}
\end{center}
and the analogous other triangle. Now, we have
$\func(\sC,\sC)  \rightarrow \func(\ind(\sC) ,\ind(\sC) )$ by \cref{2FunctorialityOfInd} and so the the triangle identity on the source gets sent to a triangle identity on the target.
\end{proof}

\begin{thm}[Diagram decomposition, {\cite[Thm. 8.1]{shahPaperII}}]\label{filteredDiagramDecomposition}
Let $\underline{\sC}$ be a $\baseCat$--category, $J$ a category, and $p_{\bullet} : J \rightarrow ({\cat}_{\baseCat})_{/\underline{\sC}}$ a  functor with colimit the $\baseCat$--functor $p : \underline{K}\rightarrow \underline{\sC}$ and suppose that for all $j \in J$, the $\baseCat$--functor $p_j : \underline{K}_j \rightarrow \underline{\sC}$ admits a $\baseCat$--colimit $\sigma_j$. Then the $\sigma_j$'s assemble to a $\baseCat$--functor $\sigma_{\bullet} : \tconstant_\baseCat(J)\rightarrow \underline{\sC}$ so that if $\sigma_{\bullet}$ admits a $\baseCat$--colimit $\sigma$, then $p$ admits a $\baseCat$--colimit given by $\sigma$.
\end{thm}

\begin{cor}[Parametrised filtered colimit decomposition, ``{\cite[Cor. 4.2.3.11]{lurieHTT}}'']\label{parametrisedFilteredColimitDecomposition}
Let $\tau \ll \kappa$ be regular cardinals and $\underline{\sC}$ be a $\baseCat$--category admitting $\tau$-small $\baseCat$--colimits and fibrewise colimits indexed by $\kappa$-small $\tau$-filtered posets. Then for any $\kappa$-small $\baseCat$--diagram $d : \underline{K} \rightarrow \underline{\sC}$, its $\baseCat$--colimit in $\underline{\sC}$ exists and can be decomposed as a fibrewise $\kappa$-small $\tau$-filtered colimit whose vertices are $\tau$-small $\baseCat$--colimits of $\underline{\sC}$. 
\end{cor}
\begin{proof}
Let $J$ denote the poset of $\tau$-small $\baseCat$--subcategories of $\underline{K}$. It is clearly $\tau$-filtered and moreover it is $\kappa$-small by the hypothesis that $\tau\ll\kappa$.  We can therefore apply the theorem above since the associated $\sigma_{\bullet} : \tconstant_\baseCat(J)\rightarrow \underline{\sC}$ will admit a $\baseCat$--colimit by hypothesis.
\end{proof}

\begin{thm}[Limit-filtered colimit exchange, special case of {\cite[Thm. C]{shahPaperII}}]\label{CommutationRelativeFiltColimAndFiniteLim}
Let $\kappa$ be a regular cardinal and $J$ a $\kappa$--filtered category. Then
$\underline{\colim}_{\tconstant_\baseCat(J)}\colon  \underline{\func}(\tconstant_{\baseCat}(J), \myuline{\spc}_{\baseCat}) \longrightarrow \myuline{\spc}_{\baseCat}$ strongly preserves $\baseCat$--$\kappa$--small $\baseCat$--colimits. 
\end{thm}

\section{Preliminaries: atomic orbital base categories}\label{section:PrelimsAtomicOrbitalBaseCategories}

Finally, we begin to impose the strictest conditions on our base category $\baseCat$. From here on, $\baseCat$ will be assumed to be both orbital \textit{and} atomic.

\subsection{Recollections: parametrised semiadditivity and stability}\label{subsec:SemiadditivityStability}
In this subsection we recall the algebraic constructions and results of \cite{nardinExposeIV,nardinThesis}.

\begin{cons} \label{nardinSpanConstructions}
The following list of constructions will be important in discussing $\baseCat$--semiadditivity and $\baseCat$--stability. See \cite[\S4]{nardinExposeIV} for the original source on these constructions or \cite[Def. 2.1.2]{nardinShah} for a more recent treatment. Note that we have adopted the notation of $\effBurn$ instead of the original notation of effective Burnside categories $A^{\mathrm{eff}}$.
\begin{enumerate}
    \item[(1)] Write $\effBurn(\baseCat) \coloneqq \effBurn(\finite_{\baseCat})$.
    
    \item[(2)] By \cite[{Cons. 4.8}]{nardinThesis}, there is a $\baseCat$--category $p : \myuline{\effBurn}(T)\rightarrow \baseCat\op$ whose objects are morphisms $[U \rightarrow V]$ in $\finite_{\baseCat}$  where $V\in \baseCat$ and the cocartesian fibration $p$ sends $[U \rightarrow V]$ to $V$. The morphisms in this category are spans
    \begin{equation}\label{fin_TSpans}
    \begin{tikzcd}
    U\dar & W \lar \rar\dar & U'\dar \\
    V & V'\lar\rar[equal] & V'
    \end{tikzcd}
\end{equation}

\item[(3)] From this we obtain a wide $\baseCat$--subcategory $\underline{\finite}_{*\baseCat} \subset \myuline{\effBurn}(T)$ whose morphisms are spans as in \cref{fin_TSpans} such that the map $W \rightarrow U\times_VV'$ in $\finite_{\baseCat}$ is a summand inclusion: this makes sense since $\baseCat$ was assumed to be orbital and so $\finite_{\baseCat}$ admits the pullback $U\times_VV'$ which will be a finite coproduct of objects of $V$. \label{nota:fin_*_T}

\item[(4)] There is a canonical inclusion  $\terminalTCat \hookrightarrow \underline{\finite}_{*\baseCat}$ given by sending $W \rightarrow V$ to 
\begin{center}
    \begin{tikzcd}
        V\dar[equal] & W \lar \rar[equal]\dar[equal] & W\dar[equal] \\
    V & W\lar\rar[equal] & W
    \end{tikzcd}
\end{center}
\end{enumerate}
\end{cons}

\begin{defn}\label{defn:funSadd}\label{defn:Lin_T}
Let $\underline{\sC}$ strongly admit finite $\baseCat$--coproducts and $\underline{\D}$ strongly admit finite $\baseCat$--products. Then we say that a $\baseCat$--functor $F : \underline{\sC} \rightarrow \underline{\D}$ is $\baseCat$\textit{-semiadditive} if it sends finite $\baseCat$--coproducts to finite $\baseCat$--products. We say that a $\baseCat$--category $\underline{\sC}$ strongly admitting finite $\baseCat$--products and $\baseCat$--coproducts is \textit{$\baseCat$--semiadditive} if the identity functor is $\baseCat$--semiadditive. If moreover $\underline{\sC}$ has fibrewise pushouts and $\underline{\D}$ has fibrewise pullbacks, then we say that $F$ is $\baseCat$\textit{--linear} if it is $\baseCat$--semiadditive and sends fibrewise pushouts to fibrewise pullbacks. We write $\underline{\func} \Tsemiadd_T(\underline{\sC},  \underline{\D})$ (resp. $\underline{\linear}_{\baseCat}(\underline{\sC}, \underline{\D})$) for the $\baseCat$--full subcategories of $\underline{\func}_{\baseCat}(\underline{\sC}, \underline{\D})$ consisting of the $\baseCat$--semiadditive functors (resp. $\baseCat$--linear functors).
\end{defn}

\begin{nota}\label{nota:Mackey_T}\label{nota:CMon_T}
For $\underline{\sC}$ strongly admitting finite $\baseCat$--limits we will denote \textit{$\baseCat$--Mackey functors} by $\underline{\mackey}_T(\underline{\sC}) \coloneqq \underline{\func} \Tsemiadd_T(\myuline{\effBurn}(T),\underline{\sC})$ and \textit{$\baseCat$--commutative monoids} by $\underline{\cmonoid}_{\baseCat}(\underline{\sC}) \coloneqq \underline{\func} \Tsemiadd_T(\underline{\finite}_{*\baseCat},\underline{\sC})$.
\end{nota}

\begin{prop}[$\baseCat$--semiadditivisation, {\cite[{Prop. 5.11}]{nardinExposeIV}}] \label{semiadditivisation}
Let $\underline{\sC}$ be a $\baseCat$--category strongly admitting finite $\baseCat$--products. Then the functor $\underline{\cmonoid}_{\baseCat}(\underline{\sC}) \rightarrow \underline{\sC}$ induced by the inclusion $\terminalTCat \hookrightarrow \underline{\finite}_{*\baseCat}$ from \cref{nardinSpanConstructions} (4) is an equivalence if and only if $\underline{\sC}$ were $\baseCat$--semiadditive.
\end{prop}

\begin{thm}[``CMon = Mackey'', {\cite[{Thm. 6.5}]{nardinExposeIV}}]\label{CMon=Mackey}
Let $\underline{\sC}$ strongly admit finite $\baseCat$--limits. Then the defining inclusion $j : \underline{\finite}_{*\baseCat} \hookrightarrow \myuline{\effBurn}(T)$ induces an equivalence 
\[j^* : \underline{\func} \Tsemiadd_T(\myuline{\effBurn}(T),\underline{\sC}) \longrightarrow \underline{\cmonoid}_{\baseCat}(\underline{\sC})\] 
\end{thm}

\begin{nota}\label{nota:funLexRexEx}
We write $\func_{\baseCat}\Texact, \func_{\baseCat}\Tlexact,$ and $\func_{\baseCat}\Trexact$ for the category of $\baseCat$--functors which strongly preserve finite $\baseCat$--(co)limits, strongly preserve finite $\baseCat$--limits, and strongly preserve finite $\baseCat$--colimits, respectively.
\end{nota}

\begin{cons}\label{nota:TCategoryOfTSpectra}
Let $\myuline{\spectra}\pointwise \colon \cat\Tlexact_T \rightarrow \cat\Tlexact_{\baseCat}$ be the functor obtained by applying $\func(\baseCat\op,-)$ to $\spectra \colon \cat\lexact \rightarrow \cat\lexact$. Now let $\underline{\D}\in\cat_{\baseCat}$ strongly admitting finite $\baseCat$--limits. Then we can define its $\baseCat$\textit{--stabilisation} to be $\myuline{\spectra}_{\baseCat}(\underline{\D}) \coloneqq \underline{\cmonoid}_{\baseCat}(\myuline{\spectra}\pointwise(\underline{\D}))$. In particular, applying this to the case $\underline{\D} = \underline{\spc}_{\baseCat}$, we get $\myuline{\spectra}_{\baseCat} \coloneqq \underline{\cmonoid}_{\baseCat}(\myuline{\spectra}\pointwise(\underline{\spc}_{\baseCat}))$ which is called $\baseCat$\textit{--category of genuine $\baseCat$--spectra}. Note that this is different from the notation in \cite[Defn. 7.3]{nardinExposeIV} where he used $\myuline{\spectra}^{\baseCat}$ instead, and reserved $\myuline{\spectra}_{\baseCat}$ for what we wrote as $\myuline{\spectra}\pointwise$. We prefer the notation we have adopted as it aligns well with all the parametrised subscripts $(-)_{\baseCat}$ and the superscripts are reserved for modifiers such as $(-)\tomega$ or $(-)^{\Delta^1}$ that we will need later.
\end{cons}

\begin{thm}[Universal property of $\baseCat$--stabilisations, {\cite[{Thm. 7.4}]{nardinExposeIV}}] \label{univPropTStabilisations}
Let $\underline{\sC}$ be a pointed $\baseCat$--category strongly admitting finite $\baseCat$--colimits and $\underline{\D}$ a $\baseCat$--category strongly admitting finite $\baseCat$--limits. Then the functor $\underline{\Omega}^{\infty} : \underline{\func}_{\baseCat}\Trexact(\underline{\sC}, \myuline{\spectra}_{\baseCat}(\underline{\D}))\longrightarrow \underline{\linear}_{\baseCat}(\underline{\sC}, \underline{\D})$ is an equivalence of $\baseCat$--categories. In particular, we see that 
$\myuline{\spectra}_{\baseCat}(\underline{\D}) \simeq \underline{\linear}_{\baseCat}(\underline{\spc}_{*\baseCat}^{\underline{\mathrm{fin}}}, \underline{\D})$.
\end{thm}


\subsection{Parametrised symmetric monoidality and commutative algebras}
\label{subsection:SymmetricMonoidalityYoneda}

\begin{recollect}\label{nota:FunTInert}\label{nota:calg_T}
There is a notion of $\baseCat$--operads mimicking the notion of $\infty$-operads, in the sense of \cite[{$\S2.1$}]{lurieHA}, due to Nardin in \cite[{$\S3$}]{nardinThesis} and further developed in \cite[\S2]{nardinShah}. A $\baseCat$--symmetric monoidal category is then a $\baseCat$--category $\sC^ {\underline{\otimes}}$ equipped with a cocartesian fibration over $\underline{\finite}_{*\baseCat}$ satisfying the $\baseCat$--operad axioms analogous to the operad axioms of \cite[{Definition 2.1.1.10}]{lurieHA}. Alternatively, the $\baseCat$--category of $\baseCat$--symmetric monoidal categories is also given as $\underline{\cmonoid}(\underline{\cat})$ much like in the unparametrised setting. Furthermore, there is also the attendant notion of $\baseCat$\textit{-inert morphisms} defined as those morphisms in $\underline{\finite}_{*\baseCat}$ where the the map $W \rightarrow U'$ is an equivalence (cf. the span notation in \cref{fin_TSpans}). The $\baseCat$--category of $\baseCat$\textit{-commutative algebras} $\myuline{\calg}_{\baseCat}(\underline{\sC}\totimes)$ of a $\baseCat$--symmetric monoidal category $\underline{\sC}^ {\underline{\otimes}}$ is then defined to be $\underline{\func}_{\underline{\finite}_{*\baseCat}}^{\mathrm{\underline{\mathrm{inert}}}}(\underline{\finite}_{*\baseCat}, \sC^ {\underline{\otimes}})$ where $\underline{\func}_{\underline{\finite}_{*\baseCat}}^{\mathrm{\underline{\mathrm{inert}}}}\subseteq \underline{\func}_{\underline{\finite}_{*\baseCat}}$ is the $\baseCat$--full subcategory of functors over $\underline{\finite}_{*\baseCat}$ preserving $\baseCat$--inert morphisms. We refer the reader to the original source \cite[{$\S3.1$}]{nardinThesis} or to \cite[\S2]{nardinShah} for details on this.
\end{recollect}

\begin{terminology}\label{term:TSymmetricMonoidalLocalisation}
Let $\underline{\sC}^ {\underline{\otimes}}, \underline{\D}^{\underline{\otimes}}$ be $\baseCat$--symmetric monoidal categories. By a $\baseCat$\textit{-symmetric monoidal localisation} $L^{\underline{\otimes}} : \underline{\sC}^ {\underline{\otimes}}\rightarrow \underline{\D}^{\underline{\otimes}}$ we mean a $\baseCat$--symmetric monoidal functor whose underlying $\baseCat$--functor is a $\baseCat$--Bousfield localisation. By the proof of \cite[{Prop. 3.5}]{nardinThesis}, we see that the $\baseCat$--right adjoint canonically refines to a $\baseCat$--lax symmetric functor. Hence in this situation we obtain a relative adjunction over $\underline{\finite}_{*\baseCat}$ 
\begin{center}
    \begin{tikzcd}
        \underline{\sC}^{\underline{\otimes}} \ar[rr, shift left = 2, "L^{\underline{\otimes}}"]\ar[dr] && \underline{\D}^{\underline{\otimes}}\ar[ll] \ar[dl]\\
        & \underline{\finite}_{*\baseCat}
    \end{tikzcd}
\end{center}
in the sense of \cite[{$\S7.3.2$}]{lurieHA} whose counit is moreover an equivalence.
\end{terminology}

\begin{lem}[$\baseCat$--adjunction on $\baseCat$--commutative algebras, ``{\cite[{Lem. 3.6}]{GGN}}'']
\label{TAdjunctionOnTCommutativeAlgebras}
Let $\underline{\sC}^{\underline{\otimes}}, \underline{\D}^{\underline{\otimes}}$ be $\baseCat$--symmetric monoidal categories and $L^{\underline{\otimes}} : \underline{\sC}^{\underline{\otimes}} \rightarrow \underline{\D}^{\underline{\otimes}}$ a $\baseCat$--symmetric monoidal localisation. Then there is an induced $\baseCat$--Bousfield localisation $L' : \myuline{\calg}_{\baseCat}(\underline{\sC})\rightarrow \myuline{\calg}_{\baseCat}(\underline{\D})$ such that the diagram
\begin{center}
    \begin{tikzcd}
    \myuline{\calg}_{\baseCat}(\underline{\sC})\rar["L'", shift left = 2] \dar& \myuline{\calg}_{\baseCat}(\underline{\D}) \lar[shift left =1 , "R'"]\dar\\
    \underline{\sC} \rar[shift left =2, "L"] & \underline{\D} \lar[shift left =1 , "R"]
    \end{tikzcd}
\end{center}
commutes, where the vertical maps are given by 
\[\myuline{\calg}_{\baseCat}(\underline{\sC}) \coloneqq \underline{\func}_{\baseCat}\Tinert(\underline{\finite}_{*\baseCat}, \underline{\sC}^{\underline{\otimes}})\times_{\underline{\func}_{\baseCat}(\underline{\finite}_{*\baseCat}, \underline{\finite}_{*\baseCat})}\terminalTCat  \longrightarrow \underline{\func}_{\baseCat}(\terminalTCat, \underline{\sC}) \simeq \underline{\sC}\] induced by the inclusion $\terminalTCat \hookrightarrow \underline{\finite}_{*\baseCat}$, which lands in the $\baseCat$--inerts. Moreover, given $A \in \myuline{\calg}_{\baseCat}(\underline{\sC})$ there is a unique $\baseCat$--commutative algebra structure on $RLA$ such that the unit map $A \rightarrow RLA$ enhances to a morphism of $\baseCat$--commutative algebras.
\end{lem}
\begin{proof}
First note that we have the adjunction squares
\begin{center}
    \begin{tikzcd}
     \underline{\func}_{\baseCat}\Tinert(\underline{\finite}_{*\baseCat},\underline{\sC}^{\underline{\otimes}}) \rar[shift left = 2, "L'"]\dar[hook] & \underline{\func}_{\baseCat}\Tinert(\underline{\finite}_{*\baseCat},\underline{\D}^{\underline{\otimes}}) \lar[shift left =1, "R'"]\dar[hook]\\
    \underline{\func}_{\baseCat}(\underline{\finite}_{*\baseCat},\underline{\sC}^ {\underline{\otimes}})\rar[shift left =2, "L_*^{\underline{\otimes}}"] & \underline{\func}_{\baseCat}(\underline{\finite}_{*\baseCat},\underline{\D}^{\underline{\otimes}})\lar[shift left =1, "R_*^{\underline{\otimes}}"] 
    \end{tikzcd}
\end{center}
where the bottom $\baseCat$--adjunction is by \cref{omnibusTAdjunctions} and has the property that the counit is an equivalence. Now \cite[{Prop. 7.3.2.5}]{lurieHA} says that relative adjunctions are stable under pullbacks and the property of being $\baseCat$--functors is of course preserved by pullbacks too, and so we get the square 

\begin{center}
\adjustbox{scale=0.85}{
    \begin{tikzcd}
    \myuline{\calg}_{\baseCat}(\underline{\sC})\dar[equal] & \myuline{\calg}_{\baseCat}(\underline{\D})\dar[equal]\\
    \underline{\func}_{\baseCat}\Tinert(\underline{\finite}_{*\baseCat},\underline{\sC}^{\underline{\otimes}})\times_{\underline{\func}_{\baseCat}(\underline{\finite}_{*\baseCat}, \underline{\finite}_{*\baseCat})}\terminalTCat \rar[shift left = 2, "L'"]\dar[hook] & \underline{\func}_{\baseCat}\Tinert(\underline{\finite}_{*\baseCat},\underline{\D}^{\underline{\otimes}})\times_{\underline{\func}_{\baseCat}(\underline{\finite}_{*\baseCat}, \underline{\finite}_{*\baseCat})}\terminalTCat\lar[shift left =1, "R'", hook]\dar[hook]\\
    \underline{\func}_{\baseCat}(\underline{\finite}_{*\baseCat},\underline{\sC}^ {\underline{\otimes}})\times_{\underline{\func}_{\baseCat}(\underline{\finite}_{*\baseCat}, \underline{\finite}_{*\baseCat})}\terminalTCat\rar[shift left =2, "L_*^{\underline{\otimes}}"] & \underline{\func}_{\baseCat}(\underline{\finite}_{*\baseCat},\underline{\D}^{\underline{\otimes}})\times_{\underline{\func}_{\baseCat}(\underline{\finite}_{*\baseCat}, \underline{\finite}_{*\baseCat})}\terminalTCat\lar[shift left =1, "R_*^{\underline{\otimes}}", hook] 
    \end{tikzcd}
    }
\end{center}

\normalsize

Then the square in the statement of the result is just composition of this square with the one induced by the inclusion $\terminalTCat \hookrightarrow \underline{\finite}_{*\baseCat}$ namely 
\begin{center}
\adjustbox{scale=0.85}{
    \begin{tikzcd}
    \myuline{\calg}_{\baseCat}(\underline{\sC})  \rar[shift left = 2, "L'"]\dar[hook] & \myuline{\calg}_{\baseCat}(\underline{\D})\lar[shift left =1, "R'", hook]\dar[hook]\\
    \underline{\func}_{\baseCat}(\underline{\finite}_{*\baseCat},\underline{\sC})\times_{\underline{\func}_{\baseCat}(\underline{\finite}_{*\baseCat}, \underline{\finite}_{*\baseCat})}\terminalTCat\rar[shift left =2, "L_*^{\underline{\otimes}}"] \dar& \underline{\func}_{\baseCat}(\underline{\finite}_{*\baseCat},\underline{\D})\times_{\underline{\func}_{\baseCat}(\underline{\finite}_{*\baseCat}, \underline{\finite}_{*\baseCat})}\terminalTCat\lar[shift left =1, "R_*^{\underline{\otimes}}", hook]\dar \\
    \underline{\sC} = \underline{\func}_{\baseCat}(\terminalTCat, \underline{\D}) \rar[shift left =2, "L_*"]   & \underline{\D} = \underline{\func}_{\baseCat}(\terminalTCat, \underline{\D}) \lar[shift left =1, "R_*", hook] 
    \end{tikzcd}
    }
\end{center}

For the next part, we know already that $R'L'A$ comes with a canonical $\baseCat$--commutative algebra map $\eta' : A \rightarrow R'L'A$ given by the $L'\dashv R'$ unit evaluated at $A$. By the square in the statement we see that this forgets to the $L\dashv R$ unit $\eta : A \rightarrow RLA$. Now if $\eta'' : A \rightarrow R'B$ is another such map of $\baseCat$--commutative algebras, then by universality of $\eta'$ we have an essentially unique factorisation $\phi \circ \eta' : A \rightarrow R'L'A \rightarrow R'B$. Now $\forget \colon \myuline{\calg}_{\baseCat}(\underline{\sC}) \rightarrow \underline{\sC}$ is conservative by \cite[{Lem. 3.2.2.6}]{lurieHA}, thus since $\phi$ forgets to the identity, $\phi$ must have been an equivalence in $\myuline{\calg}_{\baseCat}(\underline{\sC})$ as required.
\end{proof}

\subsection{Indexed (co)products of categories}\label{sec3:indexedConstructions}
We now investigate various permanence properties of indexed products on categories. To begin with, recall the following for which a good summary is \cite[Ex. 5.20]{quigleyShah}.

\begin{cons}[Indexed products of categories]\label{indexedProductsOfCategories}
Let  $f : U \rightarrow U'$ be a map of finite $\baseCat$--sets. Then \cref{cofreeParametrisations} gives us the equivalences in 
\small
\[f^* : \func(\totalCategory(\underline{U}'),\cat)\simeq \func_{\baseCat}(\underline{U}',\underline{\cat}_{\baseCat}) \rightarrow \func_{\baseCat}(\underline{U},\underline{\cat}_{\baseCat}) \simeq \func(\totalCategory(\underline{U}),\cat)\]\normalsize This has a right adjoint $f_*$ (also written $\prod_f$).  Thus, for $\underline{\sC} \in \underline{\cat}_{\underline{U}}$ and $\underline{\D} \in \underline{\cat}_{\underline{\D}}$ we have 
\[\underline{\func}_{\underline{U}'}\Big(\underline{\D}, f_*\underline{\sC}\Big)\simeq \underline{\func}_{\underline{U}}(f^*\underline{\D}, \underline{\sC})\]
By setting $\underline{\D} = \underline{V}$ we see that $f_*\underline{\sC}$ is a $\baseCat_{/U'}$--category with $V$--fibre given by 
\[\func_{\underline{U}}(\underline{U}_{\underline{V}},\underline{\sC}) \simeq \prod_{O\in\operatorname{Orbit}(U\times_{U'}V)}\sC_O\] where $\underline{U}_{\underline{V}}$ is the model for the corepresentable $\baseCat$--category associated to $U\times_{U'}V$ whose fibre over $[W\rightarrow U]$ is given by the space of commutative squares in $\finite_{\baseCat}$
\begin{center}
    \begin{tikzcd}
    W \rar\dar& U \dar \\
    V\rar & U'
    \end{tikzcd}
\end{center}
\end{cons}

\begin{lem}[Indexed constructions preserve adjunctions]\label{indexedConstructionsPreserveAdjunctions}
Let $f : W\rightarrow V$ be in $\baseCat$. Let $L : \underline{\sC} \rightleftarrows \underline{\D} : R$ be a $\baseCat_{/W}$--adjunction and $M : \underline{\A} \rightleftarrows \underline{\B} : N$ be a $\baseCat_{/V}$--adjunction. Then
\[f_*L : f_*\underline{\sC} \rightleftarrows f_*\underline{\D} : f_*R\quad\quad \quad f^*M : f^*\underline{\A} \rightleftarrows f^*\underline{\B} : f^*N\] are $\baseCat_{/V}$-- and $\baseCat_{/W}$--adjunctions respectively.
\end{lem}
\begin{proof}
By \cref{criteriaForTAdjunctions}, we need to show that these induce fibrewise adjunctions. This is clear for the pair $(f^*M, f^*N)$ since fibrewise they are the same as $(M, N)$; for $(f_*L, f_*R)$, we use that (unparametrised) products of adjunctions are again adjunctions.
\end{proof}

\begin{lem}[(Co)unit of indexed products]\label{unitCounitIndexedProductOfCategories}
The $\baseCat$--cofree category $\underline{\cat}_{\baseCat}$ strongly admits $\baseCat$--products, and for $f : W \rightarrow V$, $X \in \baseCat_{/W}$, and $Y \in \baseCat_{/V}$, we have that 
$(f_*\underline{\D})_Y \simeq \prod_{M \in \operatorname{Orbit}(Y\times_VW)}\D_M$ and moreover:
\begin{itemize}
    \item The unit is given by 
$\eta = F^* : \sC_Y \longrightarrow (f_*f^*\underline{\sC})_Y = \prod_{M\in \operatorname{Orbit}(Y\times_VW)}\sC_M$ where $F : Y\times_VW\rightarrow Y$ is the structure map from the pullback,
\item The counit is given by 
$\varepsilon  = \mathrm{proj} : (f^*f_*\underline{\sC})_X = \prod_{N\in\operatorname{Orbit}( X\times_VW)}\D_N \longrightarrow \D_X$ the component projection (see the proof for why we have this).
\end{itemize}
\end{lem}
\begin{proof}
We know that $f^* : \func((\baseCat_{/V})\op, \cat) \longrightarrow \func((\baseCat_{/W})\op, \cat)$  abstractly has a right adjoint $f_*$ via right Kan extension, and the formula for ordinary right Kan extensions gives us the required description (which is also gotten from \cref{indexedProductsOfCategories}).

To describe the (co)units, we have to check the triangle identities
\begin{equation}\label{triangleIdentities}
    \begin{tikzcd}
    f^* \rar["f^*\eta"]\ar[dr, equal] & f^*f_*f^* \dar["\varepsilon_{f^*}"] &&& f_* \rar["\eta_{f_*}"] \ar[dr,equal]& f_*f^*f_*\dar["f_*\varepsilon"]\\
    & f^* &&&& f_*
    \end{tikzcd}
\end{equation}
First of all we clarify why we have the counit map as stated. For this it will be helpful to write carefully the datum $\varphi : X \rightarrow W$ instead of just $X$. Consider
\begin{center}
    \begin{tikzcd}
    X \ar[ddr, "\varphi"', bend right = 20]\ar[drr,equal, bend left = 20] \ar[dr, dashed]\\
    & X\times_VW \rar\dar \ar[dr, phantom, very near start, "\scalebox{1.5}{$\lrcorner$}"] & X \dar["f\varphi"]\\
    & W \rar["f"] & V
    \end{tikzcd}
\end{center}
This shows that $X$ is a retract of $X\times_VW$, and so by atomicity, we get that $X$ was an orbit in the orbit decomposition of $X\times_VW$, and so the component projection $\varepsilon  : (f^*f_*\underline{\D})_X = \prod_{N\in \operatorname{Orbit}(X\times_VW)}\D_N \longrightarrow \D_X$ is well-defined

To check the first triangle identity, let $(\varphi: X \rightarrow W) \in \baseCat_{/W}$ and consider
\begin{center}
    \begin{tikzcd}
    \coprod_a N_a \rar["\coprod_a\xi_a"]\dar\ar[dr, phantom, very near start, "\scalebox{1.5}{$\lrcorner$}"] & X \dar["f\varphi"]\\
    W\rar["f"] & V
    \end{tikzcd}
\end{center}
where one of the $N_a$'s is $X$, by the argument above. Then we have that the composition in the first triangle in \cref{triangleIdentities} is
\begin{center}
\adjustbox{scale=0.9}{
    \begin{tikzcd}
        \Big((f^*\underline{\sC})_X \rar["f^*\eta"] & (f^*f_*f^*\underline{\sC})_X \rar["\varepsilon_{f^*}"] & (f^*\underline{\sC})_X \Big)\: \simeq\:\Big(\sC_X \rar["\prod_a\xi_a^*"] & \prod_a\sC_{N_a} \rar["\mathrm{proj}"] & \sC_X\Big)
    \end{tikzcd}
    }
\end{center}
which is of course the identity since $\xi_a = \id$ in the case $N_a = X$.

The second triangle identity is slightly more intricate. Let $(\psi : Y \rightarrow V)\in \baseCat_{/V}$. We consider two pullbacks (where the right square is for each $b$ appearing in the left square)
\begin{center}
    \begin{tikzcd}
    \coprod_b M_b \rar["\coprod_b\zeta_b"]\dar["\coprod_b\rho_b"']\ar[dr, phantom, very near start, "\scalebox{1.5}{$\lrcorner$}"] & Y \dar["\psi"] &&&\coprod_{c_b} \widetilde{M}_{c_b} \rar["\coprod_{c_b}\ell_{c_b}"]\dar\ar[dr, phantom, very near start, "\scalebox{1.5}{$\lrcorner$}"] & M_b \dar["f\rho_b"] \\
    W\rar["f"] & V &&& W \rar["f"] & V
    \end{tikzcd}
\end{center}
From this, the composition in the second triangle in \cref{triangleIdentities} is 
\begin{equation*}
    \begin{split}
        & \Big((f_*\underline{\D})_Y \xlongrightarrow{\eta_{f_*}} (f_*f^*f_*\underline{\D})_Y \xlongrightarrow{f_*\varepsilon}  (f_*\underline{\D})_Y \Big)\\
        &\simeq\Big(\prod_b\D_{M_b} \xlongrightarrow{\prod_b\prod_{c_b}\ell^*_{c_b}}\prod_b\prod_{c_b}\D_{\widetilde{M}_{c_b}} \xrightarrow{\prod_b\mathrm{proj}}  \prod_b\D_{M_b}\Big)
    \end{split}
\end{equation*}
which is the identity map as wanted since $M_b$ is one of the orbits in $\coprod_{c_b}\widetilde{M}_{c_b}$ by the argument above. Here we have used the diagram
\begin{center}
    \begin{tikzcd}
    (f_*\underline{\D})_Y\dar[equal] \ar[rr, "\eta_{f_*} = \prod_b\zeta_b^*"] && \prod_b (f_*\underline{\D})_{M_b}\dar[equal]\\
    \prod_b \D_{M_b} \ar[rr, "\prod_b\prod_{c_b}\ell^*_{c_b}"] && \prod_b\prod_{c_b}\D_{\widetilde{M}_{c_b}}
    \end{tikzcd}
\end{center}
to analyse the map $\eta_{f_*}$, which in turn comes from the top square in
\begin{center}
    \begin{tikzcd}
        & \coprod_{c_b}\widetilde{M}_{c_b}\ar[dd]\ar[ddrr, phantom, "\scalebox{1.5}{$\lrcorner$}", very near start] \ar[rr, "\coprod_{c_b}\ell_{c_b}"] \ar[dl,"\coprod_{c_b}\ell_{c_b}"']\ar[ddrr, phantom, "\scalebox{1.5}{$\lrcorner$}", very near start]&& M_b\ar[dl, "\zeta_b"]\ar[dd]\\
        \coprod_{b}M_b\ar[ddrr, phantom, "\scalebox{1.5}{$\lrcorner$}", very near start]\ar[dd]\ar[rr,crossing over]&& Y \\
        & W \ar[rr, "f"{xshift=-15pt}]\ar[dl,equal] && V\ar[dl,equal]\\
        W \ar[rr, "f"] && V \ar[uu, leftarrow, crossing over]
    \end{tikzcd}
\end{center}
This finishes the proof.
\end{proof}

\subsection{Norms and adjunctions}\label{subsec:normsAndAdjunctions}

We now recall the notion of $\baseCat$--distributivity and indexed tensor products (also termed \textit{norms}) of categories introduced in \cite[{$\S3.3$ and $\S3.4$}]{nardinThesis} and a nice summary of which can be found for instance in \cite[$\S5.1$]{quigleyShah}.

\begin{defn}
Let $f : U \rightarrow V$ be a map in $\finite_{\baseCat}$, $\underline{\sC} \in \cat_{\baseCat_{/U}}$, $\underline{\D} \in \cat_{\baseCat_{/V}}$, and $F : f_*\underline{\sC} \rightarrow \underline{\D}$ be a $\baseCat_{/V}$--functor. Then we say that $F$ is $\baseCat_{/V}$\textit{--distributive} if for every pullback
\begin{center}
    \begin{tikzcd}
    U'\rar["f'"]\dar["g'"']\ar[dr, phantom, very near start, "\scalebox{1.5}{$\lrcorner$}"]&V'\dar["g"]\\
    U\rar["f"] & V
    \end{tikzcd}
\end{center}
in $\finite_{\baseCat}$ and $\baseCat_{/U'}$--colimit diagram $p : \underline{K}^{\underline{\triangleright}}\rightarrow g'^*\underline{\sC}$, the $\myuline{V}'$--functor
\[(f_*'\underline{K})^{\underline{\triangleright}}\xrightarrow{\canonical} f'_*(\underline{K}^{\underline{\triangleright}}) \xrightarrow{f'_*p} f'_*p'^*\underline{\sC} \simeq g^*f_*\underline{\sC} \xrightarrow{g^*F}g^*\underline{\D}\] is a $\baseCat_{/V'}$--colimit. We write $\func^{\delta}_{\underline{V}}(f_*\underline{\sC},  \underline{\D})$ for the subcategory of $\baseCat_{/V}$--distributive functors.\label{nota:distributiveFunctors}
\end{defn}

\begin{cons}[Norms of categories]\label{normsOfCategories}
Let $f : U \rightarrow V$ be a map in $\finite_{\baseCat}$ and $\underline{\sC}$ a $\baseCat_{/U}$--category which is $\baseCat_{/U}$--cocomplete. Then we define the \textit{$f$-norm} $f_{\otimes}\underline{\sC}$, if it exists, to be a $\baseCat_{/V}$--cocomplete category admitting a $\baseCat_{/V}$--distributive functor $\tau : f_*\underline{\sC} \rightarrow f_{\otimes}\underline{\sC}$ such that for any other $\baseCat_{/V}$--cocomplete category, the following functor is an equivalence
\[\tau^* : {\func}^L_{\underline{V}}(f_{\otimes}\underline{\sC},  \underline{\D}) \rightarrow {\func}^{\delta}_{\underline{V}}(f_*\underline{\sC},  \underline{\D})\]
We also write this as $f_{\otimes} = \bigotimes_f$.
\end{cons}

\begin{lem}[Norms preserve adjunctions]\label{normsPreserveAdjunctions}
Let $F : \underline{\sC} \rightleftarrows \underline{\D} : G$ be a $\baseCat_{/U}$--adjunction such that $G$ itself admits a right adjoint and $f : U \rightarrow V$ be a map in $\finite_{\baseCat}$. Then this induces a $\baseCat_{/V}$--adjunction 
\[f_{\otimes}F : f_{\otimes}\underline{\sC} \rightleftarrows f_{\otimes}\underline{\D} : f_{\otimes}G \]
\end{lem}
\begin{proof}
Recall from \cref{indexedConstructionsPreserveAdjunctions} that we have a $\baseCat_{/V}$--adjunction
$f_*F : f_*\underline{\sC}\rightleftarrows f_*\underline{\D} : f_*G$ and since $G$ itself has a right adjoint, both $f_*F$ and $f_*G$ strongly preserve $\baseCat_{/V}$--colimits. Now observe that this adjunction can equivalently be encoded by the data of morphisms \[\big(\eta\colon \id \Rightarrow (f_*G)\circ (f_*F)\big) \in \underline{\func}^L_{\underline{V}}(f_*\underline{\sC},  f_*\underline{\sC})\]\[\big(\varepsilon\colon (f_*F)\circ (f_*G) \Rightarrow \id\big) \in \underline{\func}^L_{\underline{V}}(f_*\underline{\D}, f_*\underline{\D})\] whose images under the functors 
\[(f_*F)_* \colon {\func}^L_{\underline{V}}(f_*\underline{\sC},  f_*\underline{\sC}) \rightarrow {\func}^L_{\underline{V}}(f_*\underline{\sC},  f_*\underline{\D})\]\[(f_*F)^* \colon {\func}^L_{\underline{V}}(f_*\underline{\D}, f_*\underline{\D}) \rightarrow {\func}^L_{\underline{V}}(f_*\underline{\sC},  f_*\underline{\D})\]respectively compose to a morphism equivalent to the identity
\begin{center}
    \begin{tikzcd}
    f_*F \rar["f_*F(\eta)", Rightarrow]\ar[dr,equal] & (f_*F)\circ(f_*G)\circ(f_*F)\dar["\varepsilon_{f_*F}", Rightarrow]\\
    & f_*F
    \end{tikzcd}
\end{center}
and similarly for the other triangle identity. Now, we have commutative squares
\begin{center}
    \begin{tikzcd}
    f_*\underline{\sC} \dar["\varphi"']\rar[shift left = 1, "f_*F"] & f_*\underline{\D} \lar[shift left =1 , "f_*G"]\dar["\psi"]\\
    f_{\otimes}\underline{\sC} \rar[shift left = 1, "f_{\otimes}F"] & f_{\otimes}\underline{\D} \lar["f_{\otimes}G", shift left = 1]
    \end{tikzcd}
\end{center}
where $\varphi : f_*\underline{\sC} \rightarrow f_{\otimes}\underline{\sC}$, $\psi : f_*\underline{\D} \rightarrow f_{\otimes}\underline{\D}$ are the universal distributive functors: this is since $G$ strongly preserves $\baseCat$--colimits by hypothesis. This yields
\begin{center}
    \begin{tikzcd}
    {\func}^L_{\underline{V}}(f_*\underline{\sC},  f_*\underline{\sC}) \dar["\varphi_*"'] \rar["(f_*F)_*"]& {\func}^L_{\underline{V}}(f_*\underline{\sC},  f_*\underline{\D})\dar["\psi_*"] & {\func}^L_{\underline{V}}(f_*\underline{\D}, f_*\underline{\D})\lar["(f_*F)^*"']\dar["\psi_*"]\\
    {\func}^{\delta}_{\underline{V}}(f_*\underline{\sC},  f_{\otimes}\underline{\sC}) \rar["(f_{\otimes}F)_*"] & {\func}^{\delta}_{\underline{V}}(f_*\underline{\sC},  f_{\otimes}\underline{\D}) & {\func}^{\delta}_{\underline{V}}(f_*\underline{\D}, f_{\otimes}\underline{\D})\lar["(f_{\otimes}F)^*"']\\
    {\func}^L_{\underline{V}}(f_{\otimes}\underline{\sC},  f_{\otimes}\underline{\sC}) \uar["\varphi^*", "\simeq"'] \rar["(f_{\otimes}F)_*"]& {\func}^L_{\underline{V}}(f_{\otimes}\underline{\sC},  f_{\otimes}\underline{\D})\uar["\varphi^*", "\simeq"']& {\func}^L_{\underline{V}}(f_{\otimes}\underline{\D}, f_{\otimes}\underline{\D})\lar["(f_{\otimes}F)^*"']\uar["\psi^*", "\simeq"']
    \end{tikzcd}
\end{center}
Then the morphism $\big(\eta : \id \Rightarrow (f_{*}G)\circ (f_{*}F)\big) \in \underline{\func}^L_{\underline{V}}(f_*\underline{\sC},  f_*\underline{\sC})$ in the top left corner gets sent to a morphism $\big(\widetilde{\eta} : \id \Rightarrow (f_{\otimes}G)\circ (f_{\otimes}F)\big)\in \underline{\func}^L_{\underline{V}}(f_{\otimes}\underline{\sC},  f_{\otimes}\underline{\sC})$ in the bottom left,  and similarly for $\varepsilon$. Then by the characterisation of adjunctions above, since the composition of the images in the middle top term is equivalent to the identity, so is the image in the middle bottom term, that is, we have the commuting diagram
\begin{center}
    \begin{tikzcd}
    f_{\otimes}F \rar["f_{\otimes}F(\widetilde{\eta})", Rightarrow]\ar[dr,equal] & (f_{\otimes}F)\circ(f_{\otimes}G)\circ(f_{\otimes}F)\dar["\widetilde{\varepsilon}_{f_{\otimes}F}", Rightarrow]\\
    & f_{\otimes}F
    \end{tikzcd}
\end{center}
and similarly for the other triangle identity. This witnesses that we have a $\baseCat_{/V}$--adjunction $f_{\otimes}F \dashv f_{\otimes}G$ as required.
\end{proof}

\begin{rmk}
\label{RecollectionOnNormMorphisms}
Let $V, W \in \baseCat$ and $\sC^ {\underline{\otimes}}$ a $\baseCat$--symmetric monoidal category. Concretely speaking, we get the structure of tensor products and norm functors as follows:
\begin{itemize}
    \item (Tensor functor): Consider the morphism in $\underline{\finite}_{*\baseCat}$ given by
    \begin{center}
        \begin{tikzcd}
        V\coprod V \rar[equal]\dar["\nabla"] & V\coprod V \dar["\nabla"] \rar["\nabla"]& V\dar[equal]  \\
        V \rar[equal] & V \rar[equal] & V
        \end{tikzcd}
    \end{center}
    The cocartesian lifts along this morphism give us the tensor product
    \[\otimes : \sC_V\times \sC_V \simeq \sC_{V\coprod V} \longrightarrow \sC_V\] 
    \item (Norm functor): Suppose $f : V \rightarrow W$ is a morphism in $\baseCat$. Consider
    \begin{center}
        \begin{tikzcd}
        V \rar[equal]\dar["f"] & V \dar["f"] \rar["f"]& W\dar[equal]  \\
        W \rar[equal] & W \rar[equal] & W
        \end{tikzcd}
    \end{center}
    The cocartesian lifts along this morphism give us the norm functor \label{nota:normFunctor}
    \[\norm^f : \sC_V \simeq \sC^ {\underline{\otimes}}_{[f : V\rightarrow W]} \longrightarrow \sC^ {\underline{\otimes}}_{[W=W]} \simeq \sC_W\]  Note that it might have been tempting to define the norm functor as the pushforward along the more obvious morphism
    \begin{center}
        \begin{tikzcd}
        V \rar[equal]\dar[equal] & V \dar[equal] \rar["f"]& W\dar[equal]  \\
        V \rar[equal] & V \rar["f"] & W
        \end{tikzcd}
    \end{center} 
    instead, but the problem is that this is not a morphism in $\underline{\finite}_{*\baseCat}$ because by definition the bottom right map needs to be the identity! 
    
\end{itemize}
\end{rmk}

\section{Parametrised smallness adjectives}
\label{sec4:Compactness}

We now introduce the notion of $\baseCat$--compactness and $\baseCat$--idempotent-completeness. Not only are these notions crucial in proving the characterisations of $\baseCat$--presentables in \cref{simpsonTheorem}, they are also fundamental for the applications we have in mind for parametrised algebraic K--theory in \cite{kaifNoncommMotives}. The moral of this section is that these are essentially fibrewise notions and should present no conceptual difficulties to those already familiar with the unparametrised versions. Recall that we will assume throughout that $\baseCat$ is orbital.

\subsection{Parametrised compactness}
Recall that an object $X$ in a category $\sC$ is compact if $\map_{\sC}(X,-) : \sC \rightarrow \spc$ commutes with filtered colimits (cf. {\cite[$\S5.3.4$]{lurieHTT}}). In this subsection we introduce the parametrised analogue of this notion and study its interaction with Ind-completions.

\begin{defn}\label{defn:parametrisedCompact}
Let $\underline{\sC}$ be a $\baseCat$--category and $V\in \baseCat$. A $V$-object in $\underline{\sC}$ (ie. an object in $\func_{\baseCat}(\underline{V},\underline{\sC}) $) is $\baseCat_{/V}$--$\kappa$-\textit{compact} if it is fibrewise $\kappa$-compact. We will also use the terminology \textit{parametrised-$\kappa$-compact objects} when we allow $V$ to vary. We write $\underline{\sC}\tkappa$ for the $\baseCat$--subcategory of parametrised-$\kappa$-compact objects, that is, $(\underline{\sC}\tkappa)_V$ is given by the full subcategory of $\baseCat_{/V}$--$\kappa$-compact objects.
\end{defn}

\begin{nota}\label{nota:parametrisedCompactPreservingFunctors}
We write $\underline{\func}_{\baseCat}\tkappa$ for the full $\baseCat$--subcategory of parametrised functors preserving parametrised $\kappa$-compact objects.
\end{nota}

\begin{warning}
In general, for $V\in \baseCat\op$, the inclusion $(\underline{\sC}\tkappa)_V \subseteq (\sC_V)^{{\kappa}}$ is not an equivalence - the point is that parametrised-$\kappa$-compactness must be preserved under the cocartesian lifts $f^* : \sC_V \rightarrow \sC_W$ for all $f :  W \rightarrow V$, but these do not preserve $\kappa$-compactness in general.
\end{warning}

This definition of compactness makes sense by virtue of the following:

\begin{prop}[Characterisation of parametrised-compactness]\label{characterisationOfTCompactness}
Let $\underline{\sC}$ admit fibrewise $\kappa$-filtered $\baseCat$--colimits. A $\baseCat$--object $C \in \func_{\baseCat}(\terminalTCat, \underline{\sC}) $ is $\kappa$--$\baseCat$--compact in the sense above if and only if for all $V\in \baseCat$ and all fibrewise $\kappa$-filtered $\baseCat_{/V}$--diagram $d : \tconstant_{\underline{V}}(K) \rightarrow \underline{\sC}_{\underline{V}}$ the comparison 
\[\underline{\colim}_{\tconstant_{\underline{V}}(K)}\myuline{\map}_{\underline{\sC}_{\underline{V}}}(C_{\underline{V}},d) \rightarrow \myuline{\map}_{\underline{\sC}_{\underline{V}}}(C_{\underline{V}},\underline{\colim}_{\tconstant_{\underline{V}}(K)}d)\]
is an equivalence.
\end{prop}
\begin{proof}
Suppose $C$ is $\kappa$--$\baseCat$--compact. We are already provided with the comparison map above, and we just need to check that it is an equivalence, which can be done by checking fibrewise. Since $\totalCategory(\underline{V}) = (\baseCat_{/V})\op$ has an initial object, we can assume that $\baseCat$ has a final object. So let $W\in \baseCat$. Recall that as in the proof of \cref{paramAdjunctionMappingAnima} we have 
\[\Big(\myuline{\map}_{\underline{\sC}}(C,d)\Big)_W \simeq \Big(\map_{\sC_{\bullet}}(C_{\bullet}, d_{\bullet})\Big)_{\bullet \in (\baseCat_{/W})\op}\in \func((\baseCat_{/W})\op, \spc)\]
Then 
\begin{equation*}
    \begin{split}
        \Big(\underline{\colim}_{\tconstant_{\underline{V}}(K)}\myuline{\map}_{\underline{\sC}_{\underline{V}}}(C_{\underline{V}},d)\Big)_W & \simeq \colim_K\Big(\map_{\sC_{\bullet}}(C_{\bullet}, d_{\bullet})\Big)_{\bullet \in (\baseCat_{/W})\op}\\ & \xrightarrow{\simeq} \Big(\map_{\sC_{\bullet}}(C_{\bullet}, \colim_Kd_{\bullet})\Big)_{\bullet \in (\baseCat_{/W})\op}
    \end{split}
\end{equation*}
where the first equivalence is since fibrewise parametrised colimits are computed fibrewise, and the comparison map is an equivalence since colimits in  $\func((\baseCat_{/V})\op, \spc)$ are computed pointwise, and $C$ is pointwise $\kappa$-compact by hypothesis. 

Now for the reverse direction, let $C\in \underline{\sC}$ satisfy the property in the statement and $V\in \baseCat$ arbitrary. We want to show that $C_V\in \sC_V$ is $\kappa$-compact, that is: for any ordinary small $\kappa$-filtered diagram $d : K \rightarrow \sC_V$, we have that 
\[\colim_K\map_{\sC_V}(C_V, d) \rightarrow \map_{\sC_V}(C_V, \colim_K d)\] is an equivalence. Now recall that $\sC_V = \func_{\underline{V}}(\underline{V}, \underline{\sC}_{\underline{V}})$ by \cref{categoryOfPoints} and so by adjunction we obtain from $d : K \rightarrow \sC_V$ a $\baseCat_{/V}$--functor
$\overline{d} : \tconstant_{\underline{V}}(K) \longrightarrow \underline{\sC}_{\underline{V}}$. In this case the desired comparison is an equivalence by virtue of the following diagram
\begin{center}
    \begin{tikzcd}
    \colim_K\map_{\sC_V}(C_V, d) \dar[equal]\rar & \map_{\sC_V}(C_V, \colim_K d) \dar[equal]\\
    \big(\underline{\colim}_{\tconstant_{\underline{V}}(K)}\myuline{\map}_{\underline{\sC}_{\underline{V}}}(C_{\underline{V}},d)\big)_V \rar["\simeq"] & \big(\myuline{\map}_{\underline{\sC}_{\underline{V}}}(C_{\underline{V}},\underline{\colim}_{\tconstant_{\underline{V}}(K)}\overline{d})\big)_V
    \end{tikzcd}
\end{center}
where the bottom map is an equivalence by hypothesis. This finishes the proof.
\end{proof}

\begin{obs}\label{observationYonedaLandCompacts}
By the characterisation of $\baseCat$--compactness above together with the $\baseCat$--Yoneda  \cref{TYonedaLemma}, and that $\baseCat$--colimits in $\baseCat$--functor categories are computed in the target by \cref{TColimitsFunctorCategories} we see that the $\baseCat$--Yoneda embedding lands in $\underline{\presheaf}_{\baseCat}(\underline{\sC}) \tkappa$.
\end{obs}

\begin{prop}[$\baseCat$--compact closure, ``{\cite[Cor. 5.3.4.15]{lurieHTT}}'']\label{compactClosure}
Let $\kappa$ be a regular cardinal and $\underline{\sC}$ be $\baseCat$--cocomplete. Then $\underline{\sC}\tkappa$ is closed under $\kappa$-small $\baseCat$--colimits in $\underline{\sC}$, and hence is $\kappa$--$\baseCat$--cocomplete.
\end{prop}
\begin{proof}
Let $d : K\rightarrow \underline{\sC}\tkappa$ be a $\kappa$-small $\baseCat$--diagram. Since all $\kappa$-small $\baseCat$--colimits can be decomposed as $\kappa$-small fibrewise $\baseCat$--colimits and $\baseCat$--coproducts by the decomposition principle in \cref{omnibusTColimits} (3), we just have to treat these two special cases. The former case is clear by \cite[Cor. 5.3.4.15]{lurieHTT} since everything is fibrewise. For the latter case, let $\underline{V}$ be a corepresentable $\baseCat$--category, $A$ be a $\kappa$-filtered category, and $f : \tconstant_\baseCat(A)\rightarrow \underline{\sC}$ be a $\kappa$-filtered fibrewise $\baseCat$--diagram. We need to show that the map in $\underline{\spc}_{\baseCat}$
\[\underline{\colim}_{\tconstant_\baseCat(A)}\myuline{\map}_{\underline{\sC}}(\underline{\colim}_{\underline{V}}d, f)\longrightarrow\myuline{\map}_{\underline{\sC}}(\underline{\colim}_{\underline{V}}d, \underline{\colim}_{\tconstant_\baseCat(A)}f)\] is an equivalence.
In this case, since we have for the source \[\underline{\colim}_{\tconstant_\baseCat(A)}\myuline{\map}_{\underline{\sC}}(\underline{\colim}_{\underline{V}}d, f) \simeq \underline{\colim}_{\tconstant_\baseCat(A)}\underline{\lim}_{{\underline{V}}\vop}\myuline{\map}_{\underline{\sC}}(d, f)\] and for the target 
\[\myuline{\map}_{\underline{\sC}}(\underline{\colim}_{\underline{V}}d, \underline{\colim}_{\tconstant_\baseCat(A)}f) \simeq \underline{\lim}_{{\underline{V}}\vop}\underline{\colim}_{\tconstant_\baseCat(A)}\myuline{\map}_{\underline{\sC}}(d, f), \]
\cref{CommutationRelativeFiltColimAndFiniteLim} gives the required equivalence, using also that $\underline{V}\vop$ is still corepresentable by \cref{oppositeCorepresentability}.
\end{proof}

\subsection{Parametrised Ind-completions and accessibility}

\begin{prop}[$\underline{\ind}$ fully faithfulness, ``{\cite[Prop. 5.3.5.11]{lurieHTT}}''] \label{TIndFullyFaithfulness}
Let $\underline{\sC} \in \cat_{\baseCat}$ and  $\underline{\D} \in \widehat{\cat}_{\baseCat}$ which strongly admits fibrewise $\kappa$-filtered colimits. Suppose $F : \underline{\ind}_{\kappa}\underline{\sC}  \rightarrow \underline{\D}$ strongly preserves fibrewise $\kappa$-filtered colimits and $f  = F\circ j : \underline{\sC} \rightarrow \underline{\D}$.
\begin{enumerate}
    \item If $f$ is $\baseCat$--fully faithful and the $\baseCat$--essential image lands in $\underline{\D}\tkappa$, then $F$ is $\baseCat$--fully faithful.
    \item If $f$ is $\baseCat$--fully faithful, lands in $\underline{\D}\tkappa$, and the $\baseCat$--essential image of $f$ generates $\underline{\D}$ under fibrewise $\kappa$-filtered colimits, then $F$ is moreover a $\baseCat$--equivalence.
\end{enumerate}
\end{prop}
\begin{proof}
We prove (i) two steps. The goal is to show that 
\[\myuline{\map}_{\underline{\ind}_{\kappa}\underline{\sC} }(A,B) \rightarrow \myuline{\map}_{\underline{\D}}(FA,FB)\] is an equivalence. First suppose $A \in \underline{\sC}$ and write $B \simeq \underline{\colim}_iB_i$ as a fibrewise filtered colimit where $B_i \in \underline{\sC}$. We can equivalently compute $\myuline{\map}_{\underline{\ind}_{\kappa}\underline{\sC} }(A,B)$ as $\myuline{\map}_{\underline{\presheaf}_{\baseCat}(\underline{\sC}) }(A,B)$, and so 
\begin{equation*}
    \begin{split}
        \myuline{\map}_{\underline{\ind}_{\kappa}\underline{\sC} }(A,B)\simeq \myuline{\map}_{\underline{\presheaf}_{\baseCat}(\underline{\sC}) }(A,\underline{\colim}_iB_i) &\simeq \underline{\colim}_i\myuline{\map}_{\underline{\presheaf}_{\baseCat}(\underline{\sC}) }(A,B_i)\\
        &\simeq \underline{\colim}_i\myuline{\map}_{\underline{\ind}_{\kappa}\underline{\sC} }(A,B_i)
    \end{split}
\end{equation*}
and \[ \myuline{\map}_{\underline{\D}}(FA,F\underline{\colim}_iB_i) \simeq \underline{\colim}_i\myuline{\map}_{\underline{\D}}(fA,fB_i)\] where for the second equivalence we have used both hypotheses that $F$ preserves fibrewise $\kappa$-filtered colimits and that the image lands in $\underline{\D}\tkappa$. This completes this case. For a general $A \simeq \underline{\colim}_iA_i$ where $A_i \in \underline{\sC}$ and the $\baseCat$--colimit is fibrewise $\kappa$-filtered,  we have 
\begin{equation*}
    \begin{split}
        \myuline{\map}_{\underline{\presheaf}_{\baseCat}(\underline{\sC}) }(A,B) \simeq \myuline{\map}_{\underline{\presheaf}_{\baseCat}(\underline{\sC}) }(\underline{\colim}_iA_i,B) & \simeq \underline{\lim}_i\myuline{\map}_{\underline{\presheaf}_{\baseCat}(\underline{\sC}) }(A_i,B)\\
        &\xrightarrow{\simeq} \underline{\lim}_i\myuline{\map}_{\underline{\D}}(FA_i,FB) \\
        &\simeq \myuline{\map}_{\underline{\D}}(FA, FB)
    \end{split}
\end{equation*}
 where the third equivalence is by the special case above, and so we are done. For (ii), we have shown $\baseCat$--fully faithfulness, and $\baseCat$--essential surjectivity is by hypothesis.
\end{proof}

\begin{lem}\label{indYonedaPreserveFiniteColimits}
Let $\underline{\D}\in\cat_{\baseCat}$. Then the $\baseCat$--Yoneda embedding $y \colon \underline{\D}\hookrightarrow \underline{\func}_{\baseCat}\Tlexact(\underline{\D}\vop, \underline{\spc}_{\baseCat})$ strongly preserves finite $\baseCat$--colimits.
\end{lem}
\begin{proof}
Suppose $k : \underline{K} \rightarrow \underline{\D}$ is a finite $\baseCat$--diagram. We need to show that the map 
\[\underline{\colim}_{\underline{K}}\myuline{\map}_{\underline{\D}}(-,k)\rightarrow \myuline{\map}_{\underline{\D}}(-,\underline{\colim}_{\underline{K}}k)\] in $\underline{\func}_{\baseCat}\Tlexact(\underline{\D}\vop, \underline{\spc}_{\baseCat})$ is an equivalence. So let $\varphi \in \underline{\func}_{\baseCat}\Tlexact(\underline{\D}\vop, \underline{\spc}_{\baseCat})$ be an arbitrary object. Then mapping the morphism above into this and using Yoneda, we obtain
\[\varphi(\underline{\colim}_{\underline{K}}k) \longrightarrow\underline{\lim}_{\underline{K}\vop}\varphi(k)\] which is an equivalence since $\varphi$ is a $\baseCat$--left exact functor. 
\end{proof}

We thank Maxime Ramzi for teaching us the following slick proof, which is different from the standard one from \cite[Prop. 3.2]{BGT13}, for instance.

\begin{prop}\label{functorCategoryFormulaForIndCompletion}
Let $\underline{\D}\in \cat_{\baseCat}$. Then $\underline{\ind}(\underline{\D})\simeq \underline{\func}\Tlexact(\underline{\D}\vop, \underline{\spc}_{\baseCat})$. In particular, if  $\underline{\D}$ were $\baseCat$--stable, then $\underline{\ind}(\underline{\D}) \simeq \underline{\func}\Texact(\underline{\D}\vop, \myuline{\spectra}_{\baseCat})$.
\end{prop}
\begin{proof}
First of all, note that $\underline{\func}\Tlexact(\underline{\D}\vop, \underline{\spc}_{\baseCat}) \subseteq \underline{\func}(\underline{\D}\vop, \underline{\spc}_{\baseCat})$ is closed under fibrewise filtered colimits since fibrewise filtered colimits commutes with finite $\baseCat$--limits in $\underline{\spc}_{\baseCat}$ by \cref{CommutationRelativeFiltColimAndFiniteLim}. Hence $y \colon \underline{\D} \hookrightarrow \underline{\func}\Tlexact(\underline{\D}\vop, \underline{\spc}_{\baseCat})$ induces 
$\overline{y} \colon \underline{\ind}(\underline{\D}) \longrightarrow \underline{\func}\Tlexact(\underline{\D}\vop, \myuline{\spc}_{\baseCat})$
which we then know is $\baseCat$--fully faithful by \cref{TIndFullyFaithfulness}. Moreover, since $y$ strongly preserves finite $\baseCat$--colimits by \cref{indYonedaPreserveFiniteColimits}, $\overline{y}$ strongly preserves small $\baseCat$--colimits. Hence, by \cref{parametrisedAdjointFunctorTheorem}, it has a right adjoint $R : \underline{\func}\Texact(\underline{\D}\vop, \myuline{\spc}_{\baseCat}) \rightarrow \underline{\ind}(\underline{\D})$ (we are free to use this result here since the present situation will not feature anywhere in the proof of adjoint functor theorem). If we can show that this right adjoint is conservative, then we would have shown that $\overline{y}$ and $R$ are inverse equivalences. But conservativity is clear by mapping from representable functors and an immediate application of Yoneda. Finally, the statement for the $\baseCat$--stable case is a straightforward consequence of \cref{univPropTStabilisations}.
\end{proof}

\begin{prop}[``{\cite[Prop. 5.3.5.12]{lurieHTT}}'']\label{whenCocompletionsAreIndCompletions}
Let $\underline{\sC}\in \cat_{\baseCat}$  and $\kappa$ a regular cardinal. Then the canonical functor $F : \underline{\ind}_{\kappa}(\underline{\presheaf}_{\baseCat}(\underline{\sC}) \tkappa) \rightarrow \underline{\presheaf}_{\baseCat}(\underline{\sC}) $ is  an equivalence.
\end{prop}
\begin{proof}
To see that $F$ is an equivalence, we want to apply  \cref{TIndFullyFaithfulness}. Let $j : \underline{\presheaf}_{\baseCat}(\underline{\sC}) \tkappa\hookrightarrow \underline{\ind}_{\kappa}(\underline{\presheaf}_{\baseCat}(\underline{\sC}) \tkappa)$ be the canonical embedding. That the composite $f \coloneqq F\circ j$  is $\baseCat$--fully faithful and lands in $\underline{\presheaf}_{\baseCat}(\underline{\sC}) \tkappa$ is clear. To see that the essential image of $f$ generates $\underline{\presheaf}_{\baseCat}(\underline{\sC}) $ under fibrewise $\kappa$-filtered colimits, recall that any $X\in\underline{\presheaf}_{\baseCat}(\underline{\sC}) $ can be written as a small $\baseCat$--colimit of a diagram valued in $\underline{\sC} \subseteq \underline{\presheaf}_{\baseCat}(\underline{\sC}) $ by \cref{TYonedaDensity}. Then \cref{parametrisedFilteredColimitDecomposition} gives 
that $X$ can be written as a fibrewise $\kappa$-filtered colimit taking values in $\underline{\E}\subseteq \underline{\presheaf}_{\baseCat}(\underline{\sC}) $ where each object of $\underline{\E}$ is itself a $\kappa$-small $\baseCat$--colimit of some diagram taking values in $\underline{\sC} \subseteq \underline{\presheaf}_{\baseCat}(\underline{\sC}) \tkappa.$ But then by \cref{compactClosure} we know that $\underline{\E} \subseteq \underline{\presheaf}_{\baseCat}(\underline{\sC}) \tkappa$, and so this completes the proof.
\end{proof}

\begin{prop}[Characterisation of $\baseCat$--compacts in $\baseCat$--presheaves, ``{\cite[Prop. 5.3.4.17]{lurieHTT}}'']\label{TCompactsInTPresheaves}
Let $\underline{\sC}\in\cat_{\baseCat}$ and $\kappa$ a regular cardinal. Then a $\baseCat$--object $C\in \underline{\presheaf}_{\baseCat}(\underline{\sC}) $ is $\kappa$--$\baseCat$--compact if and only if it is a retract of a $\kappa$-small $\baseCat$--colimit indexed in $\underline{\sC} \subseteq \underline{\presheaf}_{\baseCat}(\underline{\sC}) $. 
\end{prop}
\begin{proof}
The if direction is clear since $\underline{\sC} \subseteq \underline{\presheaf}_{\baseCat}(\underline{\sC}) \tkappa$ and by the compact closure of \cref{compactClosure} we know that $\kappa$--$\baseCat$--compacts are closed under $\kappa$-small $\baseCat$--colimits and retracts.

Now suppose $C$ is $\kappa$--$\baseCat$--compact. First of all recall by \cref{TYonedaDensity} that $C \simeq \underline{\colim}_aj(B_a)$ where $j : \underline{\sC} \hookrightarrow \underline{\presheaf}_{\baseCat}(\underline{\sC}) $ is the $\baseCat$--Yoneda embedding and $B_a\in \underline{\sC}$. Combining this with \cref{parametrisedFilteredColimitDecomposition} yields
\[C = \underline{\colim}_aj(B_a) \simeq \underline{\colim}_{f\in \tconstant_\baseCat(F)} \underline{\colim}_{\big(p_f : K_f \rightarrow \underline{\sC} \subseteq \underline{\presheaf}_{\baseCat}(\underline{\sC}) \big)}p_f\] where $F$ is a $\kappa$-filtered category. But then by \cref{characterisationOfTCompactness} we then have that 
\[\id_C \in \myuline{\map}_{\underline{\sC}}(C, C) \simeq \underline{\colim}_{f\in \tconstant_\baseCat(F)}\myuline{\map}_{\underline{\sC}}(C, \underline{\colim}_{\big(p_f : K_f \rightarrow \underline{\sC} \subseteq \underline{\presheaf}_{\baseCat}(\underline{\sC}) \big)}p_f)\] Hence we see that $C$ is a retract of some $\underline{\colim}_{\big(p_f : K_f \rightarrow \underline{\sC} \subseteq \underline{\presheaf}_{\baseCat}(\underline{\sC}) \big)}p_f$ as required.
\end{proof}

\begin{defn}
Let $\kappa$ be a regular cardinal and $\underline{\sC}$ a $\baseCat$--category. We say that $\underline{\sC}$ is $\kappa$--$\baseCat$--\textit{accessible} if there is a small $\baseCat$--category $\underline{\sC}^0$ and a $\baseCat$--equivalence $\underline{\ind}_{\kappa}(\underline{\sC}^0)\rightarrow \underline{\sC}$. We say that $\underline{\sC}$ is $\baseCat$--\textit{accessible} if it is $\kappa$--$\baseCat$--accessible for some regular cardinal $\kappa$. A $\baseCat$--functor out of a $\baseCat$--accessible $\underline{\sC}$ is said to be $\baseCat$\textit{-accessible} if it strongly preserves all fibrewise $\kappa$-filtered colimits for some regular cardinal $\kappa$.
\end{defn}

\begin{lem}[$\baseCat$--accessibility of $\baseCat$--adjoints, ``{\cite[Prop. 5.4.7.7]{lurieHTT}}'']\label{accessibilityOfAdjoints}
Let $G : \underline{\sC} \rightarrow \underline{\sC}'$ be a $\baseCat$--functor between $\baseCat$--accessibles. If $G$ admits a right or a left $\baseCat$--adjoint, then $G$ is $\baseCat$--accessible.
\end{lem}
\begin{proof}
The case of left $\baseCat$--adjoints is clear since these strongly preserve all $\baseCat$--colimits, so suppose $G\dashv F$. Choose a regular cardinal $\kappa$ so that $\underline{\sC}'$ is $\kappa$-accessible, ie. $\underline{\sC}' = \underline{\ind}_{\kappa}\underline{\D}$ for some $\underline{\D}$ small. Consider the composite
$\underline{\D} \xrightarrow{j} \underline{\ind}_{\kappa}\underline{\D} \xrightarrow{F} \underline{\sC}$. Since $\underline{\D}$ is small there is a regular cardinal $\tau \gg \kappa$ so that both $\underline{\sC}$ is $\tau$-accessible and the essential image of $F\circ j$ consists of $\tau$-$\baseCat$--compact objects of $\underline{\sC}$. We will show that $G$ strongly preserves fibrewise $\tau$-filtered colimits.

Since $\underline{\ind}_{\kappa}\underline{\D} \subseteq \underline{\presheaf}_{\baseCat}(\underline{\D})$ is stable under small $\tau$-filtered colimits by \cref{UnivPropInd} it will suffice to prove that
\[G' : \underline{\sC} \xrightarrow{G}\underline{\ind}_{\kappa}\underline{\D} \rightarrow \underline{\presheaf}_{\baseCat}(\underline{\D})\] preserves fibrewise $\tau$-filtered colimits. Since colimits in presheaf categories are computed pointwise by \cref{TColimitsFunctorCategories} it suffices to show this when evaluated at each $D\in \D_V$ for all $V \in \baseCat$. Without loss of generality we just work with $D \in \D$, ie. a $\baseCat$--object $D \in \func_{\baseCat}(\terminalTCat, \underline{\D})$. In other words, by the $\baseCat$--Yoneda lemma we just need to show that
\[G'_D : \underline{\sC} \xrightarrow{G}\underline{\ind}_{\kappa}\underline{\D} \hookrightarrow \underline{\presheaf}_{\baseCat}(\underline{\D}) \xrightarrow{\myuline{\map}_{\underline{\presheaf}_{\baseCat}(\underline{\D})}(j(D), -)} \underline{{\spc}}_{\baseCat}\] preserves fibrewise $\tau$-filtered colimits. But $G$ is a right adjoint and so by \cref{paramAdjunctionMappingAnima}
\[\myuline{\map}_{\underline{\presheaf}_{\baseCat}(\underline{\D})}(j(D), G(-)) \simeq \myuline{\map}_{\underline{\ind}_{\kappa}\underline{\D}}(j(D), G(-)) \simeq \myuline{\map}_{\underline{\sC}}(Fj(D), -)\] By assumption on $\tau$, $Fj$ lands in $\tau$-compact objects, completing the proof.
\end{proof}

\subsection{Parametrised idempotent-completeness}
Recall that every retraction $r : X \rightleftarrows M : i$ gives rise to an idempotent self-map $i\circ r$ of $X$ since $(i\circ r)\circ (i\circ r) \simeq i\circ (r\circ i)\circ r \simeq i\circ r$. On the other hand, in general, not every idempotent self-map of an object in a category arises in this way, and a category is defined to be idempotent-complete if every idempotent self-map of an object arises from a retraction (cf. \cite[$\S4.4.5$]{lurieHTT}). We now introduce the parametrised version of this.
\begin{defn}
A $\baseCat$--category is said to be $\baseCat$\textit{-idempotent-complete} if it is so fibrewise. A $\baseCat$--functor $f : \underline{\sC} \rightarrow \underline{\D}$ is said to be a $\baseCat$\textit{-idempotent-completion} if it is fibrewise  an idempotent-completion (cf. \cite[Def. 5.1.4.1]{lurieHTT}).
\end{defn}

\begin{obs}[Consequences of fibrewise definitions]\label{consequencesOfFibrewiseDefinitions}
Here are some facts we can immediately glean from our fibrewise definitions.
\begin{enumerate}
    \item We know that for $\underline{\sC}$ small, $\ind_{\kappa}(\ind_{\kappa}(\sC) ^{\kappa})\simeq \ind_{\kappa}(\sC) $, and so since $\baseCat$--compactness and $\baseCat$--Ind objects are fibrewise notions, we also get that for any small $\baseCat$--category $\underline{\sC}$ we have $\underline{\ind}_{\kappa}\big(\underline{\ind}_{\kappa}(\underline{\sC}) \tkappa\big) \simeq \underline{\ind}_{\kappa}\underline{\sC} $. Here we have used crucially that $\underline{\ind}_{\kappa}(\underline{\sC}) \tkappa$ is really just fibrewise compact, that is, that the cocartesian lifts of the cocartesian fibration $\underline{\ind}_{\kappa}\underline{\sC}  \rightarrow \baseCat\op$ preserve $\kappa$-compact objects. This is because $\ind_{\kappa}(-)^{\kappa}$ computes the idempotent-completion by \cite[Lem. 5.4.2.4]{lurieHTT}, which is a functor.
    \item By the same token, $\underline{\sC} \rightarrow (\underline{\ind}_{\kappa}\underline{\sC}) \tkappa$ exhibits the $\baseCat$--idempotent-completion of $\underline{\sC}$ for any small $\baseCat$--category $\underline{\sC}$.
\end{enumerate}

\end{obs}

The following result will be crucial in the proof of  \cref{simpsonTheorem}.

\begin{prop}[$\baseCat$--Yoneda of idempotent-complete, ``{\cite[Prop. 5.3.4.18]{lurieHTT}}'']\label{yonedaOfIdempotentComplete}
Let $\underline{\sC}$ be a small $\baseCat$--{idempotent-complete} $\baseCat$--category which is $\kappa$--$\baseCat$--cocomplete. Then the $\baseCat$--Yoneda embedding $j : \underline{\sC} \rightarrow \underline{\presheaf}_{\baseCat}(\underline{\sC}) \tkappa$ has a $\baseCat$--left adjoint.
\end{prop}
\begin{proof}
By \cref{pointwiseConstructionOfAdjunction} we construct the adjunction objectwise. Let $\underline{\D} \subseteq \underline{\presheaf}_{\baseCat}(\underline{\sC})$ be the full subcategory generated by all presheaves $M$ where there exists $\ell M \in \underline{\sC}$ satisfying 
\[\myuline{\map}_{\underline{\presheaf}_{\baseCat}(\underline{\sC})}(M, j(-)) \simeq \myuline{\map}_{\underline{\sC}}(\ell M, -)\]
By definition, the desired left adjoint exists on this full subcategory, and hence it would suffice now to show that $\underline{\presheaf}_{\baseCat}(\underline{\sC})\tkappa\subseteq \underline{\D}$.

We first claim that $\underline{\D}$ is closed under retracts and inherits $\kappa$--$\baseCat$--cocompleteness from $\underline{\presheaf}_{\baseCat}(\underline{\sC})$. If $\myuline{\map}_{\underline{\presheaf}_{\baseCat}(\underline{\sC})}(N, j(-))$ is a retract of $\myuline{\map}_{\underline{\presheaf}_{\baseCat}(\underline{\sC})}(M, j(-))$ inside $\underline{\presheaf}_{\baseCat}(\underline{\sC})$. But then $\myuline{\map}_{\underline{\presheaf}_{\baseCat}(\underline{\sC})}(M, j(-))$ is in the Yoneda image from $\underline{\sC}$, which is idempotent-complete, and hence its retract is also in the Yoneda image. 

To see that $\underline{\D}\subseteq\underline{\presheaf}_{\baseCat}(\underline{\sC})$ inherits $\kappa$--$\baseCat$--cocompleteness, consider
\begin{equation*}
        \begin{split}
            \myuline{\map}_{\underline{\presheaf}_{\baseCat}(\underline{\sC})}(\underline{\colim}_{\underline{K}}M_k, j(-)) &\simeq \underline{\lim}_{\underline{K}\vop}\myuline{\map}_{\underline{\presheaf}_{\baseCat}(\underline{\sC})}(M_k, j(-))\\
            &\simeq \underline{\lim}_{\underline{K}\vop}\myuline{\map}_{\underline{\sC}}(\ell M_k, -)\\
            &\simeq \myuline{\map}_{\underline{\sC}}(\underline{\colim}_{\underline{K}}\ell M_k, -)
        \end{split}
\end{equation*}
where the last is since $\underline{\sC}$ is $\kappa$--$\baseCat$--cocomplete by hypothesis. 

Now \cref{TCompactsInTPresheaves} says that everything in $\underline{\presheaf}_{\baseCat}(\underline{\sC})$ is a retract of $\kappa$-small $\baseCat$--colimits of the Yoneda image $\underline{\sC}\subseteq\underline{\presheaf}_{\baseCat}(\underline{\sC})$. Hence, since $\underline{\sC}\subseteq \underline{\D}$ clearly, the paragraphs above yield that 
$\underline{\presheaf}_{\baseCat}(\underline{\sC})\tkappa \subseteq \underline{\D}$ as required.
\end{proof}

\section{Parametrised presentability}
\label{sec5:presentability}

We are now ready to formulate and prove two of the main results in this paper, namely the characterisations of $\baseCat$--presentables in \cref{simpsonTheorem} and the $\baseCat$--adjoint functor theorem,  \cref{parametrisedAdjointFunctorTheorem}. As we shall see, given all the technology that we have, the proofs for these parametrised versions will present us with no especial difficulties either because we can mimic the proofs of \cite{lurieHTT} almost word-for-word, or because we can deduce them from the unparametrised versions (as in the cases of the adjoint functor theorem or the presentable Dwyer-Kan localisation \cref{parametrisedPresentableDwyerKanLocalisation}). In subsections \cref{subsec5.3:DwyerKan} and \cref{subsec5.4:localisationCocompletions} we will also develop the important construction of \textit{localisation--cocompletions}. We will then prove the parametrised analogue of the correspondence between presentable categories and small idempotent-complete ones in \cref{TPresentableIdempotentCorrespondence} as well as record the  various expected permanence properties for parametrised presentability in \cref{subsec4.6:functorCategoriesAndPresentables} and \cref{indexedProductsPresentables}.

\subsection{Characterisations of parametrised presentability}

\begin{defn}
A  $\baseCat$--category $\underline{\sC}$ is $\baseCat$--\textit{presentable} if $\underline{\sC}$ is $\baseCat$--accessible and is $\baseCat$--cocomplete.
\end{defn}

We are now ready for the Lurie-Simpson-style characterisations of parametrised presentability. Note that characterisation (7) is a purely parametrised phenomenon and has no analogue in the unparametrised world. The proofs for the equivalences between the first six characterisations is exactly the arguments in \cite{lurieHTT} and so the expert reader might want to jump ahead to the parts that concern point (7).

\begin{thm}[Characterisations for parametrised presentability, ``{\cite[Thm. 5.5.1.1]{lurieHTT}}'']\label{simpsonTheorem}
Let $\underline{\sC}$ be a $\baseCat$--category. Then the following are equivalent:
\begin{enumerate}
    \item[(1)] $\underline{\sC}$ is $\baseCat$--presentable.
    \item[(2)] $\underline{\sC}$ is $\baseCat$--accessible, and for every regular cardinal $\kappa$, $\underline{\sC}\tkappa$ is $\kappa$--$\baseCat$--cocomplete.
    \item[(3)] There exists a regular cardinal $\kappa$ such that $\underline{\sC}$ is $\kappa$--$\baseCat$--accessible and $\underline{\sC}\tkappa$ is $\kappa$--$\baseCat$--cocomplete
    \item[(4)] There exists a regular cardinal $\kappa$, a small $\baseCat$--idempotent-complete and $\kappa$--$\baseCat$--cocomplete category $\underline{\D}$, and an equivalence $\underline{\ind}_{\kappa}\underline{\D} \rightarrow \underline{\sC}$. In fact, this $\underline{\D}$ can be chosen to be $\underline{\sC}\tkappa$.
    \item[(5)] There exists a small $\baseCat$--idempotent-complete category $\underline{\D}$ such that $\underline{\sC}$ is a $\kappa$--$\baseCat$--accessible Bousfield localisation of $\underline{\presheaf}_{\baseCat}(\underline{\D})$. By definition, this means that the image is $\kappa$--$\baseCat$--accessible, and so by \cref{accessibilityOfAdjoints} the $\baseCat$--right adjoint is also a $\kappa$--$\baseCat$--accessible functor and hence the Bousfield localisation preserves $\kappa$--$\baseCat$--compacts. 
    \item[(6)] $\underline{\sC}$ is locally small and is $\baseCat$--cocomplete, and there is a regular cardinal $\kappa$ and a small set $\sG$ of  T-$\kappa$-compact objects of $\underline{\sC}$ such that every $\baseCat$--object of $\underline{\sC}$ is a small $\baseCat$--colimit of objects in $\sG$.
    \item[(7)] $\underline{\sC}$ satisfies the left Beck-Chevalley condition (\cref{beckChevalleyMeaning}) and there is a regular cardinal $\kappa$ such that the straightening $C : \baseCat\op \longrightarrow \widehat{\cat} $ factors through  $C : \baseCat\op \longrightarrow \presentable_{L,\kappa}$.
\end{enumerate}
\end{thm}
\begin{proof}
That (1) implies (2) is immediate from \cref{compactClosure}. That (2) implies (3)  is because by definition of $\baseCat$--accessibility, there is a $\kappa$ such that $\underline{\sC}$ is $\kappa$--$\baseCat$--accessible, and since the second part of (2) says that $\underline{\sC}^{\myuline{\tau}}$ is $\tau$-$\baseCat$--cocomplete for all $\tau$, this is true in particular for $\tau = \kappa$ so chosen. To see (3) implies (4), note that accessibility is a fibrewise condition and so we can apply the characterisation of accessibility in \cite[Prop. 5.4.2.2 (2)]{lurieHTT}. To see (4) implies (5), let $\underline{\D}$ be given by (4). We want to show that $\underline{\sC}$ is a  $\baseCat$--accessible Bousfield localisation of $\underline{\presheaf}_{\baseCat}(\underline{\D})$. Consider the $\baseCat$--Yoneda embedding (it lands in $\kappa$--$\baseCat$--compacts by \cref{observationYonedaLandCompacts})
\[j : \underline{\D} \hookrightarrow\underline{\presheaf}_{\baseCat}(\underline{\D})\tkappa\] This has a $\baseCat$--left adjoint $\ell$ by \cref{yonedaOfIdempotentComplete}. Define $L \coloneqq \underline{\ind}_{\kappa}(\ell)$ and $J \coloneqq \underline{\ind}_{\kappa}(j)$, so that, since $\underline{\ind}_{\kappa}$ is a fibrewise construction, we have a $\baseCat$--adjunction by \cref{indAdjunctions}
\[L : \underline{\ind}_{\kappa}(\underline{\presheaf}_{\baseCat}(\underline{\D})\tkappa) \rightleftarrows \underline{\ind}_{\kappa}\underline{\D} : J\] where $J$ is $\baseCat$--fully faithful by \cref{TIndFullyFaithfulness}. But then by \cref{whenCocompletionsAreIndCompletions}, we get $\underline{\ind}_{\kappa}(\underline{\presheaf}_{\baseCat}(\underline{\D})\tkappa) \simeq \underline{\presheaf}_{\baseCat}(\underline{\D})$ and this completes this implication.

To see (5) implies (6), first of all $\underline{\presheaf}_{\baseCat}(\underline{\D})$ is locally small and so $\underline{\sC}\subseteq \underline{\presheaf}_{\baseCat}(\underline{\D})$ is too. Moreover, Bousfield local $\baseCat$--subcategories always admit $\baseCat$--colimits admitted by the ambient category and so $\underline{\sC}$ is $\baseCat$--cocomplete. For the last assertion, consider the composite 
\[\varphi : \underline{\D} \hookrightarrow \underline{\presheaf}_{\baseCat}(\underline{\D}) \xrightarrow{L} \underline{\sC}\] Since $\underline{\presheaf}_{\baseCat}(\underline{\D})$ is generated by $\underline{\D}$ under small $\baseCat$--colimits by \cref{TYonedaDensity} and since $L$ preserves $\baseCat$--colimits, we see that $\underline{\sC}$ is generated under $\baseCat$--colimits by $\im
\varphi.$ To see that $\im\varphi\subseteq \underline{\sC}\tkappa$, note that since by hypothesis $\underline{\sC}$ was $\kappa$--$\baseCat$--accessible, we know from \cref{accessibilityOfAdjoints} that the $\baseCat$--right adjoint of $L$ is automatically $\baseCat$--accessible, and so $L$ preserves $\kappa$--$\baseCat$--compacts, and we are done.

To see (6) implies (1), by definition, we just need to check that $\underline{\sC}$ is $\kappa$--$\baseCat$--accessible. Assumption (6) says that everything is a $\baseCat$--colimit of $\baseCat$--compacts, but we need to massage this to say that everything is a fibrewise $\kappa$-filtered $\baseCat$--colimit of an essentially small subcategory - note this is where we need the assumption about $\sG$ and not just use all of $\underline{\sC}\tkappa$, the problem being that the latter is not necessarily small. Let $\underline{\sC}'\subseteq \underline{\sC}\tkappa$ be generated by $\sG$ and $\underline{\sC}'\subseteq \underline{\sC}''\subseteq \underline{\sC}\tkappa$ be the $\kappa$--$\baseCat$--colimit closure of $\underline{\sC}'$: here we are using that $\underline{\sC}'' \subseteq \underline{\sC}\tkappa$ since $\kappa$--$\baseCat$--compacts are closed under $\kappa$-small $\baseCat$--colimits \cref{compactClosure}. Then since small $\baseCat$--colimits decompose as $\kappa$-small $\baseCat$--colimits and fibrewise $\kappa$-filtered colimits, we get that $\underline{\sC}$ is generated by $\underline{\sC}''\subseteq \underline{\sC}$ under $\kappa$-filtered colimits, as required.

Now to see (5) implies (7), suppose we have a  $\baseCat$--Bousfield localisation $F : \underline{\presheaf}_{\baseCat}(\underline{\sC})  \rightleftarrows \underline{\D} : G$. For $f : W \rightarrow V$ in $\baseCat$ we have 
\begin{center}
    \begin{tikzcd}
    \underline{\presheaf}_{\baseCat}(\underline{\sC}) _V = \func(\totalCategory(\underline{\sC}\vop\times\underline{V}), \spc) \dar["f^*"] \rar[shift left = 2, "F_V"] & \D_V\dar["f^*"]\lar[hook, "G_V"]\\
    \underline{\presheaf}_{\baseCat}(\underline{\sC}) _W = \func(\totalCategory(\underline{\sC}\vop\times\underline{W}), \spc) \ar[u, "f_!", shift left = 6]\rar[shift left = 2, "F_W"] \uar["f_*"', shift right = 6]& \D_W\lar[hook, "G_W"]\ar[u, "f_!", shift left = 5,dashed]\uar["f_*"', shift right = 5,dashed]
    \end{tikzcd}
\end{center}
where all the solid squares commute. We need to show a few things, namely:
\begin{itemize}
    \item That the dashed adjoints exist.
    \item That $f^* : \D_V\rightarrow \D_W$ preserves $\kappa$-compacts.
    \item That $f_! \dashv f^*$ on $\underline{\D}$ satisfies  the left Beck-Chevalley conditions. 
\end{itemize}
To see that the dashed arrows exist, define $f_!$ to be $F_V\circ f_!\circ G_W$. This works since
\begin{equation*}
    \begin{split}
        \map_{\D_V}(F_V\circ f_!\circ G_W-, -) & \simeq \map_{\underline{\presheaf}_{\baseCat}(\underline{\sC}) _W}(G_W-, f^*\circ G_V-)\\
        & \simeq\map_{\underline{\presheaf}_{\baseCat}(\underline{\sC}) _W}(G_W-, G_W\circ f^* -)\\
        &\simeq \map_{\D_W}(-, f^*-)
    \end{split}
\end{equation*}
To see that $f_*$ exists, we need to see that $f^*$ preserves ordinary colimits. For this, we use the description of colimits in Bousfield local subcategories. So let $\varphi : K \rightarrow \D_V$ be a diagram. Then
\begin{equation*}
    \begin{split}
        f^*\colim_{K\subseteq \D_V}\varphi & \simeq f^*F_V\big(\colim_{K\subseteq \presheaf_V}G_V\circ \varphi\big)\\
        & \simeq F_Wf^*\big(\colim_{K\subseteq \presheaf_V}G_V\circ \varphi\big)\\
        & \simeq F_W\big(\colim_{K\subseteq \presheaf_W}f^*\circ G_V\circ \varphi\big)\\
        &\simeq F_W\big(\colim_{K\subseteq \presheaf_W} G_W\circ f^*\circ \varphi\big)\\
        & =: \colim_{K\subseteq \D_W}f^*\circ \varphi
    \end{split}
\end{equation*}
And hence $f^*$ preserves colimits as required, and so by presentability, we obtain a right adjoint $f_*$. This completes the first point. Now to see that $f^* : \D_V\rightarrow \D_W$ preserves $\kappa$-compacts, note that $f^* : \underline{\presheaf}_{\baseCat}(\underline{\sC}) _V\rightarrow \underline{\presheaf}_{\baseCat}(\underline{\sC}) _W$ does since $f_* : \underline{\presheaf}_{\baseCat}(\underline{\sC}) _W\rightarrow \underline{\presheaf}_{\baseCat}(\underline{\sC}) _V$ is $\kappa$-accessible by \cref{accessibilityOfAdjoints}. Hence since $f^*F_V\simeq F_Wf^*$, taking right adjoints we get $f_*G_W\simeq G_Vf_*$. By hypothesis (5), $G$ was $\kappa$-accessible and so since it is also fully faithful fibrewise, we get that $f_* : \D_W\rightarrow \D_V$ is $\kappa$-accessible, as required. For the third point, we already know from \cref{colimitsInBousfieldLocals} that $\underline{\D}$ is $\baseCat$--cocomplete, and so $f_!$ must \textit{necessarily} give the indexed coproducts which satisfy the left Beck-Chevalley condition by \cref{omnibusTColimits}.

Finally to see (7) implies (1), \cref{omnibusTColimits} says that $\underline{\sC}$ is $\baseCat$--cocomplete, and so we are left to show that it is $\kappa$--$\baseCat$--accessible. But then this is just because $\underline{\sC} \simeq \underline{\ind}_{\kappa}(\underline{\sC}\tkappa)$ by \cite[Prop. 5.3.5.12]{lurieHTT}  (since parametrised-compacts and ind-completion is just fibrewise ordinary compacts/ind-completion because the straightening lands in $\presentable_{L,\kappa}$). This completes the proof for this step and for the theorem.
\end{proof}

\subsection{The adjoint functor theorem}

We now deduce the parametrised version of the adjoint functor theorem from the unparametrised version using characterisation (7) of \cref{simpsonTheorem}. Interestingly, and perhaps instructively, the proof shows us precisely where we need the notion of strong preservation and not just preservation (cf. \cref{definitionStrongPreservation} and the discussion in \cref{strongPreservationObservation}).

\begin{thm}[Parametrised adjoint functor theorem]\label{parametrisedAdjointFunctorTheorem}
Let $F : \underline{\sC} \rightarrow \underline{\D}$ be a $\baseCat$--functor between $\baseCat$--presentable categories. Then:
\begin{enumerate}
    \item[(1)] If $F$ strongly preserves $\baseCat$--colimits, then $F$ admits a $\baseCat$--right adjoint.
    \item[(2)] If $F$ strongly preserves $\baseCat$--limits and is $\baseCat$--accessible, then $F$ admits a $\baseCat$--left adjoint.
\end{enumerate}
\end{thm}
\begin{proof}
We want to apply \cref{criteriaForTAdjunctions}. To see (1), observe that the ordinary adjoint functor theorem gives us fibrewise right adjoints $F_V : \sC_V \rightleftarrows \D_V : G_V$. To see that this assembles to a $\baseCat$--functor $G$, we just need to check that the dashed square in the diagram 
\begin{center}
    \begin{tikzcd}
    \sC_V \ar[rr, shift left = 2, "F_V"] \dar[shift right = 2, "f_!"'] && \D_V \dar[shift right =2, "f_!"']\ar[ll, shift left =2 , "G_W",dashed]\\
    \sC_W \ar[rr, shift left = 2, "F_W"] \ar[u, shift right =2 , "f^*"',dashed]&& \D_W\ar[u, shift right =2 , "f^*"',dashed]\ar[ll, shift left =2 , "G_W",dashed]
    \end{tikzcd}
\end{center}
commutes. But then the left adjoints of the dashed compositions are the solid ones, which we know to be commutative by hypothesis that $F$ strongly preserves $\baseCat$--colimits (and so in particular indexed coproducts, see \cref{strongPreservationObservation}). Hence we are done for this case and part (2) is similar. 
\end{proof}

We will need the following characterisation of functors that strongly preserve $\baseCat$--colimits between $\baseCat$--presentables in order to understand the correspondence between $\baseCat$--presentable categories and small $\baseCat$--idempotent-complete ones.

\begin{prop}[``{\cite[Prop. 5.5.1.9]{lurieHTT}}'']\label{characterisationTLeftAdjointsPresentables}
Let $f : \underline{\sC} \rightarrow \underline{\D}$ be a $\baseCat$--functor between $\baseCat$--presentables and suppose $\underline{\sC}$ is $\kappa$--$\baseCat$--accessible. Then the following are equivalent:
\begin{itemize}
    \item[(a)] The functor $f$ strongly preserves $\baseCat$--colimits
    \item[(b)] The functor $f$ strongly preserves fibrewise $\kappa$-filtered colimits, and the restriction $f|_{\underline{\sC}\tkappa}$ strongly preserves $\kappa$--$\baseCat$--colimits.
\end{itemize}
\end{prop}
\begin{proof}
That (a) implies (b) is clear since $\underline{\sC}\tkappa \subseteq \underline{\sC}$ creates $\baseCat$--colimits by  \cref{compactClosure}. Now to see (b) implies (a), let $\underline{\sC} = \underline{\ind}_{\kappa}(\underline{\sC}\tkappa)$ where $\underline{\sC}\tkappa$ is $\kappa$--$\baseCat$--cocomplete and $\baseCat$--idempotent-complete category by \cref{compactClosure}. Now by the proof of (4) implies (5) in \cref{simpsonTheorem} we have a $\baseCat$--Bousfield adjunction
\[L : \underline{\presheaf}(\underline{\sC}\tkappa) \rightleftarrows \underline{\sC} : k\] Now consider the composite
\[j^*f : \underline{\sC}\tkappa \xrightarrow{\: j\:} \underline{\sC} \xrightarrow{\: f\:} \underline{\D}\] By the universal property of $\baseCat$--presheaves we get a strongly $\baseCat$--colimit-preserving functor $F$ fitting into the diagram
\begin{center}
    \begin{tikzcd}
    \underline{\sC}\tkappa \rar["j^*f"]\dar["j"', hook]\ar[dd, bend right = 70, "y"', hook] & \underline{\D}\\
    \underline{\sC}\dar["k"', hook]\ar[ur, "f"]\\
    \underline{\presheaf}_{\baseCat}(\underline{\sC}\tkappa)\ar[uur, "F\coloneqq y_!j^*f"']
    \end{tikzcd}
\end{center}
We know then that $f \simeq k^*y_!j^*f = k^*F$. On the other hand, we can define a functor 
\[F' \coloneqq f\circ L \simeq F\circ k\circ L : \underline{\presheaf}_{\baseCat}(\underline{\sC}\tkappa)\longrightarrow \underline{\sC} \longrightarrow \underline{\D}\]
The $\baseCat$--Bousfield adjunction unit $\id_{\underline{\presheaf}_{\baseCat}}\Rightarrow k\circ L$ gives us a natural transformation
\[\beta : F \Longrightarrow F' = F\circ k\circ L\] If we can show that $\beta$ is an equivalence then we would be done, since $F$, and so $F' = f\circ L$, strongly preserves $\baseCat$--colimits. Hence since $L$ was a $\baseCat$--Bousfield localisation, $f$ also strongly preserves $\baseCat$--colimits, as required.

To see that $\beta$ is an equivalence, let $\underline{\E}\subseteq \underline{\presheaf}_{\baseCat}(\underline{\sC}\tkappa)$ be the full $\baseCat$--subcategory on which $\beta$ is an equivalence. Since both $F$ and $F'$ strongly preserve fibrewise $\kappa$-filtered colimits, we see that $\underline{\E}$ is stable under such. Hence it suffices to show that  $\underline{\presheaf}_{\baseCat}(\underline{\sC}\tkappa)\tkappa \subseteq \underline{\E}$ since the inclusion will then induce the $\baseCat$--functor $\underline{\presheaf}_{\baseCat}(\underline{\sC}\tkappa) \simeq \underline{\ind}_{\kappa}(\underline{\presheaf}_{\baseCat}(\underline{\sC}\tkappa)\tkappa) \rightarrow \underline{\E}$ which is an equivalence by \cref{TIndFullyFaithfulness} (2).

Since $L\circ k \simeq \id$ we clearly have $\underline{\sC}\tkappa\subseteq  \underline{\E}$, ie. that $\beta : F \Rightarrow F'$ is an equivalence on $\underline{\sC}\tkappa\subseteq  \underline{\presheaf}_{\baseCat}(\underline{\sC}\tkappa)$. On the other hand, by \cref{compactClosure} we know that  $\underline{\presheaf}_{\baseCat}(\underline{\sC}\tkappa)\tkappa$ is $\kappa$--$\baseCat$--cocomplete, and its objects are retracts of $\kappa$-small $\baseCat$--colimits valued in $\underline{\sC}\tkappa \subseteq \underline{\presheaf}_{\baseCat}(\underline{\sC}\tkappa)$ by \cref{TCompactsInTPresheaves}. Thus it suffices to show that $F$ and $F'$ strongly preserve $\kappa$-small $\baseCat$--colimits when restricted to $\underline{\presheaf}_{\baseCat}(\underline{\sC}\tkappa)\tkappa$. That $F$ does is clear since it in fact strongly preserves all small $\baseCat$--colimits. That $F'$ does is because it can be written as the composition
\[F'|_{\underline{\presheaf}_{\baseCat}(\underline{\sC}\tkappa)\tkappa} : \underline{\presheaf}_{\baseCat}(\underline{\sC}\tkappa)\tkappa \xrightarrow{L} \underline{\sC}\tkappa \xrightarrow{f} \underline{\D}\] where $L$ is a $\baseCat$--left adjoint and $f$ strongly preserves $\kappa$-small $\baseCat$--colimits by assumption. Here we have crucially used that $L$ lands in $\underline{\sC}\tkappa$ since this category is $\baseCat$--idempotent-complete and $\kappa$--$\baseCat$--cocomplete.
\end{proof}

\subsection{Dwyer-Kan localisations}\label{subsec5.3:DwyerKan}
\begin{terminology}\label{nota:dwyerKanLocalisations}
We recall the clarifying terminology of \cite{hinich} in distinguishing between Bousfield localisations, as defined in \cref{bousfieldLocalisationDefinition}, and \textit{Dwyer-Kan localisations}. By the latter, we will mean the following: let $\underline{\sC}$ be a $\baseCat$--category and $S$ a class of morphisms in $\underline{\sC}$ such that $f^*(S_W) \subseteq S_V$ for all $f \colon V \rightarrow W$ in $\baseCat$. Suppose a $\baseCat$--category $S^{-1}\underline{\sC}$ exists and is equipped with a map $f : \underline{\sC} \rightarrow S^{-1}\underline{\sC}$ inducing the equivalence
\[f^* : \underline{\func}_{\baseCat}(S^{-1}\underline{\sC}, \underline{\D}) \xrightarrow{\simeq} \underline{\func}_{\baseCat}^{S^{-1}}(\underline{\sC}, \underline{\D})\] for all $\baseCat$--categories $\underline{\D}$, where $\underline{\func}_{\baseCat}^{S^{-1}}(\underline{\sC}, \underline{\D}) \subseteq \underline{\func}_{\baseCat}(\underline{\sC}, \underline{\D})$ is the full subcategory of parametrised functors sending morphisms in $S$ to equivalences. Such a $\baseCat$--category must necessarily be unique if it exists, and this is then defined to be \textit{the $\baseCat$--Dwyer-Kan localisation of} $\underline{\sC}$ \textit{with respect to} $S$. The following proposition shows that being a $\baseCat$--Bousfield localisation is stronger than that of being a $\baseCat$--Dwyer-Kan localisation.
\end{terminology}

\begin{prop}[Bousfield implies Dwyer-Kan]\label{bousfieldLocalisationsAreDwyerKan}
Let $\underline{\sC}, L\underline{\sC}$ be $\baseCat$--categories and $L : \underline{\sC} \rightleftarrows L\underline{\sC} : i$ be a $\baseCat$--Bousfield localisation. Let $S$ be the collection of morphisms in $\underline{\sC}$ that are sent to equivalences under $L$. Then the functor $L$ induces an equivalence
$L^* : \underline{\func}_{\baseCat}(L\underline{\sC}, \underline{\D}) \xrightarrow{\simeq} \underline{\func}_{\baseCat}^{S^{-1}}(\underline{\sC}, \underline{\D})$
for any $\baseCat$--category $\underline{\D}$ so that $L\underline{\sC}$ is a Dwyer-Kan localisation against $S$.
\end{prop}
\begin{proof}
Since $L \dashv i$ was a $\baseCat$--Bousfield localisation, we know that $i^* : \underline{\func}_{\baseCat}(L\underline{\sC}, \underline{\D}) \rightleftarrows \underline{\func}_{\baseCat}(\underline{\sC}, \underline{\D}) : L^*$ is also a $\baseCat$--Bousfield localisation by \cref{omnibusTAdjunctions}, and so in particular $L^*$ is $\baseCat$--fully faithful. The image of $L^*$ also clearly lands in $\underline{\func}_{\baseCat}^{S^{-1}}(\underline{\sC}, \underline{\D})$, and so we are left to show $\baseCat$--essential surjectivity. By basechanging if necessary, we just show this on $\func_{\baseCat}^{S^{-1}}(\underline{\sC}, \underline{\D})$. Let $\varphi : \underline{\sC} \rightarrow \underline{\D}$ be a $\baseCat$--functor that inverts morphisms in $S$. We aim to show that $\varphi \Rightarrow \varphi\circ i\circ L$ is an equivalence. Since $L\dashv i$ was a $\baseCat$--Bousfield localisation, the unit $\eta : \id \Rightarrow  i\circ L$ gets sent to an equivalence under $L$, and so $\eta \in S$. Since $\varphi$ inverts $S$ by assumption, in particular it inverts $\eta$.
\end{proof}

\begin{prop}\label{presentablesAreTComplete}\label{TPresentablesAreTComplete}
$\baseCat$--presentable categories are $\baseCat$--complete.
\end{prop}
\begin{proof}
Let $\underline{\sC}$ be $\baseCat$--presentable so that it is a $\baseCat$--Bousfield localisation of some $\baseCat$--presheaf category $ \underline{\presheaf}_{\baseCat}(\underline{\D})$ by description (5) of \cref{simpsonTheorem}. We know that $\underline{\presheaf}_{\baseCat}(\underline{\D})$ is $\baseCat$--complete and so all we need to show is that $\baseCat$--Bousfield local subcategories are closed under $\baseCat$--limits which exist in the ambient category. But this is clear since $\baseCat$--Bousfield local subcategories can be described by a mapping-into property.
\end{proof}

\begin{terminology}
For $S \subseteq \underline{\sC}$ a collection of morphisms, an object $X \in \underline{\sC}$ is said to be $S$--local if ${\map}_{\underline{\sC}}(-,X)$ sends morphisms in $S$ to equivalences.
\end{terminology}

In fact, as in the unparametrised case, we can give a precise description of maps that get inverted in a Bousfield localisation against an arbitrary collection of morphisms $S$, generalising the usual theory available for instance in \cite[$\S5.5.4$]{lurieHTT}.

\begin{defn}\label{defn:GSaturationAxioms}
Let $\overline{S}$ be a $\baseCat$--collection of morphisms in a ${\baseCat}$--category $\underline{\sC}$. We say that it is \textit{${\baseCat}$--strongly saturated} if the following conditions are satisfied:
\begin{enumerate}
    \item (Pushout closure) Suppose we have a fibrewise pushout square in $\underline{\sC}$
    \begin{center}
        \begin{tikzcd}
        A \rar \dar \ar[dr, phantom, very near end, "\scalebox{1.5}{$\ulcorner$}"]& B \dar \\
        C \rar & D.
        \end{tikzcd}
    \end{center}
    If the left vertical is in $\overline{S}$, then the right vertical is also in $\overline{S}$.
    
    \item (${\baseCat}$-colimit closure) The ${\baseCat}$-full subcategory $\underline{\func}^{\overline{S}}(\Delta^1, \underline{\sC}) \subseteq \underline{\func}(\Delta^1, \underline{\sC})$ of morphisms in $\overline{S}$ is closed under ${\baseCat}$-colimits.
    
    \item (2-out-of-3) If any two of the three morphisms in 
    \begin{center}
        \begin{tikzcd}
        A \rar\ar[rr, bend left = 30] & B\rar & C
        \end{tikzcd}
    \end{center}
    are in $\overline{S}$, then the third one is too.
\end{enumerate}
For any $\baseCat$--collection of morphisms $S$, we define its \textit{$\baseCat$--strong saturation closure} $\overline{S}$ to be the smallest ${\baseCat}$--collection containing $S$ which is ${\baseCat}$--strongly saturated.
\end{defn}

\begin{prop}\label{strongSaturationYoga}
Let $\underline{\sC}$ be a ${\baseCat}$-presentable category and $S$ a $\baseCat$--collection of morphisms in $\underline{\sC}$, and $\overline{S}$ its ${\baseCat}$-strong saturation. Let $L : \underline{\sC} \rightarrow L_S\underline{\sC}$ be the ${\baseCat}$-Bousfield localisation at $S$. Then the collection of $L$-equivalences consists precisely of the collection $\overline{S}$.
\end{prop}
\begin{proof}
We will bootstrap the parametrised statement from the unparametrised version in \cite[Prop. 5.5.4.15]{lurieHTT}. Let $T$ be the collection of $L$-equivalences. First of all, note that we have  $\overline{S} \subseteq T$ since it is straightforward to check that $T$ is a $\baseCat$--strongly saturated collection containing $S$ and $\overline{S}$ is by definition the minimal such collection. To see the reverse inclusion, let $f : X \rightarrow Y$ be an $L$-equivalence and  consider the square
\begin{center}
    \begin{tikzcd}
    X \rar["f"] \dar & Y \dar\\
    LX \rar["Lf"', "\simeq"] & LY
    \end{tikzcd}
\end{center}
Now since a $\baseCat$-Bousfield localisation is in particular a fibrewise Bousfield localisation, we can apply \cite[Prop. 5.5.4.15 (1)]{lurieHTT} to see that the vertical maps in the square are in $\overline{S}$. And hence by 2-out-of-3, we see that $f$ was also in $\overline{S}$, as desired.
\end{proof}

The following result, which will be crucial for our application in \cite{kaifNoncommMotives}, is another example of the value of characterisation (7) from \cref{simpsonTheorem}. The proof of the unparametrised result, given by Lurie in \cite[$\S5.5.4$]{lurieHTT}, is long and technical, and characterisation (7) allows us to obviate this difficulty by bootstrapping from Lurie's statement.

\begin{thm}[Parametrised presentable Dwyer-Kan localisations]\label{parametrisedPresentableDwyerKanLocalisation}
Let $\underline{\sC}$ be a $\baseCat$-presentable category and $S$ a small collection of $\:{\baseCat}$-morphisms of $\underline{\sC}$ (ie. if $f : V\rightarrow W$ in $\baseCat$ and $y \rightarrow z$ a morphism in $S_W$, then $f^*y\rightarrow f^*z$ is in $S_V$). Then: 
\begin{enumerate}
    \item[(1)] Writing $S_{\underline{\amalg}}\supset S$ for the closure of $S$ under finite indexed coproducts, the fibrewise full subcategory $S^{-1}_{\underline{\amalg}}\underline{\sC}\subseteq \underline{\sC}$  of $S_{\underline{\amalg}}$-local objects assembles to a $\baseCat$--full subcategory.
    
    \item[(2)] We have a  $\baseCat$--accessible $\baseCat$--Bousfield localisation
$L : \underline{\sC} \rightleftarrows S^{-1}_{\underline{\amalg}}\underline{\sC} : i$.

\item[(3)] For any $\baseCat$--category $\underline{\D}$, the $\baseCat$--functor $L^* : \underline{\func}^L_{\baseCat}(S^{-1}_{\underline{\amalg}}\underline{\sC}, \underline{\D}) \longrightarrow \underline{\func}^{L, \overline{S}^{-1}}_{\baseCat}(\underline{\sC}, \underline{\D})$ is an equivalence. Moreover, the inclusion $\underline{\func}^{L, \overline{S}^{-1}}_{\baseCat}(\underline{\sC}, \underline{\D})\subseteq \underline{\func}^{L, S^{-1}}_{\baseCat}(\underline{\sC}, \underline{\D})$ is an equivalence.
\end{enumerate}
\end{thm}
\begin{proof}
For (1), we just nee to show that $S^{-1}_{\underline{\amalg}}\underline{\sC}$ is closed under the restriction functors in $\underline{\sC}$. Let $f\colon V\rightarrow W$ be a map in $\baseCat$ and let $x\in (S^{-1}_{\underline{\amalg}}\underline{\sC})_W$. We need to show that $f^*x\in (S^{-1}_{\underline{\amalg}}\underline{\sC})_V$ is again $S_{\underline{\amalg}}$--local. So let $\varphi \colon a \rightarrow b$ be a morphism in $S_{\underline{\amalg}}$. Because $S_{\underline{\amalg}}$ is closed under finite indexed coproducts, we have the equivalence
\[{\map}_V(b,f^*x) \simeq {\map}_W(f_!b,x) \xrightarrow[\simeq]{(f_!\varphi)^*}{\map}_W(f_!a,x)\simeq {\map}_V(a,f^*x)\]
as wanted.
For (2), we know from \cite[Prop. 5.5.4.15]{lurieHTT} that we already have fibrewise Bousfield localisations, and all we need to do is show that these assemble to a $\baseCat$--Bousfield localisation via \cref{criteriaForTAdjunctions}. Let $f : V \rightarrow W$ be in $\baseCat$. We need to show that
\begin{center}
    \begin{tikzcd}
    \sC_V \rar["L_V"] & S^{-1}_{\underline{\amalg}}\sC_V \\
    \sC_W \rar["L_W"]\uar["f^*"] & S^{-1}_{\underline{\amalg}}\sC_W\uar["f^*"']
    \end{tikzcd}
\end{center}
commutes, and for this, we first note that the diagram
\begin{center}
    \begin{tikzcd}
    \sC_V \dar["f_*"'] & S^{-1}_{\underline{\amalg}}\sC_V \lar["i_V"', hook]\dar["f_*"]\\
    \sC_W & S^{-1}_{\underline{\amalg}}\sC_W\lar["i_W", hook]
    \end{tikzcd}
\end{center}
commutes where here $f_*$ exists since $\underline{\sC}$ is $\baseCat$--complete by \cref{presentablesAreTComplete}. Now recall by definition that $f^*(S_W)\subseteq S_V$ and so
for $y\rightarrow z$ in $S_W$ the map
\[\map_{W}(z, f_*x) \simeq \map_{V}(f^*z, x) \longrightarrow \map_{V}(f^*y, x) \simeq \map_{V}(y,f_*x)\] is an equivalence, which implies that $f_*$ takes $S_{\underline{\amalg}}$-local objects to $S_{\underline{\amalg}}$-local objects. Now by uniqueness of left adjoints, the first diagram commutes, as required. Finally, the first sentence of (3) is just a consequence of \cref{bousfieldLocalisationsAreDwyerKan} and \cref{strongSaturationYoga}, noting also that $\overline{S}=\overline{S}_{\underline{\amalg}}$. That the inclusion is an equivalence is because if $F\in \underline{\func}^{L, S^{-1}}_{\baseCat}(\underline{\sC}, \underline{\D})$ and we have a morphism in $\overline{S}$ of the form $\underline{\colim}_{\underline{J}}\varphi\colon \underline{\colim}_{\underline{J}}a\rightarrow \underline{\colim}_{\underline{J}}b$ for some $\underline{J}$--indexed diagram of morphisms all landing in $S$, then $F\underline{\colim}_{\underline{J}}\varphi\simeq \underline{\colim}_{\underline{J}}F\varphi$ is an equivalence by the hypothesis on $F$.
\end{proof}

\subsection{Localisation--cocompletions}\label{subsec5.4:localisationCocompletions}
In this subsection we formulate and prove the construction of \textit{localisation--cocompletions} whose proof is exactly analogous to that of \cite{lurieHTT}. As far as we can see, unfortunately the proof cannot be bootstrapped from the unparametrised statement as with the proof of \cref{parametrisedPresentableDwyerKanLocalisation} because the notion of a parametrised collection of diagrams might involve diagrams that are not fibrewise in the sense of \cref{importantClassOfColimits}.

\begin{defn}[Parametrised collection of diagrams]\label{parametrisedCollectionOfDiagrams}
Let $\underline{\sC}\in\cat_{\baseCat}$. A \textit{parametrised collection of diagrams in } $\underline{\sC}$ is defined to be a triple $(\underline{\sC}, \K, \R)$ where:
\begin{itemize}
    \item $\K$ is a collection of small categories parametrised over $\baseCat\op$, ie. a collection $\K_V$ of small $\baseCat_{/V}$--categories for each $V\in\baseCat$. 
    \item $\R$ is a parametrised collection of diagrams in $\underline{\sC}$ whose indexing categories belong to $\K$, ie. for each $V\in \baseCat$ a collection of coconed diagrams $\R_V$ indexed over categories in $\K_V$.
\end{itemize}
\end{defn}

\begin{thm}[$\baseCat$--localisation--cocompletions, ``{\cite[Prop. 5.3.6.2]{lurieHTT}}'']\label{arbitraryCocompletions}\label{nota:localisationCocompletions}
Let $(\underline{\sC}, \K, \R)$ be a parametrised collection of diagrams in $\underline{\sC}$. Then there is a $\baseCat$--category $\underline{\presheaf}^{\K}_{\R}(\underline{\sC}) $ and a $\baseCat$--functor $j : \underline{\sC} \rightarrow\underline{\presheaf}^{\K}_{\R}(\underline{\sC}) $ such that:
\begin{enumerate}
    \item The category $\underline{\presheaf}^{\K}_{\R}(\underline{\sC}) $ is $\K$--$\baseCat$--cocomplete, ie. it strongly admits $\K$--indexed $\baseCat$--colimits, $\underline{\sC}_{\underline{V}}$ admits $K$--indexed $\baseCat_{/V}$--colimits.
    \item For every $\K$-$\baseCat$--cocomplete category $\underline{\D}$, the map $j$ induces an equivalence of $\baseCat$--categories
    \[j^* : \underline{\func}_{\baseCat}^{\K}(\underline{\presheaf}^{\K}_{\R}(\underline{\sC}) ,\underline{\D}) \longrightarrow \underline{\func}_{\baseCat}^{\R}(\underline{\sC},\underline{\D})\] where the source denotes the $\baseCat$--category of functors which strongly preserve $\K$--indexed colimits and the target consists of those functors carrying each diagram in $\R$ to a parametrised colimit diagram in $\underline{\D}$.
    \item If each member of $\R$ were already a $\baseCat$--colimit diagram in $\underline{\sC}$, then in fact $j$ is $\baseCat$--fully faithful.
\end{enumerate}
\end{thm}
\begin{proof}
We give first all the constructions. By enlarging the universe, if necessary, we may reduce to the case where:
\begin{itemize}
    \item Every element of $\K$ is small
    \item That $\underline{\sC}$ is small
    \item The collection of diagrams $\R$ is small
\end{itemize}
Let $y : \underline{\sC} \hookrightarrow \underline{\presheaf}_{\baseCat}(\underline{\sC}) $ be the $\baseCat$--yoneda embedding and let $V\in \baseCat$.  For a $\baseCat_{/V}$--diagram $\Bar{p} : K\tcocone \rightarrow \underline{\sC}_{\underline{V}}$ with cone point $Y$, let $X$ denote the $\baseCat_{/V}$--colimit of $y\circ p : K \rightarrow \underline{\presheaf}_{\baseCat}(\underline{\sC}) _{\underline{V}}$. This induces a $\baseCat_{/V}$--morphism in $\underline{\presheaf}_{\baseCat}(\underline{\sC}) _{\underline{V}}$ 
\[s : X \rightarrow y(Y)\]  Here we have used that $\underline{\presheaf}_{\baseCat}(\underline{\sC}) _{\underline{V}}\simeq \underline{\presheaf}_{\underline{V}}(\underline{\sC}_{\underline{V}})$ by \cref{TFunctorCategory}. Now let $S$ be the set of all such $\baseCat_{/V}$--morphisms running over all $V\in \baseCat$. This is small by our assumption and so let $L : \underline{\presheaf}_{\baseCat}(\underline{\sC})  \rightarrow S^{-1}\underline{\presheaf}_{\baseCat}(\underline{\sC}) $ denote the $\baseCat$--Bousfield localisation from \cref{parametrisedPresentableDwyerKanLocalisation}.  Now we define $\underline{\presheaf}_{\R}^{\K}(\underline{\sC})  \subseteq S^{-1}\underline{\presheaf}_{\baseCat}(\underline{\sC}) $ to be the smallest $\K$--cocomplete full $\baseCat$--subcategory containing the image of $L\circ y : \underline{\sC} \rightarrow \underline{\presheaf}_{\baseCat}(\underline{\sC})  \rightarrow S^{-1}\underline{\presheaf}_{\baseCat}(\underline{\sC}) $. We show that this works and prove each point in turn.

Point (i) is true by construction, and so there is nothing to do. For point (ii), let $\underline{\D}$ be $\K$--$\baseCat$--cocomplete. We now perform a reduction to the case when $\underline{\D}$ is $\baseCat$--cocomplete. By taking the opposite Yoneda embedding we see that $\underline{\D}$ sits $\baseCat$--fully faithfully in a $\baseCat$--cocomplete category $\underline{\D}'$ and the inclusion strongly preserves $\K$--colimits. We now have a square of $\baseCat$--categories (where the vertical functors are $\baseCat$--fully faithful by \cref{functorsPreserveTFullyFaithfulness}) 
    \begin{center}
        \begin{tikzcd}
        \underline{\func}_{\baseCat}^{\K}(\underline{\presheaf}^{\K}_{\R}(\underline{\sC}) ,\underline{\D}) \rar["\phi \coloneqq j^*"] \dar[hook] & \underline{\func}_{\baseCat}^{\R}(\underline{\sC},\underline{\D}) \dar[hook]\\
       \underline{\func}_{\baseCat}^{\K}(\underline{\presheaf}^{\K}_{\R}(\underline{\sC}) ,\underline{\D}') \rar["\phi' \coloneqq j^*"] & \underline{\func}_{\baseCat}^{\R}(\underline{\sC},\underline{\D}') 
        \end{tikzcd}
    \end{center}
    We claim this is cartesian in $\widehat{\cat}_{\baseCat}$ if $\phi'$ were an equivalence: given this, to prove that $\phi$ is an equivalence, it suffices to prove that $\phi'$ is an equivalence.  For this, we need to show that the map into the pullback is an equivalence. That $\phi'$ is an equivalence ensures that the map into the pullback is fully faithful. To see essential surjectivity, let $F : \underline{\presheaf}^{\K}_{\R}(\underline{\sC})  \rightarrow \underline{\D}'$ be a strongly $\K$--colimit preserving functor that restricts to $\underline{\sC} \rightarrow \underline{\D}$. Then in fact $F$ lands in $\underline{\D}\subseteq \underline{\D}'$ since $\underline{\D} \subseteq \underline{\D}'$ is stable under $\K$--indexed colimits, and by construction, $\underline{\presheaf}^{\K}_{\R}(\sC) $ is generated under $\K$--indexed colimits by $\underline{\sC}$. 
    
    Now we turn to showing $\phi$ is an equivalence in the case $\underline{\D}$ is $\baseCat$--cocomplete. Let $\underline{\E} \subseteq \underline{\presheaf}_{\baseCat}(\underline{\sC}) $ be the inverse image $L^{-1}\underline{\presheaf}^{\K}_{\R}(\underline{\sC}) $ and $\overline{S}$ be the collection of all morphisms $\alpha$ in $\underline{\E}$ such that $L\alpha$ is an equivalence. Since the $\baseCat$--Bousfield localisation $L : \underline{\presheaf}(\underline{\sC})  \rightleftarrows S^{-1}\underline{\presheaf}(\underline{\sC})  : i$ induces a $\baseCat$--Bousfield localisation $L : \underline{\E} \rightleftarrows \underline{\presheaf}^{\K}_{\R}(\underline{\sC})  : i$ we see by \cref{bousfieldLocalisationsAreDwyerKan} that $L^* : \underline{\func}_{\baseCat}(\underline{\presheaf}^{\K}_{\R}(\underline{\sC}) ,\underline{\D}) \rightarrow \underline{\func}_{\baseCat}^{\overline{S}^{-1}}(\underline{\E},\underline{\D})$ is  an equivalence. Furthermore, by using the description of colimits in $\baseCat$--Bousfield local subcategories as being given by applying the localisation $L$ to the colimit in the ambient category, we see that $f : \underline{\presheaf}^{\K}_{\R}(\underline{\sC})  \rightarrow \underline{\D}$  strongly preserves $\K$--colimits if and only if $f\circ L : \underline{\E} \rightarrow \underline{\D}$ does. This gives us the following factorisation of $\phi$ 
    \[\phi : \underline{\func}_{\baseCat}^{\K}(\underline{\presheaf}^{\K}_{\R}(\underline{\sC}) ,\underline{\D}) \xrightarrow[\simeq]{L^*} \underline{\func}_{\baseCat}^{\overline{S}^{-1}, \K}(\underline{\E},\underline{\D}) \xrightarrow{j^*} \underline{\func}_{\baseCat}^{\R}(\underline{\sC},\underline{\D}) \] and hence we need to show that the functor $j^*$ is an equivalence. Since $\underline{\D}$ is $\baseCat$--cocomplete, we can consider the $\baseCat$--adjunction
    $j_! : \underline{\func}_{\baseCat}^{\R}(\underline{\sC},\underline{\D})\rightleftarrows \underline{\func}_{\baseCat}^{\K}(\underline{\E},\underline{\D}) : j^*$. We need to show:
    \begin{itemize}
        \item that $j_!$ lands in $\underline{\func}_{\baseCat}^{\overline{S}^{-1},\K}(\underline{\E}, \underline{\D})$,
        \item that $j_!\circ j^* \simeq \id$ on $\underline{\func}_{\baseCat}^{\overline{S}^{-1},\K}(\underline{\E},\underline{\D})$ and $j^*\circ j_! \simeq \id.$
    \end{itemize}
    For the first point, fix a $V\in \baseCat$. Since relative adjunctions are closed under pullbacks by \cref{relativeAdjunctionPullbacks} and since $\underline{\func}_{\baseCat}(\underline{\sC}, \underline{\D})_{\underline{V}}\simeq \underline{\func}_{\underline{V}}(\underline{\sC}_{\underline{V}}, \underline{\D}_{\underline{V}})$ by \cref{TFunctorCategory}, we also get a $\baseCat_{/V}$-adjunction 
    $j_! : \underline{\func}_{\underline{V}}^{\R_{\underline{V}}}(\underline{\sC}_{\underline{V}},\underline{\D}_{\underline{V}})\rightleftarrows \underline{\func}_{\underline{V}}^{\K_{\underline{V}}}(\underline{\E}_{\underline{V}},\underline{\D}_{\underline{V}}) : j^*$. Suppose $F : \underline{\sC}_{\underline{V}} \rightarrow \underline{\D}$ is a ${\underline{V}}$--functor that sends $\R_{\underline{V}}$ to ${\underline{V}}$--colimit diagrams. We want to show that $j_!F : \underline{\E}_{\underline{V}} \rightarrow \underline{\D}_{\underline{V}}$ inverts maps in $\overline{S}$, ie. those maps that get inverted by $L_{\underline{V}}$. Consider
    \begin{center}
        \begin{tikzcd}
        \underline{\sC}_{\underline{V}} \dar[hook, "j"']\ar[dr,"F"]\ar[dd, bend right = 50, "y"', hook] \\
        \underline{\E}_{\underline{V}} \dar[hook, "k"'] & \underline{\D}_{\underline{V}}\\
        \underline{\presheaf}_{\underline{V}}(\underline{\sC}_{\underline{V}})\ar[ur, "y_!F"'] 
        \end{tikzcd}
    \end{center}
Note that $j_!F \simeq k^*y_!F$ since $j_!=\id\circ j_!\simeq k^*k_!j_! \simeq k^*y_!$. Now since $y_!F : \underline{\presheaf}_{\underline{V}}(\underline{\sC}_{\underline{V}}) \rightarrow \underline{\D}_{\underline{V}}$ strongly preserves ${\underline{V}}$--colimits and since $\underline{\E}_{\underline{V}}$ is stable under $\K$--indexed colimits in $\underline{\presheaf}_{\underline{V}}(\underline{\sC}_{\underline{V}})$ (since $\underline{\presheaf}^{\K_{\underline{V}}}_{\R_{\underline{V}}}(\underline{\sC}_{\underline{V}})$ was closed under $\K$--colimits by construction) it follows that $j_!F \simeq k^*y_!F$ strongly preserves $\K$--colimits. Now note that the maps in $S \subseteq \underline{\presheaf}_{\underline{V}}(\underline{\sC}_{\underline{V}})$ are inverted by $y_!F$ since these were the maps comparing colimit in $\underline{\presheaf}_{\underline{V}}(\underline{\sC}_{\underline{V}})$ and cone point in $\underline{\sC}_{\underline{V}}$, and by hypothesis, $F$, and hence $y_!F$ turns these into equivalences. Therefore, by the universal property of Dwyer-Kan localisations \cref{parametrisedPresentableDwyerKanLocalisation}, $y_!F : \underline{\presheaf}_{\underline{V}}(\underline{\sC}_{\underline{V}}) \rightarrow \underline{\D}$ factors through the Bousfield localisation $L$, and so in particular inverts $\overline{S}$, so that $j_!F \simeq k^*y_!F$ does too. Also $y_!F$ strongly preserves all ${\underline{V}}$--colimits by the universal property of presheaves, and so $j_!F \simeq k^*y_!F$ strongly preserves $\K$--colimits since the inclusion $k : \underline{\E}_{\underline{V}} \hookrightarrow \underline{\presheaf}_{\underline{V}}(\underline{\sC}_{\underline{V}})$ does. 

For the second point, since $j$ was $\baseCat$--fully faithful, we have that $j^*\circ j_!\simeq \id$ as usual by \cref{fullyFaithfulTKanExtension}. For the equivalence $j_!\circ j^* \simeq \id$, suppose $F \in \underline{\func}_{\baseCat}^{\overline{S}^{-1},\K}(\underline{\E},\underline{\D})$. Write $F' \coloneqq j_!j^*F$. By universal property of Kan extensions we have $\alpha : F' = j_!j^*F \rightarrow F$ and we want to show this is an equivalence. Since $F$ inverts $\overline{S}$ by hypothesis and $j_!j^*F$ also inverts $\overline{S}$ by the claim of the previous paragraph, we get the diagram
    \begin{center}
        \begin{tikzcd}
        \underline{\sC}_{\underline{V}} \dar[hook, "j"'] \\
        \underline{\E}_{\underline{V}} \ar[dr,"L_{\underline{V}}"']\ar[r, "F"]\rar[bend left = 50, "F'"] & \underline{\D}_{\underline{V}}\\
        & \underline{\presheaf}^{\K_{\underline{V}}}_{\R_{\underline{V}}}(\underline{\sC}_{\underline{V}})\uar["f'"]\uar[bend right = 50, "f"']
        \end{tikzcd}
    \end{center}
    The transformation $\alpha$ induces a transformation $\beta : f' \rightarrow f$  since $\underline{\func}_{\underline{V}}^{ \overline{S}^{-1}}(\underline{\E}_{\underline{V}},\underline{\D}_{\underline{V}}) \simeq \underline{\func}_{\underline{V}}(\underline{\presheaf}^{\K_{\underline{V}}}_{\R_{\underline{V}}}(\underline{\sC}_{\underline{V}}),\underline{\D}_{\underline{V}})$ and we want to show that $\beta$ is an equivalence. To begin with, note that it is an equivalence on the image of the embedding $j : \underline{\sC}_{\underline{V}} \hookrightarrow \underline{\presheaf}^{\K_{\underline{V}}}_{\R_{\underline{V}}}(\underline{\sC}_{\underline{V}})$. Since $F$ and $F'$ strongly preserve $\K$--colimits, hence so do $f'$ and $f$. Therefore, $f' \rightarrow f$ is an equivalence on all of $\underline{\presheaf}^{\K_{\underline{V}}}_{\R_{\underline{V}}}(\underline{\sC}_{\underline{V}})$ since this ${\underline{V}}$-category was by construction generated under these colimits by $\underline{\sC}$. This completes the proof of point (ii). 
    
    Finally, for point (iii), suppose every element of $\R$ were already a colimit diagram in $\underline{\sC}$. The Yoneda map  can be factored, by construction, as $j : \underline{\sC} \hookrightarrow \underline{\E} \xrightarrow{L} \underline{\presheaf}^{\K}_{\R}(\underline{\sC})$     where the first map is $\baseCat$--fully faithful. Since the restriction $L|_{S^{-1}\underline{\presheaf}_{\baseCat}(\underline{\sC}) } \simeq \id$, it will suffice to show that $j$ lands in $S^{-1}\underline{\presheaf}_{\baseCat}(\underline{\sC})$. That is, that $\underline{\sC}$ is $S$-local, ie. for each $V\in \baseCat$ and $C\in \sC_V$, and for each $f : W\rightarrow V$ in $\baseCat$ and $s : X \rightarrow jY$ in $S_W$, we need to see that 
    \[s^* : \map_{\underline{\presheaf}_{\baseCat}(\underline{\sC}) _W}(jY, jf^*C) \longrightarrow \map_{\underline{\presheaf}_{\baseCat}(\underline{\sC}) _W}(X,jf^*C) \] is an equivalence. To see this, the hypothesis of (iii) gives $Y = \underline{\colim}_{K\subseteq \sC_{\underline{W}}} \varphi$. Then
    \small
    \[\myuline{\map}_{\underline{\presheaf}_{\baseCat}(\underline{\sC}) _{\underline{W}}}(jY, jf^*C)  \simeq 
    \myuline{\map}_{\underline{\sC}_{\underline{W}}}(\underline{\colim}_{K\subseteq \sC_{\underline{W}}} \varphi, f^*C)\simeq \underline{\lim}_{K\vop\subseteq \sC_{\underline{W}}{}\vop}\map_{\sC}(\varphi, f^*C)\]
    \normalsize
    where the first equivalence is by Yoneda. On the other hand,
    \begin{equation*}
        \begin{split}
            \myuline{\map}_{\underline{\presheaf}_{\baseCat}(\underline{\sC}) _{\underline{W}}}(X, jf^*C) & \simeq \myuline{\map}_{\underline{\presheaf}_{\baseCat}(\underline{\sC}) _{\underline{W}}}(\underline{\colim}_Kj\circ \varphi, jf^*C)\\ & \simeq \underline{\lim}_{K\vop}\myuline{\map}_{\underline{\presheaf}_{\baseCat}(\underline{\sC}) _{\underline{W}}}(j\varphi, jf^*C) \\ & \simeq \underline{\lim}_{K\vop}\myuline{\map}_{\underline{\sC}_{\underline{W}}}(\varphi, f^*C)
        \end{split}
    \end{equation*}
    and so taking the section over $W$, one checks that these two identifications are compatible with the map $s^*$. This completes the proof of (iii).
\end{proof}

\subsection{The presentables--idempotents equivalence}
We want to formulate the equivalence between presentables and idempotent-completes in the parametrised world, and so we need to introduce some definitions. To avoid potential confusion, we will for example use the terminology \textit{parametrised}-accessibles instead of $\baseCat$--accessibles to indicate that we take $\baseCat_{/V}$--accessibles in the fibre over $V$.

\begin{defn}
Let $\kappa$ be a regular cardinal.
\begin{itemize}
    \item Let $\underline{\accessible}_{\baseCat,\kappa} \subset \underline{\widehat{\cat}}_{\baseCat}$ \label{nota:TAccessibles} be the non-full $\baseCat$--subcategory of $\kappa$-parametrised-accessible categories and $\kappa$-parametrised-accessible functors preserving $\kappa$-parametrised-compacts. 
    \item Let $\underline{\cat}_{\baseCat}^{\underline{\idem}} \subseteq \underline{\widehat{\cat}}_{\baseCat}$ \label{nota:parametrisedIdempotentCompleteCategories} be the full $\baseCat$--subcategory on the small parametrised-idempotent-complete categories. 
    
    \item Let $\underline{{\cat}}_{\baseCat}^{\underline{\mathrm{rex}}(\kappa)}\subset \underline{\widehat{\cat}}_{\baseCat}$ be the non-full subcategory whose objects are $\kappa$-parametrised--cocomplete small categories and morphisms those parametrised--functors that strongly preserve $\kappa$-small parametrised--colimits.

    \item Let $\underline{\cat}_{\baseCat}^{\underline{\idem}(\kappa)}\subseteq \underline{{\cat}}_{\baseCat}^{\underline{\mathrm{rex}}(\kappa)}$\label{nota:parametrisedIdempotentCompleteCategoriesKappa} be the full subcategory whose objects are $\kappa$-parametrised--cocomplete small parametrised-idempotent-complete categories.
    
    \item Let $\underline{\presentable}_{\baseCat, L, \kappa} \subset \underline{\accessible}_{\baseCat,\kappa}$ \label{nota:TCategoryOfTPresentables} be the non-full $\baseCat$--subcategory whose objects are parametrised-presentables and whose morphisms are parametrised-left adjoints that preserve $\kappa$-parametrised-compacts.
    
    \item Let $\underline{\presentable}_{\baseCat, R,\kappa\operatorname{-filt}} \subset \underline{\widehat{\cat}}_{\baseCat}$ be the non-full $\baseCat$--subcategory of parametrised presentable categories and morphisms the parametrised $\kappa$-accessible functors which strongly preserve parametrised limits.
\end{itemize}
\end{defn}

\begin{nota}
Let $\underline{\func}_{\baseCat}\tkappa\subseteq \underline{\func}_{\baseCat}$ be the full subcategory of $\kappa$--$\baseCat$--compact-preserving functors.
\end{nota}

\begin{lem}[``{\cite[Prop. 5.4.2.17]{lurieHTT}}'']\label{TAccessibleScaffolding}
Let $\kappa$ be a regular cardinal. Then $(-)\tkappa : \underline{\accessible}_{\baseCat,\kappa} \longrightarrow \underline{\widehat{\cat}}_{\baseCat}$ induces an equivalence to $\underline{\cat}_{\baseCat}^{\underline{\idem}}$, whose inverse $\underline{\cat}_{\baseCat}^{\underline{\idem}}\rightarrow \underline{\accessible}_{\baseCat,\kappa}$ is $\underline{\ind}_{\kappa}$. 
\end{lem}
\begin{proof}
To see $\baseCat$--fully faithfulness, \cref{UnivPropInd} gives
\begin{equation*}
    \begin{split}
        \underline{\func}_{\baseCat}^{\kappa\operatorname{-filt},\underline{\kappa}}(\underline{\ind}_{\kappa}\underline{\sC} , \underline{\ind}_{\kappa}\underline{\D}) &\xrightarrow{\simeq} \underline{\func}_{\baseCat}\tkappa(\underline{\ind}_{\kappa}(\underline{\sC})\tkappa,\underline{\ind}_{\kappa}\underline{\D})\\
        &\xrightarrow{\simeq} \underline{\func}_{\baseCat}(\underline{\ind}_{\kappa}(\underline{\sC})\tkappa,\underline{\ind}_{\kappa}(\underline{\D})\tkappa)
    \end{split}
\end{equation*}
where we have also used, by \cref{consequencesOfFibrewiseDefinitions} (1), that $\underline{\ind}_{\kappa}(\underline{\ind}_{\kappa}(\underline{\sC})\tkappa)\simeq \underline{\ind}_{\kappa}\underline{\sC} $. As for the essential image, let $\underline{\sC}$ be a small $\baseCat$--idempotent-complete category. Then by \cref{consequencesOfFibrewiseDefinitions} (2) we know that $\underline{\sC} \simeq \underline{\ind}_{\kappa}(\underline{\sC})\tkappa$, and so it is in the essential image as required. Finally to see the statement about the inverse, just note that we already have the functors and the appropriate natural transformations on compositions. Then using \cref{consequencesOfFibrewiseDefinitions} again, we see that the transformations are pointwise equivalences, and so equivalences.
\end{proof}

\begin{thm}[$\baseCat$--presentable-idempotent correspondence, ``{\cite[Prop. 5.5.7.8 and Rmk. 5.5.7.9]{lurieHTT}}'']\label{TPresentableIdempotentCorrespondence}
Let $\kappa$ be a regular cardinal. Then 
$(-)\tkappa : \underline{\presentable}_{\baseCat,L,\kappa} \longrightarrow \underline{\widehat{\cat}}_{\baseCat}^{\underline{\mathrm{rex}}(\kappa)}$ is $\baseCat$--fully faithful with essential image $\underline{\cat}_{\baseCat}^{\underline{\idem}(\kappa)}$, and inverse $\underline{\cat}_{\baseCat}^{\underline{\idem}(\kappa)}\rightarrow \underline{\presentable}_{\baseCat,L,\kappa}$ given by $\underline{\ind}_{\kappa}$. 
\end{thm}
\begin{proof}
That it is $\baseCat$--fully faithful with the specified essential image is by \cref{TAccessibleScaffolding} together with \cref{compactClosure} and \cref{characterisationTLeftAdjointsPresentables}. That the inverse from \cref{TAccessibleScaffolding} via $\underline{\ind}_{\kappa}$ lands in $\baseCat$--presentables is by \cref{simpsonTheorem} (4).
\end{proof}

\subsection{Indexed products of presentables}\label{indexedProductsPresentables}
The purpose of this subsection is to show that the (non-full) inclusions $\underline{\presentable}_{\baseCat, L,\kappa}, \underline{\presentable}_{\baseCat, R,\kappa\operatorname{-filt}} \subset \underline{\widehat{\cat}}_{\baseCat}$ create indexed products.

\begin{lem}[Indexed products of $\baseCat$--presentables]\label{indexedProductsOfTPresentables}
Let $f : W \rightarrow V$ be in $\baseCat$ and $\underline{\sC}$ be a $\baseCat_{/W}$--presentable category. Then $f_*\underline{\sC}$ is a $\baseCat_{/V}$--presentable category. 
\end{lem}
\begin{proof}
We first note that if $\underline{\D}$ is a $\baseCat_{/W}$--category, then $f_*\underline{\func}_{\underline{W}}(\underline{\D}, \underline{\spc}_{\underline{W}}) \simeq \underline{\func}_{\underline{V}}(f_!\underline{\D}, \underline{\spc}_{\underline{V}}).$ To see this, let $\underline{\E}$ be a $\baseCat_{/V}$--category. Then 
\begin{equation*}
    \begin{split}
        \map_{\cat_{\baseCat_{/V}}}(\underline{\E}, f_*\underline{\func}_{\underline{W}}(\underline{\D}, \underline{\spc}_{\underline{W}})) & \simeq \map_{\cat_{\baseCat_{/W}}}(f^*\underline{\E}, \underline{\func}_{\underline{W}}(\underline{\D}, \underline{\spc}_{\underline{W}}))\\
        &\simeq \map_{\cat_{\baseCat_{/W}}}(\underline{\D}, \underline{\func}_{\underline{W}}(f^*\underline{\E}, \underline{\spc}_{\underline{W}}))\\
        &\simeq \map_{\cat_{\baseCat_{/W}}}(\underline{\D}, f^*\underline{\func}_{\underline{V}}(\underline{\E}, \underline{\spc}_{\underline{V}}))\\
        &\simeq \map_{\cat_{\baseCat_{/V}}}(f_!\underline{\D}, \underline{\func}_{\underline{V}}(\underline{\E}, \underline{\spc}_{\underline{V}}))\\
        & \simeq \map_{\cat_{\baseCat_{/V}}}(\underline{\E}, \underline{\func}_{\underline{V}}(f_!\underline{\D}, \underline{\spc}_{\underline{V}}))
    \end{split}
\end{equation*}
By \cref{simpsonTheorem} we have a accessible $\baseCat_{/W}$--Bousfield localisation $\underline{\func}_{\underline{W}}(\underline{\D}, \underline{\spc}_{\underline{W}}) \rightleftarrows \underline{\sC}$ for some small $\baseCat_{/W}$--category $\underline{\D}$. Hence by \cref{indexedConstructionsPreserveAdjunctions}, we obtain the accessible adjunction
\begin{center}
    \begin{tikzcd}
        \underline{\func}_{\underline{V}}(f_!\underline{\D}, \underline{\spc}_{\underline{V}}) \simeq f_*\underline{\func}_{\underline{W}}(\underline{\D}, \underline{\spc}_{\underline{W}}) \rar[shift left = 1] & f_*\underline{\sC} \lar[hook, shift left = 1]
    \end{tikzcd}
\end{center}
Therefore, $f_*\underline{\sC}$ must be $\baseCat_{/V}$--presentable, again by \cref{simpsonTheorem}.
\end{proof}

\begin{prop}[Creation of indexed products for presentables]\label{creationIndexedProductsPresentables}
The (non-full) inclusions $\underline{\presentable}_{\baseCat, L,\kappa}, \underline{\presentable}_{\baseCat, R,\kappa\operatorname{-filt}} \subset \underline{\widehat{\cat}}_{\baseCat}$ create indexed products.
\end{prop}
\begin{proof}
Let $f : W \rightarrow V$ be in $\baseCat$ and $\underline{\sC}, \underline{\D}$ be $\baseCat_{/V}$-- and $\baseCat_{/W}$--presentables, respectively. We know from \cref{unitCounitIndexedProductOfCategories} that $\underline{\widehat{\cat}}_{\baseCat}$ has indexed products. We need to show that 
\[\map^L_{\underline{V}}(\underline{\sC},  f_*\underline{\D}) \simeq \map^L_{\underline{W}}(f^*\underline{\sC},  \underline{\D})\]\[\map^{R,\kappa\operatorname{-filt}}_{\underline{V}}(\underline{\sC},  f_*\underline{\D}) \simeq \map^{R,\kappa\operatorname{-filt}}_{\underline{W}}(f^*\underline{\sC},  \underline{\D})\]
We claim that the unit and counit in $\underline{\widehat{\cat}}_{\baseCat}$ are already in both $\underline{\presentable}_{\baseCat, L,\kappa}$ and $\underline{\presentable}_{\baseCat, R,\kappa\operatorname{-filt}}$. If we can show this then we would be done by the following pair of diagrams
\begin{center}
    \begin{tikzcd}
    \map_{\underline{V}}^L(\underline{\sC},  f_*\underline{\D}) \rar["f^*"] \dar[hook]& \map_{\underline{W}}^L(f^*\underline{\sC},  f^*f_*\underline{\D})\dar[hook] \rar["\varepsilon_*"] & \map_{\underline{W}}^L(f^*\underline{\sC},  \underline{\D})\dar[hook]\\
    \map_{\underline{V}}(\underline{\sC},  f_*\underline{\D}) \rar["f^*"] \ar[rr, bend right = 15, "\simeq"]& \map_{\underline{W}}(f^*\underline{\sC},  f^*f_*\underline{\D}) \rar["\varepsilon_*"] & \map_{\underline{W}}(f^*\underline{\sC},  \underline{\D})
    \end{tikzcd}
\end{center}
\begin{center}
    \begin{tikzcd}
    \map^L_{\underline{V}}(\underline{\sC}, f^*\underline{\D})  \dar[hook]& \map^L_{\underline{V}}(f_*f^*\underline{\sC},  f^*\underline{\D})\lar["\eta^*"']  \dar[hook]& \map^L_{\underline{W}}(f^*\underline{\sC}, \underline{\D}) \lar["f_*"']\dar[hook]\\
    \map_{\underline{V}}(\underline{\sC}, f^*\underline{\D})   & \map_{\underline{V}}(f_*f^*\underline{\sC},  f^*\underline{\D})\lar["\eta^*"']  & \map_{\underline{W}}(f^*\underline{\sC}, \underline{\D}) \lar["f_*"']\ar[ll, bend left = 15, "\simeq"']
    \end{tikzcd}
\end{center}
and similarly when we replace $\map^L$ by $\map^{R,\kappa\operatorname{-filt}}$: that the (co)units are in $\underline{\presentable}_{\baseCat, R,\kappa\operatorname{-filt}}$ and $\underline{\presentable}_{\baseCat, L,\kappa}$ imply that the maps $\varepsilon_*$ and $\eta^*$ above takes $\map^L$ to $\map^L$; that $f^*$ and $f_*$ also do these is by \cref{indexedConstructionsPreserveAdjunctions}; and finally the bottom equivalences are inverse to each other, and so restrict to inverse equivalences to the top row of each diagram. 

We now prove the claims. That they preserve $\kappa$--$\baseCat$--compact objects is clear by \cref{unitCounitIndexedProductOfCategories} and \cref{simpsonTheorem}. To see that the counit $\varepsilon : f^*f_*\underline{\D} \rightarrow \underline{\D}$ strongly preserves $\baseCat$--(co)limits, since it is clear that they preserve fibrewise $\baseCat$--(co)limits, by \cref{characterisationStrongPreservations} we are left to show that they preserve the indexed (co)products. So let $\xi : Y\rightarrow Z$ be in $\baseCat_{/W}$. For this we will need to know that $\underline{\D}$ has indexed coproducts and products (for the latter, see \cref{TPresentablesAreTComplete}).  We need to show that the squares with the dashed arrows in 
\begin{equation}\label{counitSquare}
    \begin{tikzcd}
    (f^*f_*\underline{\D})_Z \rar["\varepsilon"]\dar["\xi^*"] & \underline{\D}_Z\dar["\xi^*"]\\
    (f^*f_*\underline{\D})_Y \rar["\varepsilon"]\uar[bend left = 50, dashed, "\xi_!"]\uar[bend right = 50, dashed, "\xi_*"'] & \underline{\D}_Y\uar[bend left = 50, dashed, "\xi_!"]\uar[bend right = 50, dashed, "\xi_*"'] 
    \end{tikzcd}
\end{equation}
commute. We analyse this in terms of the counit formula from \cref{unitCounitIndexedProductOfCategories}. For this, consider the diagram of orbits
\begin{equation}\label{comparisonOfPullbacks}
    \begin{tikzcd}
        & \coprod_{b}R_b\ar[dd]\ar[ddrr, phantom, "\scalebox{1.5}{$\lrcorner$}", very near start] \ar[rr] \ar[dl,"\coprod_{b}\xi_{a_b}"']\ar[ddrr, phantom, "\scalebox{1.5}{$\lrcorner$}", very near start]&& Y \ar[dl, "\xi"]\ar[dd]\\
        \coprod_aS_a \ar[ddrr, phantom, "\scalebox{1.5}{$\lrcorner$}", very near start]\ar[dd]\ar[rr, crossing over]&& Z\\
        & W \ar[rr, "f"{xshift=-15pt}]\ar[dl,equal] && V\ar[dl,equal]\\
        W \ar[rr, "f"] && V  \ar[uu, leftarrow,crossing over]
    \end{tikzcd}
\end{equation}
where the top square is also a pullback since we can view this diagram as 
\begin{center}
    \begin{tikzcd}
    \coprod_bR_b \dar\rar & \coprod_aS_a \rar \dar& W \dar["f"]\\
    Y \rar & Z \rar & V
    \end{tikzcd}
\end{center}
with the right square and the outer rectangle being pullbacks. From this we obtain that the diagram \cref{counitSquare} is equivalent to 
\begin{center}
    \begin{tikzcd}
    \prod_a\D_{S_a} \rar["\pi_Z"]\dar["\prod_b\xi^*_{a_b}"'] & \D_Z \dar["\xi^*"]\\
    \prod_b\D_{R_b} \rar["\pi_Y"]\uar[bend left = 80, dashed, "\xi_!"]\uar[bend right = 50, dashed, "\xi_*"'] & \D_Y\uar[bend left = 50, dashed, "\xi_!"]\uar[bend right = 50, dashed, "\xi_*"']
    \end{tikzcd}
\end{center}
where the counits have been identified with the projections $\pi_Z$ (resp. $\pi_Y$) onto the $\underline{\D}_Z$ (resp. $\underline{\D}_Y$) components by virtue of \cref{unitCounitIndexedProductOfCategories}. Here $\prod_b\xi_{a_b}^*$ is supposed to mean forgetting about the components of $\coprod_a{S_a}$ that do not receive a map from $\coprod_bR_b$ and the functor $\xi^*_{a_b}$ for the other components: this makes sense because an orbit in a coproduct can only map to a unique orbit. Since $\underline{\sC}$ was $\baseCat_{/W}$--presentable, it in particular admits an $\baseCat_{/W}$--initial object. And so we can easily use these, together with the adjoints $(\xi_{a_b})_!$ and fibrewise coproducts to obtain a left adjoint $\xi_!$ of $\prod_b\xi^*_{a_b}$, and similarly a right adjoint $\xi_*$. It is then immediate that the dashed squares also commute since the counits just project left/right adjoints from the left vertical to those on the right.

To see that the unit strongly preserves $\baseCat$--(co)limits, similarly as above, we are reduced to the case of showing that it preserves indexed (co)products. Let $\zeta : U \rightarrow X$ be in $\baseCat_{/V}$. And so we want the squares with the dashed arrows 
\begin{center}
    \begin{tikzcd}
    \underline{\sC}_X \dar["\zeta^*"']\rar["\eta"] & (f_*f^*\underline{\sC})_X \dar["\zeta^*"]\\
    \underline{\sC}_U \rar["\eta"] \uar[bend left = 50, dashed, "\zeta_!"]\uar[bend right = 50, dashed, "\zeta_*"']& (f_*f^*\underline{\sC})_U \uar[bend left = 50, dashed, "\zeta_!"]\uar[bend right = 50, dashed, "\zeta_*"']
    \end{tikzcd}
\end{center}
to commute. For this consider the pullback comparison
\begin{center}
    \begin{tikzcd}
        & \coprod_{b}M_b\ar[dd]\ar[ddrr, phantom, "\scalebox{1.5}{$\lrcorner$}", very near start] \ar[rr] \ar[dl,"\coprod_{b}\zeta_{a_b}"']\ar[ddrr, phantom, "\scalebox{1.5}{$\lrcorner$}", very near start]&& U \ar[dl, "\zeta"]\ar[dd]\\
        \coprod_aN_a \ar[dd]\ar[rr, crossing over]\ar[ddrr, phantom, "\scalebox{1.5}{$\lrcorner$}", very near start]&& X \\
        & W \ar[rr, "f" {xshift=-15pt}]\ar[dl,equal] && V\ar[dl,equal]\\
        W \ar[rr, "f"] && V\ar[uu, crossing over, leftarrow]
    \end{tikzcd}
\end{center}
where the top square is also a pullback by the argument for the previous case. Since
\[(f_*f^*\underline{\sC})_X = \prod_a\sC_{N_a} \quad\text{ and }\quad (f_*f^*\underline{\sC})_U = \prod_b\sC_{M_b}\] we see that the units $\eta$ arise as restrictions along the maps $\coprod_aN_a \rightarrow X$ and $\coprod_bM_b \rightarrow U$ respectively. Then the required dashed squares commute by the Beck-Chevalley property of indexed (co)products of $\underline{\sC}$ associated to the top pullback square. This completes the proof.
\end{proof}

\subsection{Functor categories and tensors of presentables}\label{subsec4.6:functorCategoriesAndPresentables}
In this final subsection, we record several basic results about the interaction between parametrised-presentability and functor categories, totally analogous to the unparametrised setting.

\begin{lem}[Small cotensors preserve $\baseCat$--presentability]\label{smallCotensorPreservePresentability}
Let $\underline{\sC}$ be a small $\:{\baseCat}$-category and $\underline{\D}$ be $\baseCat$--presentable. Then $\underline{\func}_{\baseCat}(\underline{\sC}, \underline{\D})$ is also $\baseCat$--presentable.
\end{lem}
\begin{proof}
As a special case, suppose first that $\underline{\D} \simeq \underline{\presheaf}_{\baseCat}(\underline{\D}')$ for a small $\baseCat$--category $\underline{\D}'$. Then $\underline{\func}_{\baseCat}(\underline{\sC}, \underline{\D}) \simeq \underline{\func}_{\baseCat}(\underline{\sC} \times {\underline{\D}'}\vop, \underline{\spc}_{\baseCat})$, and so it is also a $\baseCat$--presheaf category, and so is $\baseCat$--presentable. For a general $\baseCat$--presentable $\underline{\D}$, we know that we have a $\kappa$--$\baseCat$--accessible Bousfield localisation
$L : \underline{\presheaf}_{\baseCat}(\underline{\D}') \rightleftarrows \underline{\D} : i$ for some small $\baseCat$--category $\underline{\D}'$. Then we get a $\kappa$--$\baseCat$--accessible Bousfield localisation
$L_* : \underline{\func}_{\baseCat}(\underline{\sC}, \underline{\presheaf}_{\baseCat}(\underline{\D}'))  \rightleftarrows \underline{\func}_{\baseCat}(\underline{\sC}, \underline{\D}) : i_*$ and so since $\underline{\func}_{\baseCat}(\underline{\sC}, \underline{\presheaf}_{\baseCat}(\underline{\D}'))$ was $\baseCat$--presentable by the first part above, by characterisation \cref{simpsonTheorem} (5) we get that $\underline{\func}_{\baseCat}(\underline{\sC}, \underline{\D})$ is too.
\end{proof}

\begin{lem}[``{\cite[Lem. 5.5.4.17]{lurieHTT}}'']\label{htt5.5.4.17}
Let $F : \underline{\sC} \rightleftarrows \underline{\D} : G$ be a $\baseCat$--adjunction between $\baseCat$--presentables. Suppose we have a $\baseCat$--accessible Bousfield localisation $L : \underline{\sC} \rightleftarrows \underline{\sC}^0 : i$. Let $\underline{\D}^0 \coloneqq G^{-1}(\underline{\sC}^0)\subseteq \underline{\D}$. Then we have a $\baseCat$--accessible Bousfield localisation $L' : \underline{\D} \rightleftarrows \underline{\D}^0 : i'$.
\end{lem}
\begin{proof}
The $\baseCat$--accessibility of the Bousfield localisation $L : \underline{\sC} \rightleftarrows \underline{\sC}^0 :i$ ensures that there is a small set of morphisms of $\underline{\sC}$ such that $\underline{\sC}^0$ are precisely the $S$-local objects. Then it is easy to see that $\underline{\D}^0 \subseteq \underline{\D}$ is precisely the $F(S)$-local $\baseCat$--subcategory by using the adjunction.
\end{proof}

\begin{lem}[``{\cite[Lem. 5.5.4.18]{lurieHTT}}'']\label{htt5.5.4.18}
Let $\underline{\sC}$ be a $\baseCat$--presentable category and $\{\underline{\sC}_a\}_{a\in A}$ be a family of $\baseCat$--accessible Bousfield local subcategories indexed by a small set $A$. Then $\bigcap_{a\in A}\underline{\sC}_a$ is also a $\baseCat$--accessible Bousfield local subcategory.
\end{lem}
\begin{proof}
This is because, if we write $S(a)$ for the morphisms of $\underline{\sC}$ such that $\underline{\sC}_a$ is the $S(a)$-local objects, then $\bigcap_{a\in A}\underline{\sC}_a$ are the $\bigcup_{a\in A}S(a)$-local objects.
\end{proof}

For the remaining results, recall from \cref{LFunRFunNotations} that $\underline{\func}_{\baseCat}^R$ and $\underline{\func}_{\baseCat}^L$ denote strongly $\baseCat$--limit- and $\baseCat$--colimit-preserving functors, respectively, and $\underline{\rfunc}_{\baseCat}$ and $\underline{\lfunc}_{\baseCat}$ denote $\baseCat$--right and $\baseCat$--left adjoint functors, respectively.

\begin{lem}[Presentable functor categories, ``{\cite[Lem. 4.8.1.16]{lurieHA}}'']\label{presentableFunctorCategories}
Let $\underline{\sC}, \underline{\D}$ be $\baseCat$--presentables. Then $\underline{\func}_{\baseCat}^R(\underline{\sC}\vop, \underline{\D})$ and $\underline{\func}^L_{\baseCat}(\sC, \D)$ are also $\baseCat$--presentable.
\end{lem}
\begin{proof}
By characterisation (5) of \cref{simpsonTheorem} and that Bousfield localisations are Dwyer-Kan \cref{bousfieldLocalisationsAreDwyerKan}, we know that $\underline{\sC} \simeq S^{-1}\underline{\presheaf}_{\baseCat}(\underline{\sC}')$ for some small $\baseCat$--category $\underline{\sC}'$ and $S$ a small collection of morphisms in $\underline{\presheaf}_{\baseCat}(\underline{\sC}')$.  Then we have
\begin{equation*}
    \begin{split}
        \underline{\func}^R_{\baseCat}(\underline{\presheaf}_{\baseCat}(\underline{\sC}')\vop, \underline{\D}) \simeq \underline{\func}^L_{\baseCat}(\underline{\presheaf}_{\baseCat}(\underline{\sC}'), \underline{\D}\vop) \vop &\simeq \underline{\func}_{\baseCat}(\underline{\sC}', \underline{\D}\vop) \vop\\
        &\simeq \underline{\func}_{\baseCat}({\underline{\sC}'}\vop, \underline{\D})
    \end{split}
\end{equation*}
where the first and last equivalence is by \cref{opOfFunctorCats}, and the second by \cref{UnivPropInd} and since $\baseCat$--presentables are also $\baseCat$--complete by \cref{presentablesAreTComplete}. The right hand term is $\baseCat$--presentable by \cref{smallCotensorPreservePresentability}, and so  $\underline{\func}^R_{\baseCat}(\underline{\presheaf}_{\baseCat}(\underline{\sC}')\vop, \underline{\D})$ is too by the equivalence above. Now note that we have $\underline{\func}^R_{\baseCat}(\underline{\sC}\vop, \underline{\D}) \simeq \underline{\func}^{R, S^{-1}}_T(\underline{\presheaf}_{\baseCat}(\underline{\sC}')\vop, \underline{\D})$: this is by virtue of the following diagram
\begin{equation*}
    \begin{split}
        \underline{\func}^R_{\baseCat}((S^{-1}\underline{\presheaf}_{\baseCat}(\underline{\sC}'))\vop, \underline{\D}) & \simeq  \underline{\func}^L_{\baseCat}(S^{-1}\underline{\presheaf}_{\baseCat}(\underline{\sC}'), \underline{\D}\vop)\vop\\
        & \xrightarrow[\simeq]{L^*} \underline{\func}_{\baseCat}^{L, S^{-1}}(\underline{\presheaf}_{\baseCat}(\underline{\sC}'), \underline{\D}\vop)\vop\\
        &\simeq \underline{\func}_{\baseCat}^{R, S^{-1}}(\underline{\presheaf}_{\baseCat}(\underline{\sC}')\vop, \underline{\D})
    \end{split}
\end{equation*}
where we have the equivalence $L^*$ owing to the formula for $\baseCat$--colimits in $\baseCat$--Bousfield local subcategories. Therefore, if for each $\alpha \in S$ we write $\underline{\E}(\alpha)\subseteq \underline{\func}^R_{\baseCat}(\underline{\presheaf}_{\baseCat}(\underline{\sC}')\vop, \underline{\D})$ to be the $\baseCat$--full subcategory of those functors which carry $\alpha$ to an equivalence in $\underline{\D}$, then $\underline{\func}_{\baseCat}^R(\underline{\sC}\vop, \underline{\D}) \simeq \bigcap_{\alpha\in S}\underline{\E}(\alpha) \subseteq \underline{\func}_{\baseCat}^R(\underline{\presheaf}_{\baseCat}(\underline{\sC}')\vop, \underline{\D})$. Hence to show $\underline{\func}^R(\underline{\sC}\vop, \underline{\D})$ is a $\baseCat$--accessible Bousfield localisation of $\underline{\func}_{\baseCat}^R(\underline{\presheaf}_{\baseCat}(\underline{\sC}')\vop, \underline{\D})$, it will be enough to show it, by \cref{htt5.5.4.18}, for each $\underline{\E}(\alpha)$. Now these $\alpha$'s are morphisms in the various fibres over $\baseCat\op$ but since everything interacts well with basechanges, we can just assume without loss of generality that $\baseCat\op$ has an initial object and that $\alpha$ is a morphism in the fibre of this initial object. Given this, it is clear that we have the pullback
\begin{center}
    \begin{tikzcd}
    \underline{\E}(\alpha) \rar\dar \ar[dr, phantom, very near start, "\scalebox{1.5}{$\lrcorner$}"]& \underline{\func}^R_{\baseCat}(\underline{\presheaf}_{\baseCat}(\underline{\sC}')\vop, \underline{\D}) \dar["\eval_{\alpha}"]\\
    \underline{\E} \rar[hook] & \underline{\func}_{\baseCat}(\tconstant_\baseCat(\Delta^1), \underline{\D})
    \end{tikzcd}
\end{center}
where $\underline{\E}$ is the full subcategory spanned by the equivalences. Hence by \cref{htt5.5.4.17} it will suffice to show that $\underline{\E} \subseteq \underline{\func}_{\baseCat}(\tconstant_\baseCat(\Delta^1), \underline{\D})$ is a $\baseCat$--accessible Bousfield localisation. But this is clear since it is just given by the $\baseCat$--left Kan extension along $\ast \rightarrow \Delta^1$.

The statement for $\underline{\func}_{\baseCat}^L(\underline{\sC}, \underline{\D})$ is proved analogously, but without having to take opposites in showing that  $L^* : \underline{\func}^L_{\baseCat}(S^{-1}\underline{\presheaf}_{\baseCat}(\underline{\sC}'), \underline{\D}) \rightarrow \underline{\func}_{\baseCat}^{L, S^{-1}}(\underline{\presheaf}_{\baseCat}(\underline{\sC}'), \D) $ is an equivalence.
\end{proof}

The following result was stated as Example 3.26 in \cite{nardinThesis} without proof, and so we prove it here. Here the tensor product is the one constructed in \cite[$\S3.4$]{nardinThesis}.

\begin{prop}[Formula for presentable $\baseCat$--tensors]\label{TTensorsOfPresentables} Let $\baseCat$ be an atomic orbital category, and let $\underline{\sC}, \underline{\D}$ be $\baseCat$--presentable categories. Then 
$\underline{\sC}\otimes \underline{\D} \simeq \underline{\func}_{\baseCat}^{R}(\underline{\sC}\vop,\underline{\D})$. 
\end{prop}
\begin{proof}
This is just a consequence of the universal property of the tensor product. To wit, let $\underline{\E}$ be an arbitrary $\baseCat$--presentable category and write $\underline{\func}_{\baseCat}^{R, \mathrm{acc}}$ for $\baseCat$--accessible strongly $\baseCat$--limit preserving functors. Then
\begin{equation*}
     \begin{split}
         \underline{\func}^{L, L}(\underline{\sC} \times  \underline{\D},\underline{\E}) &\simeq \underline{\func}^L(\underline{\sC}, \underline{\func}^L(\underline{\D},\underline{\E}))\\
         &\simeq \underline{\func}^R_{\baseCat}(\underline{\sC}\vop, \underline{\func}^L_{\baseCat}(\underline{\D},\underline{\E})\vop)\vop\\
         &\simeq \underline{\func}^R_{\baseCat}(\underline{\sC}\vop, \underline{\lfunc}_{\baseCat}(\underline{\D},\underline{\E})\vop)\vop\\
         &\simeq \underline{\func}^R_{\baseCat}(\underline{\sC}\vop, \underline{\rfunc}_{\baseCat}(\underline{\E},\underline{\D}))\vop\\
         &\simeq \underline{\func}^R_{\baseCat}(\underline{\sC}\vop, \underline{\func}^{R,\text{acc}}_T(\underline{\E},\underline{\D}))\vop\\
         &\simeq \underline{\func}^{R,\text{acc}}_T(\underline{\E}, \underline{\func}^{R}_T(\underline{\sC}\vop, \underline{\D}))\vop\\
         &\simeq \underline{\rfunc}_{\baseCat}(\E, \underline{\func}^{R}_T(\underline{\sC}\vop, \underline{\D}))\vop\\
         &\simeq \underline{\lfunc}_{\baseCat}(\underline{\func}^{R}_T(\underline{\sC}\vop, \underline{\D}), \underline{\E})\\
         &\simeq \underline{\func}^L_{\baseCat}(\underline{\func}^{R}_T(\underline{\sC}\vop, \underline{\D}), \underline{\E})
     \end{split}
 \end{equation*}
where the second equivalence is by \cref{opOfFunctorCats}; the third, fifth, seventh, and ninth equivalence is by the adjoint functor \cref{parametrisedAdjointFunctorTheorem}; the fourth and eighth are from \cref{HTT5.2.6.2}. In the seventh and ninth equivalence, we have also used that $\underline{\func}^{R}_T(\underline{\sC}\vop, \underline{\D})$ is $\baseCat$--presentable, which is provided by \cref{presentableFunctorCategories}. Therefore, $\underline{\func}^{R}_T(\underline{\sC}\vop, \underline{\D})$ satisfies the universal property of $\underline{\sC}\otimes \underline{\D}$.
\end{proof}

\printbibliography
\end{document}